\documentclass[letterpaper, 10pt,reqno]{amsart}

\usepackage{amssymb}
\usepackage{bbm}
\usepackage{url}
\usepackage{a4wide}
\usepackage{tikz}
\usetikzlibrary{shapes.geometric}
\usepackage[textsize=small]{todonotes}
\usepackage{enumitem}

\usepackage{mathtools}
\mathtoolsset{showonlyrefs}

\definecolor{MyDarkblue}{rgb}{0,0.08,0.50}
\definecolor{Brickred}{rgb}{0.65,0.08,0}

\usepackage{orcidlink}

\usepackage{hyperref}
\hypersetup{
colorlinks=true,       
linkcolor=MyDarkblue,          
citecolor=Brickred,        
filecolor=red,      
urlcolor=cyan           
}

\newtheorem*{theorem*}{Theorem}
\newtheorem{theorem}{Theorem}[section]
\newtheorem{lemma}[theorem]{Lemma}
\newtheorem{proposition}[theorem]{Proposition}
\newtheorem{corollary}[theorem]{Corollary}
\newtheorem{problem}[theorem]{Open Problem}

\theoremstyle{definition}
\newtheorem{definition}[theorem]{Definition}

\newtheorem{remark}[theorem]{Remark}
\newtheorem{example}[theorem]{Example}

\newenvironment{assx}{
\textbf{Assumption $\mathbf{x}$. }}{}

\renewcommand{\P}{\mathbb{P}}
\newcommand{\prob}{\mathbb{P}}
\newcommand{\Pv}{\mathbb{P}}

\newcommand{\CC}{\mathcal{C}}

\newcommand{\eps}{\varepsilon}

\newcommand{\cA}{\mathcal{A}}\newcommand{\cB}{\mathcal{B}}\newcommand{\cC}{\mathcal{C}}
\newcommand{\cD}{\mathcal{D}}\newcommand{\cE}{\mathcal{E}}\newcommand{\cF}{\mathcal{F}}
\newcommand{\cG}{\mathcal{G}}\newcommand{\cH}{\mathcal{H}}\newcommand{\cI}{\mathcal{I}}
\newcommand{\cJ}{\mathcal{J}}
\newcommand{\cO}{\mathcal{O}}
\newcommand{\cP}{\mathcal{P}}\newcommand{\cR}{\mathcal{R}}
\newcommand{\cS}{\mathcal{S}}
\newcommand{\cW}{\mathcal{W}}

\newcommand{\Var}{{\rm Var}}

\newcommand{\e}{{\mathrm e}}

\newcommand{\R}{\mathbb{R}}
\newcommand{\N}{\mathbb{N}}
\newcommand{\Z}{\mathbb{Z}}

\newcommand{\dd}{\mathrm{d}}

\renewcommand{\emptyset}{\varnothing}

\newcommand{\CI}{\mathcal {I}}
\newcommand{\CJ}{\mathcal {J}}

\newcommand*{\wt}{\widetilde}

\newcommand*{\be}{\begin{equation}}
\newcommand*{\ee}{\end{equation}}
\newcommand*{\ba}{\begin{aligned}}
\newcommand*{\ea}{\end{aligned}}
\newcommand*{\barr}{\begin{array}{c}}
\newcommand*{\earr}{\end{array}}
\def \toinp    {\buildrel {\Pv}\over{\longrightarrow}}
\def \toindis  {\buildrel {\dd}\over{\longrightarrow}}
\def \toas     {\buildrel {\mathrm{a.s.}}\over{\longrightarrow}}

\newcommand*{\ind}{\mathbbm{1}}
\makeatletter
\def\namedlabel#1#2{\begingroup
#2%
\def\@currentlabel{#2}%
\phantomsection\label{#1}\endgroup
}

\makeatother

\newcommand{\bes}{\begin{equation*}}
\newcommand{\ees}{\end{equation*}}

\renewcommand{\P}[1]{\mathbb{P}\!\left(#1\right)}
\newcommand{\E}[1]{\mathbb{E}\left[#1\right]}

\renewcommand{\N}{\mathbb{N}}

\numberwithin{equation}{section}

\renewcommand{\e}{\mathrm{e}}

\newcommand{\floor}[1]{\lfloor #1\rfloor}
\newcommand{\ceil}[1]{\lceil #1\rceil}
\newcommand{\x}{\mathbf{x}}

\newcommand{\invisible}[1]{}
\newcommand{\ensymboldremark}{\hfill$\blacktriangleleft$}

\renewcommand{\aa}{\mathrm{a}}
\newcommand{\ff}{\mathrm{f}}

\newcommand{\xref}{\hyperref[ass:A]{\x}}

\title[High-degree vertices in uniform recursive DAGs with freezing]{High-degree vertices in uniform recursive directed acyclic graphs with freezing}
\author{Rafael Engel}
\address{Universität Augsburg, Department of Mathematics, D-86135 Augsburg, Germany}
\email{rafael.engel@uni-a.de}
\author{Bas Lodewijks\orcidlink{0000-0001-5624-2410}}
\address{University of Sheffield, School of Mathematical and Physical Sciences, Sheffield, England}
\email{bas.lodewijks@sheffield.ac.uk}
\date{\today} 

\begin{document}
\maketitle 
\begin{abstract}
We study uniform recursive directed acyclic graphs with freezing. Here, a graph is built by adding vertices one-by-one and connecting a new vertex to $m\in\N$ uniformly selected vertices already present. At certain steps vertices can also be frozen, and arriving vertices are not allowed to connect to frozen vertices. This model generalises the uniform attachment tree with freezing, introduced by Bellin et.\ al (which corresponds to the case $m=1$) as well as the uniform recursive directed acyclic graph model (where no vertices are frozen). Under mild assumptions on when vertices are frozen, we study the empirical degree distribution, large degrees in the graph, and other properties of large-degree vertices such as their label and distance to the first vertex in the graph. Our work improves and/or extends various results from the literature on uniform attachment trees (with freezing) and uniform recursive directed acyclic graphs without freezing. In particular, our results show that statistics that are determined `locally' (e.g.\ the empirical degree distribution and maximum degree) are essentially unaffected by the freezing of vertices, whereas statistics that are determined `globally' (e.g.\ the length of paths between vertices) are highly affected by introducing freezing. The analysis relies on adapting the Kingman coalescent construction for uniform attachment trees to the non-tree setting.  
\end{abstract}

\section{Introduction}

Random graphs have gained significant attention as models for real-world networks in recent decades, see e.g.\ the work of Van der Hofstad~\cite{Hofstad.2017,Hof24} and the references therein for a good introduction to the field. Among the many models that exist, evolving random graphs, where a sequence of graphs is constructed recursively by adding vertices and edges sequentially, model the temporal evolution of real-world networks. One such model is the uniform recursive directed acyclic graph (URD), first introduced by Devroye and Lu~\cite{Devroye.Lu.1995}. Given a parameter $m\in\N$, one constructs a sequence of graphs by  starting with a single vertex labelled $1$, and adding vertices with labels $2,3,\ldots$ one-by-one. Each vertex $i>1$ that is added independently selects $\min\{i,m\}$ many distinct vertices that are already present  uniformly at random and connects to each of them by a directed edge. The case $m=1$ yields the uniform attachment tree, also known as the random recursive tree (RRT). 

An extension to the URD model is to incorporate freezing. That is, at certain steps, rather than adding a new vertex one freezes an existing vertex chosen uniformly at random among all non-frozen vertices. New vertices are then allowed to connect to non-frozen (called \emph{active}) vertices only. In the context of evolving real-world networks, the freezing dynamics are natural. For example, people in a social network pass away and are afterwards not able to make new social connections, scientists in a collaboration network retire and stop forming new collaborations, and papers in a citation network may lose relevance and stop being cited.

The uniform attachment tree model with freezing (i.e.\ the case $m=1$) has recently received attention~\cite{Bellin.2023,Bellin.2024,Brandenberger.2025}, though a related model studied by Deijfen~\cite{Dei10} that allows for (to some extent) more general attachment and freezing rules was introduced earlier as well. A model similar to URDs with freezing was studied by D\'iaz, Lichev, and the second author~\cite{DiazLichLod22}, where each new vertex connects to a random number of vertices chosen uniformly at random and frozen vertices are removed from the graph, together with their incident edges. 

The purpose of this article is to extend the study of uniform attachment trees with freezing  model to  the graph setting. The main tool we use is the Kingman coalescent construction of uniform attachment trees with freezing, which can be viewed as a time-reversed construction of the tree that provides several analytical advantages over other approaches. This construction has been used first in the setting without freezing~\cite{Addario.Eslava.2018,Eslava.2017,Eslava.2021,Lodewijks.2023} and has recently been adapted to the setting with freezing~\cite{Bellin.2023,Bellin.2024,Brandenberger.2025}. We further extend it to the \emph{non-tree} setting of URDs with freezing and leverage this methodology to analyse the degree distribution, fine distributional properties of the maximum degree, and other properties of high-degree vertices. 

Our results provide  extensions of results on the empirical degree distribution and high degrees of uniform attachment trees by Adarrio-Berry and Eslava~\cite{Addario.Eslava.2018}. These results are  more refined compared to known results on the maximum degree in URDs of Devroye and Lu~\cite{Devroye.Lu.1995}.  Furthermore, we  prove central limit theorems  for the labels of uniform high-degree vertices and the length of greedy longest paths of such uniform high-degree vertices to the `root' vertex labelled $1$. These extend results by Devroye and Janson~\cite{Devroye.Janson.2011} on greedy long paths in URDs as well as extend work of the second author~\cite{Lodewijks.2023} and Eslava~\cite{Eslava.2021} on properties of high-degree vertices in uniform attachment trees. Interestingly, our results show that freezing essentially does not change the behaviour of the empirical degree distribution, the distributional properties of the maximum degree, and the number vertices with (near-)maximal degree, whereas the behaviour of other properties of high-degree vertices, such as the length of long paths and their labels, is highly affected by freezing vertices.

\section{Model definition and results}

In this section, we formally introduce the URD with freezing model and state the main results.

Let $\mathbf x=(x_i)_{i\in\N}\in\{-1,1\}^\N$ be the \emph{choice sequence} and define
\be \label{eq:C}
A_n=A_n(\mathbf x) \coloneq \sum_{i=1}^nx_i\qquad\text{for }n\in\N.
\ee
We recursively construct a random directly acyclic graph, where the choice sequence $\mathbf x$ encodes at which steps we add and freeze vertices. More formally: 

\begin{definition}[URD with freezing]
\label{def:dag}
Fix $m \in \N$ and a choice sequence $\mathbf x\in \{-1,1\}^\N$. We recursively construct a sequence $(G^{(n)})_{n\in\N}$ of directed graphs. We initialise $G^{(1)}$ as the graph consisting of one \emph{active} vertex with label $(1,a)$ if $x_1=1$ and we initialise $G^{(1)}\coloneq\emptyset$ otherwise. Conditionally on $G^{(i-1)}$ for some $i\geq 2$, we construct $G^{(i)}$ from $G^{(i-1)}$ in the following manner. If $A_{i-1}>0$, then:
\begin{itemize}
\item If $x_i=-1$, choose an active vertex $(v,a)$ in $G^{(i-1)}$ uniformly at random.  \emph{Freeze} this vertex by changing its label to $(v,i)$.
\item If $x_i=1$, add an active vertex labelled $(i,a)$ to $G^{(i-1)}$. Select  $m\land A_{i-1}$ many distinct active vertices $(k_1,a),\ldots,( k_{m \land A_{i-1}},a)$ in $G^{(i-1)}$ uniformly at random, and connect $(i,a)$ to $(k_j,a)$ by a directed edge for each $j\in[m\wedge A_{i-1}]$.
\end{itemize} 
Else, if $A_{i-1}\leq0$,  set $G^{(j)} \coloneq G^{(i-1)}$ for all $j\geq i-1$ and terminate the process.
\end{definition}

\begin{remark} 
The second element of the label of vertices indicates whether a vertex is active (the letter $a$) or frozen (an integer, say $i\in\N$). For frozen vertices the second element additionally provides the step at which said vertex was frozen. \ensymboldremark 
\end{remark} 

Note that choosing $m=1$ in Definition~\ref{def:dag} results in a uniform attachment tree with freezing as introduced by Bellin et.\ al in \cite{Bellin.2023}. Choosing $\mathbf x=(1,1,\ldots)$ yields a uniform recursive DAG, as introduced by Devroye and Lu in~\cite{Devroye.Lu.1995} (for which $m=1$ is a further special case that yields the random recursive tree or uniform attachment tree model).

Before we state our results, we introduce some further notation and several (minor) assumptions on the choice sequence $\mathbf x$. For ease of writing, let us set 
\be \label{eq:theta}
\theta=\theta(m)\coloneq \frac{m+1}{m},
\ee 
which is a model parameter that governs the exponential decay of the limiting degree distribution. Given a choice sequence $\mathbf x$, we introduce 
\be \label{eq:taux}
\tau(\mathbf x)\coloneq \inf\{i\in\N\colon A_i(\mathbf x)\leq 0\}, 
\ee 
with the convention that the infimum equals $\infty$ when the set on the right-hand side is empty. By its definition, all vertices in the graph $G^{(\tau(\mathbf{x}))}$ are frozen and the process is terminated at step $\tau(\mathbf x)$ if $\tau(\mathbf x)<\infty$, and the process never terminates otherwise. For $n\in\N$, we define
\begin{equation}
\begin{aligned}
\mathbb A_n &= \mathbb A_n(\mathbf x) \coloneq \{v\colon (v,a)\in G^{(n)}\}, 
&\quad
\mathbb F_n &= \mathbb F_n(\mathbf x) \coloneq \{v\colon (v,i)\in G^{(n)}\text{ with } v,i\in[n]\},\\
\cA_n &= \cA_n(\mathbf x) \coloneq \{i\in[n]\colon x_i=1\}, 
&\quad
\cF_n &= \cF_n(\mathbf x) \coloneq \{i\in[n]\colon x_i=-1\}.
\end{aligned}
\end{equation}
$\mathbb A_n$ and $\mathbb F_n$ are the sets of active and frozen vertices in $G^{(n)}$, respectively. If $A_n>0$, note that $A_n = |\mathbb A_n| = 2|\cA_n| -n$. We also set $F_n \coloneq |\mathbb F_n| = |\cF_n| = (n-A_n)/2$ as the number of frozen vertices in $G^{(n)}$.
Furthermore, we define 
\begin{equation}
\label{eq:hn+}
h_n^+\coloneq \sum_{i=1}^n \ind_{\{x_i=1\}}\frac{1}{A_i}. 
\end{equation}
We interpret $mh_n^+$ as (an approximation of) the expected degree of vertex $1$ in $G^{(n)}$. If $\tau(\mathbf x)>n$, we have by~\cite[Lemma $16$]{Bellin.2023} the lower bound
\begin{equation}
\label{eq:hn+lb}
\qquad h_n^+\geq \log(n/2).
\end{equation}
We also allow the choice sequence to depend on $n$. That is, we construct $G^{(1)},\ldots, G^{(n)}$ using a choice sequence $\mathbf x^{(n)}$ (such that $\tau(\mathbf x^{(n)})>n-1$). 

Finally, we introduce the following assumptions on $(\mathbf x^{(n)})_{n\in\N}$. 

\begin{assx}
\label{ass:A}
The choice sequences $(\mathbf x^{(n)})_{n\in\N}$ satisfy the following statements.
\begin{enumerate}
\item \label{item:geq1} $\tau(\mathbf x^{(n)})>n-1$ for all $n\in\N$. 
\end{enumerate}
Given $\delta\in(0,1/2)$ and $\eps,\eta\in(0,1)$, there exist $I_{\xref}=I_{\xref}(n)\in\N$ and $N\in\N$ such that $I_{\xref}(n)\to\infty$ with $n$ and for all $n\geq N$, we have
\begin{enumerate}[resume]
\item\label{item:hI}  $h_{I_{\xref}}^+\leq \eps h_n^+$.
\item \label{item:lb}  $A_i(\mathbf x^{(n)})\geq i^{1/2+\delta}$ for all $I_{\xref}\leq i\leq n$.
\item \label{item:Jn} $ I_{\xref}\leq F_n^\eta$.
\end{enumerate}
\end{assx}

\begin{remark}
Note that, once Part~\ref{item:geq1} holds, Part~\ref{item:hI} is automatically satisfied for sublogarithmic sequences $I_{\xref}$ by the lower bound in~\eqref{eq:hn+lb}. In some of our results, however, we may require $I_{\xref}$ to grow faster in $n$,  for which Part~\ref{item:hI} is no longer trivially satisfied. Additionally, taking a smaller $I_{\xref}$ may not satisfy Part~\ref{item:lb}. \ensymboldremark
\end{remark}
Part~\ref{item:geq1} ensures that the recursive construction of $G^{(n)}$ does not terminate until step $n$. Since $m h_n^+$ can be interpreted as (an approximation of) the expected degree of the root (or, similarly, of a fixed vertex), Part~\ref{item:hI} implies that the (expected) degree of early vertices (added before step $I_{\xref}$) is of the same order as that of the root. Part~\ref{item:lb} ensures there are sufficiently many active vertices present at all late steps, so that correlations between statistics of distinct vertices can be controlled. Finally, Part~\ref{item:Jn} ensures that the number of frozen vertices grows sufficiently fast. Again, this is to control correlations between frozen vertices, and it is natural if one wants to say something about the behaviour of a `typical' frozen vertex.
Part~\ref{item:Jn} can be omitted if we are interested in statistics of \emph{active} vertices only. We provide several classes of choice sequences for which these assumptions are met in Section~\ref{sec:assdiscuss}. In what follows we suppress the superscript $(n)$ of $\mathbf x^{(n)}$ to ease notation, though all the results do hold for $n$-dependent choice sequences that satisfy Assumption~$\xref$.

\subsection{Statement of the main results}
\label{subsec:statement_results}

We split the presentation of our results into three parts. The first part is concerned with typical degrees of active and frozen vertices, the second with  (near)-maximal degrees of active and frozen vertices, and the third part deals with further properties of high-degree active vertices. 

\subsubsection{Empirical degree distribution.}\label{sec:degresults} Let $\deg_n(v)$ denote the \emph{in-degree} of $v$ in $G^{(n)}$ (here, connections with both active and frozen vertices count towards the degree of a vertex). Our first result concerns the distribution of the degree of typical vertices in the URD model with freezing.

\begin{theorem}
\label{theorem:joint_degree_distribution}
Fix $m\in\N$, $c\in(0,m+1)$ and $k,\ell\in\N_0$ such that $k+\ell\geq 1$. Fix a choice sequence $\mathbf x$ such that Assumption~$\xref$ is satisfied  for some $\eps<1-c/(m+1)$ in  Part~\ref{item:hI} and some $\eta<1-c/(m+1)$ in Part~\ref{item:Jn}. Let $V_1,\ldots,V_k\in\mathbb A_n$ be distinct active vertices and $W_1,\ldots,W_\ell\in\mathbb F_n$ be distinct frozen vertices selected uniformly at random from $G^{(n)}$. Then, there exists $\alpha\in(0,1)$ such that uniformly over all natural numbers $d_{a_1},\ldots,d_{a_k} < ch_n^+$ and $d_{f_1},\ldots,d_{f_\ell} < c\log F_n$,
\begin{equation}
\begin{aligned}
\P{\deg_n(V_v)\geq d_{a_v},\deg_n(W_w)\geq d_{f_w}\text{ for all }v\in[k],w\in[\ell]}= \theta^{-\sum_{v=1}^k d_{a_v}-\sum_{w=1}^\ell d_{f_w}}(1+o(I_{\xref}^{-\alpha})).
\end{aligned}
\end{equation}
\end{theorem}

Theorem~\ref{theorem:joint_degree_distribution} generalises a result of Addario-Berry and Eslava~\cite{Addario.Eslava.2018} for the random recursive tree (i.e.\  the case $m=1$ and $\mathbf x=(1,1,\ldots)$). When the   $d_{a_v}$ and $d_{f_w}$ for $v\in[k]$ and $w\in[\ell]$ are not `too large', Theorem~\ref{theorem:joint_degree_distribution} can be extended to a weak law of large numbers. 

\begin{corollary}\label{cor:LLN}
Fix $m\in\N$,  a choice sequence $\mathbf x$ such that Assumption~$\xref$ is satisfied for some $\eps\in(0,1/((m+1)\log \theta))$ in  Part~\ref{item:hI} and some $\eta\in(0,1/((m+1)\log\theta))$ in Part~\ref{item:Jn}. Define for $i\in\N_0$,        \begin{align}
N^{(n)}_i &\coloneq  |\{v \in \mathbb A_n \colon \deg_n(v) =  i\}|, &&\,N^{(n)}_{\geq i} \coloneq  |\{v \in \mathbb A_n \colon \deg_n(v) \geq   i\}|,
\intertext{and}
M^{(n)}_i &\coloneq  |\{v \in \mathbb F_n \colon \deg_n(v) =    i\}|, 
&&M^{(n)}_{\geq i} \coloneq  |\{v \in \mathbb F_n \colon \deg_n(v) \geq    i\}|.
\end{align} 
Then, for any $i=i(n),j=j(n)\in\N_0$ such that $\log_\theta(A_n)-i(n)\to\infty$ and $\log_\theta(F_n)-j(n)\to\infty$ with $n$, 
we have 
\be 
\frac{ N^{(n)}_i}{(1-\theta^{-1})\theta^{-i}A_n}\toinp 1,\quad \frac{N^{(n)}_{\geq i}}{\theta^{-i}A_n}\toinp1,\quad \frac{M^{(n)}_j}{(1-\theta^{-1})\theta^{-j}F_n}\toinp 1,\quad \frac{M^{(n)}_{\geq j}}{\theta^{-j}F_n}\toinp 1.
\ee 
\end{corollary}

\begin{remark}\label{rem:degdistr}
$(i)$ For both Theorem~\ref{theorem:joint_degree_distribution} and Corollary~\ref{cor:LLN}, when one is only interested in the degrees of typical \emph{active} vertices (i.e.\ $\ell=0$ in Theorem~\ref{theorem:joint_degree_distribution}), Assumption~$\xref$\ref{item:Jn} is not required.

$(ii)$ For Theorem~\ref{theorem:joint_degree_distribution}, when one is only interested in the degree of a \emph{single} vertex, active or frozen (i.e.\ $k+\ell=1$), then Assumptions~$\xref$\ref{item:lb} and~$\xref$\ref{item:Jn} can be weakened to $\lim_{n\to\infty} A_n=\infty$ and  $\lim_{n\to\infty}F_n=\infty$, respectively.

$(iii)$ The condition on $i(n)$ and $j(n)$ in Corollary~\ref{cor:LLN} ensures that $\theta^{-i}A_n$ and $\theta^{-j}F_n$ tend to infinity with $n$. \ensymboldremark
\end{remark}

A strong law of large numbers for $N_i^{(n)}$, $N_{\geq i}^{(n)}$, $M_i^{(n)}$, and $M_{\geq i}^{(n)}$ with $i,j\in\N_0$ fixed has been obtained in the following two cases: $(1)$ Uniform recursive directed acyclic graphs without freezing (i.e. $m\in\N$ and $\mathbf x=(1,1,\ldots)$) by the second author and Ortgiese~\cite[Theorems $2.5$]{Lodewijks.Ortgiese.2020}, and $(2)$ Uniform attachment with i.i.d.\ freezing (i.e.\ he case $m=1$,  $\mathbf x$ a sequence of independent biased Rademacher random variables that equal one with probability $p\in(1/2,1]$, conditionally on $A_n(\mathbf x)>0$ for all $n\in\N$) in recent work of Desmarais~\cite{Des26}. Desmarais also proves joint asymptotic normality for $N_i^{(n)},N_{\geq i}^{(n)},M_j^{(n)}$, and $M_{\geq j}^{(n)}$ for fixed $i,j\in\N_0$. Here, we provide a more general but slightly weaker result that extends to any $m\in\N$, degrees that depend on $n$, and  more general $n$-dependent choice sequences $\mathbf x$. We believe that a strong law of large numbers should hold at least for fixed $i,j\in\N_0$ in our setting as well, but the error rate $I_{\xref}^{-\alpha}$ in Theorem~\ref{theorem:joint_degree_distribution} is not summable in $n$ and hence the result is not sufficiently strong to obtain a strong law of large numbers.

\subsubsection{(Near-)maximal  degrees} Corollary~\ref{cor:LLN} heuristically implies that the largest degree among active vertices in $G^{(n)}$ is of the order $\log_{\theta}A_n $, as $N^{(n)}_{\geq \lfloor \log_\theta(A_n)\rfloor}\approx 1$. By a similar reasoning, the largest degree among frozen vertices in $G^{(n)}$ is of the order $\log_\theta F_n$. This agrees with work of Devroye and Lu~\cite[Theorem 2]{Devroye.Lu.1995} on the largest degree in random recursive directed acyclic graphs, i.e.\ for the choice sequence $\mathbf x=(1,1,\ldots)$. Theorem~\ref{theorem:joint_degree_distribution} allows us to provide several stronger results regarding the maximal degree and `near-maximal' degrees that also hold for more general choice sequences. These are inspired by results from Addario-Berry and Eslava~\cite{Addario.Eslava.2018} for the random recursive tree.

For $i\in\Z,n\in\N$, we define the random variables 
\begin{equation}\label{eq:Xdef}
\begin{aligned}
X^{(n)}_i &\coloneq N^{(n)}_{\lfloor \log_\theta A_n\rfloor+i}= |\{v \in \mathbb A_n \colon \deg_n(v) = \lfloor \log_\theta A_n \rfloor + i\}|,\\
X^{(n)}_{\geq i} &\coloneq N^{(n)}_{\geq \lfloor \log_\theta A_n\rfloor+i}= |\{v \in \mathbb A_n \colon \deg_n(v) \geq \lfloor \log_\theta A_n \rfloor + i\}|,
\end{aligned}
\end{equation}
and, similarly,
\begin{equation}\label{eq:Ydef}
\begin{aligned}
Y^{(n)}_i &\coloneq  M^{(n)}_{\lfloor \log_\theta F_n\rfloor+i}= |\{v \in \mathbb F_n \colon \deg_n(v) = \lfloor \log_\theta F_n \rfloor + i\}|,\\
Y^{(n)}_{\geq i} &\coloneq  M^{(n)}_{\geq \lfloor \log_\theta F_n\rfloor+i}=|\{v \in \mathbb F_n \colon \deg_n(v) \geq \lfloor \log_\theta F_n \rfloor + i\}|.
\end{aligned}
\end{equation}    

Let $\Z^*\coloneq\Z\cup\{\infty\}$. We endow $\Z^*$ with the metric $d$ defined by
\begin{equation}
\label{eq:metric_d}
d(i,j)=|2^{-j}-2^{-i}|,\qquad\text{and}\qquad d(i,\infty)=2^{-i},
\end{equation}
for $i,j\in\Z$, so that $[i,\infty]$ is a compact set for any $i\in\Z$. Define $\mathcal{M}^\#_{\Z^*}$ as the set of boundedly finite measures of $\Z^*$ (i.e.\ locally finite measures on $\Z^*$). Let $\mathcal{P}$ be an inhomogeneous Poisson point process on $\R$ with rate function
\begin{equation}
\lambda(x)\coloneq\theta^{-x}\log\left(\theta\right)\qquad \text{for }x\in\R.
\end{equation}
For each $\varepsilon\in[0,1]$, let $\mathcal{P}^\varepsilon$ be the point process on $\Z^*$ given by
\begin{equation}
\mathcal{P}^\varepsilon \coloneq \sum_{x\in\mathcal{P}}\delta_{\floor{x+\varepsilon}},
\end{equation}
where $\delta$ is a Dirac measure. Similarly, for all $n\in\N$, let
\begin{equation}
\mathcal{P}_\mathrm{a}^{(n)} \coloneq \sum_{v\in\mathbb A_n}\delta_{\deg_n(v)-\floor{\log_\theta A_n}}\qquad\text{and} \qquad \cP^{(n)}_\mathrm{f}\coloneq \sum_{v\in \mathbb F_n}\delta_{\deg_n(v)-\floor{\log_\theta F_n}}.
\end{equation}
Then, for each $i\in\Z$, we have
\begin{equation}
\mathcal{P}^\varepsilon\{i\} \coloneq \mathcal{P}^\varepsilon(\{i\}) = |\{x\in\mathcal{P}\colon \floor{x+\varepsilon}=i\}|=|\{x\in\mathcal{P}\colon x\in[i-\varepsilon,i+1-\varepsilon)\}|
\end{equation}
and $\mathcal{P}^{(n)}\{i\} \coloneq \mathcal{P}^{(n)}(\{i\})=X_i^{(n)}$. By the definition of $\cP$, it is clear that for $i\in\Z$ we have $\mathcal{P}^\varepsilon\{i\}\sim\text{Poi}\left((1-\theta^{-1})\theta^{-i+\varepsilon}\right)$.

Finally, we define 
\be \label{eq:epsn}
\eps_n^\mathrm{a}\coloneq \log_\theta A_n-\lfloor\log_\theta A_n\rfloor\qquad\text{and}\qquad \eps_n^\mathrm{f}\coloneq \log_\theta F_n-\floor{\log_\theta F_n}.
\ee 
The following results shows joint weak convergence of the point processes $\cP_\aa^{(n)}$ and $\cP^{(n)}_\ff$ along subsequences $(n_\ell)$ such that $\eps_{n_\ell}^\aa$ and $\eps_{n_\ell}^\ff$ converge, which shows that the number of vertices with degrees of the order of the maximum degree (of active and frozen vertices) is asymptotically Poisson. 

\begin{theorem}
\label{theorem:poisson_convergence}
Fix $m\in\N$ and a choice sequence $\mathbf x$ such that Assumption~$\xref$ is satisfied for some $\eps<1-1/((m+1)\log\theta)$ in Part~\ref{item:hI} and some $\eta<1-1/((m+1)\log\theta)$ in Part~\ref{item:Jn}. Fix $\varepsilon^\mathrm{a},\eps^\ff\in[0,1]$ and let $\cP_\mathrm{a}\overset \dd=\cP^{\eps^\mathrm{a}}$ and $\cP_\ff\overset \dd= \cP^{\eps^\ff}$ be independent point processes. Let $(n_\ell)_{\ell\in\N}$ be a subsequence of integers such that $n_\ell\to\infty$ and $\varepsilon_{n_\ell}^\square\to\varepsilon^\square$ as $\ell\to\infty$ for $\square\in\{\aa,\ff\}$. Then, $\mathcal{P}_\mathrm{a}^{(n_\ell)}$ and $\cP_\ff^{(n_\ell)}$ jointly converge weakly to $\mathcal{P}_\mathrm{a}$ and $\cP_\ff$ as $\ell\to\infty$  in $\mathcal{M}^\#_{\Z^*}$, respectively. Equivalently, for any $i<i^\prime\in\Z$ and $j<j'\in\Z$, jointly as $\ell\to\infty$,
\be\ba\label{eq:poisson_convergence_by_continuity_sets_recovered_from_FDDs}
({}&X_i^{(n_\ell)},\ldots,X_{i^\prime-1}^{(n_\ell)},X_{\geq i^\prime}^{(n_\ell)},Y_j^{(n_\ell)},\ldots,Y_{j^\prime-1}^{(n_\ell)},Y_{\geq j^\prime}^{(n_\ell)})\\ 
&\toindis(\mathcal{P}_{\mathrm{a}}\{i\},\ldots,\mathcal{P}_{\mathrm{a}}\{i^\prime -1\},\mathcal{P}_\mathrm{a}[i^\prime,\infty),\mathcal{P}_{\mathrm{f}}\{j\},\ldots,\mathcal{P}_{\mathrm{f}}\{j^\prime -1\},\mathcal{P}_{\mathrm{f}}[j^\prime,\infty)).
\ea\ee
\end{theorem}

\begin{remark}
If one is interested only in the convergence of $\cP_\aa^{(n_\ell)}$, then Assumption~$\xref$\ref{item:Jn} is not required.\ensymboldremark 
\end{remark} 

For the proof of Theorem~\ref{theorem:poisson_convergence}, it suffices to show~\eqref{eq:poisson_convergence_by_continuity_sets_recovered_from_FDDs}, since weak convergence of $\mathcal{P}_\square^{(n_\ell)}$ follows from the convergence of its finite-dimensional distributions (see e.g.\ page 143 in \cite{Daley.2008}), which can be represented using the distribution of $(X_i^{(n_\ell)},\ldots,X_{i^\prime-1}^{(n_\ell)},X_{\geq i^\prime}^{(n_\ell)})$ and $(Y_j^{(n_\ell)},\ldots,Y_{j^\prime-1}^{(n_\ell)},Y_{\geq j^\prime}^{(n_\ell)})$ for $i<i^\prime, j<j'\in\Z$ fixed. 

We define the largest degree among all active and frozen vertices in $G^{(n)}$ as 
\be 
\Delta_n^\mathrm{a}\coloneq  \max\{ \deg_n(v)\colon v\in \mathbb A_n\}\qquad\text{and}\qquad \Delta_n^\mathrm{f}\coloneq \max\{ \deg_n(v)\colon v\in \mathbb F_n\},
\ee 
respectively. The following result provides the asymptotic distribution of these maximum degrees. 
\begin{theorem}
\label{theorem:maximum_degree}
Fix $m\in\N$, $c\in(0,m+1)$, and a choice sequence $\mathbf x$ such that Assumption~$\xref$ is satisfied for some $\eps<1-c/(m+1)$ in Part~\ref{item:hI} and some $\eta<1-c/(m+1)$ in Part~\ref{item:Jn}. Let $(i_n^\mathrm{a})_{n\in\N}$ and $(i_n^\mathrm{f})_{n\in\N}$ be sequences such that $i_n^\mathrm{a}+\log_\theta A_n >0$ and $i_n^\mathrm{f}+\log_\theta F_n >0$ for all $n$, and 
\be 
\limsup_{n\to\infty} \frac{i_n^\mathrm{a}+\log_\theta A_n}{h_n^+}<c\qquad\text{and}\qquad \limsup_{n\to\infty}\frac{i_n^\mathrm{f}+\log_\theta F_n}{\log F_n}<c.
\ee 
Then,
\begin{equation}
\prob(\Delta_n^\mathrm{a}\geq \floor{\log_\theta A_n}+i_n^\mathrm{a}, \Delta_n^\mathrm{f}\geq \floor{\log_\theta F_n}+i_n^\mathrm{f}) =(1+o(1))\prod_{\square\in\{\mathrm a,\mathrm f\}}\!\!\! \big(1-\exp\big(-\theta^{-i_n^\square+\eps_n^\square}\big)\big).
\end{equation}
\end{theorem}

\begin{remark}
$(i)$ Convergence in distribution of $\Delta_n^\mathrm{a} -\lfloor \log_\theta A_n\rfloor$ and $\Delta_n^\mathrm{f} -\floor{\log_\theta F_n}$ does not hold due to a lattice effect, caused by the floor function applied to $\floor{\log_\theta A_n}$ and $\floor{\log_\theta F_n}$.  

$(ii)$ If one is interested in the marginal distribution of $\Delta_n^\mathrm{a}$, Assumption~$\xref$\ref{item:Jn} is not required.\ensymboldremark
\end{remark}

The final result of this subsection establishes joint asymptotic normality for $X_i^{(n)}$ and $Y_i^{(n)}$ when $i$ tends to $-\infty$ with respect to $n\in\N$ at a slow enough rate, capturing the asymptotic distribution of the number of active and frozen vertices with `near-maximal' degree.
\begin{theorem}
\label{theorem:asymptotic_normality_smaller_degrees}
Fix $m\in\N$ and a choice sequence $\mathbf x$ such that Assumption~$\xref$ is satisfied for  $I_{\xref}=\lfloor n^{\eps}\rfloor$  and $0<\eps<\eta<1-1/((m+1)\log\theta)$ in Parts~\ref{item:hI} and~\ref{item:Jn}. Let $(i_n^\aa)_{n\in\N}$ and $(i_n^\ff)_{n\in\N}$ be integer-valued sequences such that $i_n^\square\to-\infty$ and $i_n^\square=o(\log n)$ for $\square\in\{\aa,\ff\}$, and let $Z_1$ and $Z_2$ be two i.i.d.\ standard normal random variables. Then,
\begin{equation}
\Big(\frac{X_{i_n^\aa}^{(n)}-(1-\theta^{-1})\theta^{-(\lfloor \log_\theta A_n\rfloor+i_n^\aa)}A_n}{\sqrt{(1-\theta^{-1})\theta^{-(\lfloor \log_\theta A_n\rfloor+i_n^\aa)}A_n}},\frac{Y_{i_n^\ff}^{(n)}-(1-\theta^{-1})\theta^{-(\lfloor \log_\theta F_n\rfloor+i_n^\ff)}F_n}{\sqrt{(1-\theta^{-1})\theta^{-(\lfloor \log_\theta F_n\rfloor+i_n^\ff)}F_n}}\Big)\toindis(Z_1,Z_2).
\end{equation}
\end{theorem}

\begin{remark}
$(i)$ The additional assumption on the growth rate   $I_{\xref}$ ensures a polynomial error rate in Theorem~\ref{theorem:joint_degree_distribution}, which is necessary to deal with error terms in the normal approximation of the order $\theta^{-i_n^\square}$, which grow subpolynomially in $n$ when $i_n^\square =o(\log n)$.

$(ii)$ For a central limit theorem for $X^{(n)}_{i^\aa_n}$ only, Assumption~$\xref$\ref{item:Jn} is not required. \ensymboldremark
\end{remark}
The convergence in Theorem~\ref{theorem:asymptotic_normality_smaller_degrees} is proved using a version of the method of moments for factorial moments. This theorem complements Corollary~\ref{cor:LLN} and Theorem~\ref{theorem:poisson_convergence}, as it yields a central limit theorem for the number of active and frozen vertices attaining a `near-maximum' degree,  (with respect to the maximum degree among all active/frozen vertices). In contrast, Theorem~\ref{theorem:poisson_convergence} shows that the number of active and frozen vertices with a degree of the order $\log_\theta A_n$ and $\log_\theta F_n$, respectively, is asymptotically Poisson, and Corollary~\ref{cor:LLN} provides a first-order result only.

Interestingly, we observe that the behaviour of the empirical degree distribution and that of (near-)maximal degrees does not change under the influence of freezing. Indeed, under the relatively mild conditions in Assumption~$\xref$, the results in Theorem~\ref{theorem:joint_degree_distribution} and Corollary~\ref{cor:LLN} for the empirical degree distribution stay almost unchanged compared to the setting without freezing, i.e.\ $\mathbf x=(1,1,\ldots)$, and Theorems~\ref{theorem:poisson_convergence}, \ref{theorem:maximum_degree}, and~\ref{theorem:asymptotic_normality_smaller_degrees} for (near-)maximal degrees are in similar spirit to those for the random recursive tree (i.e.\ $m=1$ and $\mathbf x=(1,1,\ldots)$). This is in contrast with results on the local behaviour near the root, the depth of typical vertices and the height of uniform attachment trees with freezing, as shown in~\cite{Bellin.2023,Bellin.2024} (and as we will see in Theorem~\ref{theorem:label_multiple_vertices_conditional_degrees}). We believe that this should carry through for other `local' properties as well, such as the local limit (i.e.\ the distribution of the neighbourhood of typical vertices), which we leave as an open problem here. 

\begin{problem}
Show that the weak local limit of $G^{(n)}$ is the same for any choice sequence $\mathbf x$ that satisfies Assumption~$\xref$ (and perhaps additional assumptions), so in particular it is the same as the local weak limit the the uniform recursive directed acyclic graph model without freezing.
\end{problem}

At the same time, results on e.g.\ the local weak limit on a related URD with freezing model where frozen vertices and their incident edges are removed from the graph~\cite[Theorem 1.3 and Corollary 1.4]{DiazLichLod22} (thus creating a random graph with multiple components) shows that taking into account the `state' of vertices (i.e.\ active or frozen) is susceptible to freezing. We also leave an open problem in direction here, as to how freezing influences the behaviour of such statistics. 

\begin{problem}
Let the active (resp.\ frozen) in-degree denote the number of incoming edges to $v$ from active (resp.\ frozen) vertices. Determine scaling limits for the number of vertices with a given active and/or frozen in-degree, similar to Corollary~\ref{cor:LLN}.
\end{problem}

\subsubsection{The label and `depth' of high-degree active vertices.} Finally, we turn to further properties of high-degree active vertices. Given a uniform active vertex in $G^{(n)}$ such that its degree is at least $d=d(n)$, say, we are interested in the order of the \emph{label} of the vertex, and on its `\emph{distance}' to vertex $1$. That is, \emph{when} and \emph{where} was this now high-degree vertex initially introduced into the graph. Here, `distance' does not refer to the shortest or longest path between the high-degree vertex and vertex $1$, but rather the length of a greedy long path between the two vertices, whose construction we explain now. 

For each vertex $v\in G^{(n)}$ with $v\neq 1$, let $\cC(v)\subseteq [v-1]$ denote the set of all vertices that $v$ connects to when  $v$ is added to the graph, and let $v_{\max}\coloneq \max\cC(v)$ denote the vertex with the largest label that $v$ connects to. We then define the \emph{greedy longest path} $\mathrm{GP}(v)$ between $v$ and $1$ recursively as follows. We set $\mathrm{GP}(1)\coloneq \varnothing$ and 
\be 
\mathrm{GP}(v)\coloneq \{v\}\cup \mathrm{GP}(v_{\max}),\qquad\text{for }v\neq 1.
\ee 
We define
\be 
u_n(v)\coloneq |\mathrm{GP}(v)|\qquad\text{for }v\in G^{(n)}
\ee
as the length of the greedy longest path between $v$ and $1$. We observe that in the case $m=1$, i.e.\ the uniform attachment tree with freezing, the notion of the greedy longest path is equivalent to the depth of the vertex. 

The following result presents the joint normality of the length of the greedy longest path and the label of high-degree active vertices in $G^{(n)}$ (or rather, $h^+$ applied to the label), which generalises a result of the second author for the random recursive tree (i.e.\ $m=1$ and $\mathbf x=(1,1,\ldots)$) presented in~\cite[Theorem $2.4$]{Lodewijks.2023}, as well as a result for the length of greedy longest paths in URDs without freezing (i.e.\  arbitrary $m$ and $\mathbf x=(1,1,\ldots)$) by Devroye and Janson~\cite{Devroye.Janson.2011}.

\begin{theorem}
\label{theorem:label_multiple_vertices_conditional_degrees}
Fix $m,k\in\N$ and $(b_v)_{v\in [k]}\in[0,m+1)^k$. Let $V_1,\ldots,V_k\in\mathbb{A}_n$ be distinct active vertices selected uniformly at random. Let $(d_v(n))_{v\in[k]}$ be $k$ integer-valued sequences diverging to infinity such that, for all $v\in[k]$, $\lim_{n\to\infty}d_v(n)/h_n^+ = b_v$. Fix a choice sequence $\mathbf x$ such that Assumption~$\xref$\ref{item:geq1} and~$\xref$\ref{item:lb} are satisfied  with $I_{\xref}=\cO((h_n^+)^{\gamma_0(1/2-\delta)})$, where $\gamma_0\in(1,\min\{(1/2+\delta)^{-1},(1-2\delta)^{-1}\})$ and $\delta$ is as in Part~\ref{item:lb}.
Let $(M_i)_{i\in[k]}$ and $(N_i)_{i\in[k]}$ be i.i.d.\ standard normal random variables. Then, conditionally on the event $\{\deg_n(V_v)\geq d_v\text{ for all }v\in[k]\}$,
\be \ba 
\bigg({}&\frac{u_n(V_v)-(mh_n^+-\frac{m}{m+1}d_v)}{\sqrt{mh_n^+-\frac{m}{(m+1)^2}d_v}}, \frac{h_{V_v}^+-(h_n^+-\frac{1}{m+1}d_v)}{\sqrt{\frac{1}{(m+1)^2}d_v}}\bigg)_{v\in [k]}\\
&\toindis \bigg(M_i\sqrt{\frac{mb_v}{(m+1)^2-b_v}}+N_i\sqrt{1-\frac{mb_v}{(m+1)^2-b_v}},M_i\bigg)_{i\in[k]}.
\ea \ee 
\end{theorem}

\begin{remark}\label{rem:greedy}
$(i)$ The condition that $d_v(n)$ tends to infinity with $n$ for each $v\in[k]$ can be omitted if one is interested in the asymptotic normality of the $u_n(V_v)$ only. In particular, we can take $d_v(n)=0$ for all $n$ and $v$, which for the special case $\mathbf x=(1,1,\ldots)$ recovers the result of Devroye and Janson on the greedy longest path of a typical vertex in URDs.

$(ii)$ The assumption that $I_{\xref}=\cO((h_n^+)^{\gamma_0(1/2-\delta)})$ can be weakened to $I_{\xref}=o(\sqrt{h_n^+})$ when $k=1$. In general, this additional condition which is not present in other results ensures that the contribution to the greedy longest path of vertices $V_1,\ldots, V_k$ by vertices $1,2,\ldots, I_{\xref}$ is sufficiently small, i.e.\ long paths from vertices in $[I_{\xref}]$ to $1$ are sufficiently short.

$(iii)$ Assumption~$\xref$\ref{item:hI} is not included, since the upper bound on $I_{\xref}$ and the choice of $\gamma_0$ already yield that $h_{I_{\xref}}^+\leq I_{\xref}=o(\sqrt{h_n^+})$, so that Assumption~$\xref$\ref{item:hI} is satisfied for any $\eps>0$. \ensymboldremark
\end{remark}

Theorem~\ref{theorem:label_multiple_vertices_conditional_degrees} shows that the label and greedy longest path of typical active vertices, conditionally on having a large degree, does behave significantly different under the influence of freezing. Indeed, $h_n^+=\log n+\cO(1)$ when $\mathbf x=(1,1,\ldots)$, but $h_n^+$ can grow much faster for other choice sequences (see the examples discussed in the next subsection), leading to distinctly different behaviour of the graph $G^{(n)}$, which contrasts with the `local' properties of $G^{(n)}$ presented earlier in this section. 

Similar to studying greedy long paths, one could study greedy \emph{short} paths, where the path follows the vertices with the smallest label rather than the largest label at each step, or the length of the longest or shortest path between a vertex $v$ and $1$. All these types of paths have been studied by Devroye and Janson in~\cite{Devroye.Janson.2011}. We do not investigate these paths here, as we feel the approach used here (a time-reversed construction of the URD with freezing, known as the Kingman coalescent construction) is not sufficiently tractable for a study of these paths, and other techniques are probably more appropriate.

\subsection{Examples}\label{sec:assdiscuss}

To conclude this section, we provide a range of examples of choice sequences $\mathbf x$ that satisfy Assumption~$\xref$. The variety of choice sequences supported by these classes demonstrates that our assumptions are relatively mild and natural.
\begin{example}[Bounded, sparse, or linear freezing]
\label{ex:bsl_freeze}
Let $\mathbf x$ be such that Assumption~$\xref$\ref{item:geq1} is satisfied, and that either $F_i$ is bounded in $i$ (in which case Part~\ref{item:Jn} can be omitted) or $F_i\to\infty$ as $i\to\infty$ so that $\limsup_{i\to\infty} F_i/i <1/2$. In either case, this implies that there exists $\xi>0$ and $N\in\N$ such that $A_i=i-2F_i>\xi i$ for all $i\geq N$, and that $h_n^+=\Theta(\log n)$ (the lower bound follows directly from~\eqref{eq:hn+lb}). As a result, setting $I_{\xref}=\lfloor n^\zeta\rfloor$ when $F_n$ is bounded in $n$, or $I_{\xref}=\lfloor F_n^\zeta\rfloor$ when $F_n$ is unbounded, for some sufficiently small $\zeta=\zeta(\eps,\eta)\in(0,1)$, it follows that Parts~\ref{item:hI} through~\ref{item:Jn} are also satisfied.
\end{example}

\begin{example}[All but polynomial freezing]\label{ex:poly_freeze} Fix $\alpha\in(1/2,1)$ and let $\mathbf x$ be such that Part~\ref{item:geq1} is satisfied, and that
\be 
\frac12< \liminf_{i\to\infty}\frac{\log(i/2-F_i)}{\log i}\leq \limsup_{i\to\infty}\frac{ \log(i/2-F_i)}{\log i} \leq \alpha.
\ee 
As $A_i=i-2F_i$, the bounds imply that there exist $\delta\in(0,\alpha-1/2)$, $\xi\in(0,1-\alpha)$, and $N\in\N$ so that $ i^{1/2+\delta}\leq A_i\leq i^{\alpha+\xi}$ for all $i\geq N$. Since $h_k^+\leq k$ and $|\cA_n|>n/2$, we can choose $\gamma\in(0,1-(\alpha+\xi))$ and set $I_{\xref}=\lfloor n^{\eta\wedge \gamma}\rfloor$, so that for any $\eps\in(0,1)$,
\be 
\eps h_n^+\geq \eps\sum_{i=N}^n \ind_{\{x_i=1\}}\frac{1}{i^{\alpha+\xi}}\geq \eps\sum_{i=\lceil n/2\rceil +N}^n \frac{1}{i^{\alpha+\xi}}\geq (\eps/2+o(1))n^{1-(\alpha+\xi)}\geq  I_{\xref}\geq h_{I_{\xref}}^+,
\ee 
We thus obtain that Parts~\ref{item:hI} through~\ref{item:Jn} are also satisfied.
\end{example}

The authors of~\cite{Bellin.2023} study a class of choice sequences with a linear number of active vertices, similar to the linear case in Example~\ref{ex:bsl_freeze}. In~\cite{Bellin.2024}, they extend this analysis to choice sequences with a polynomial amount of active vertices, similar to Example~\ref{ex:poly_freeze} with $\alpha\in(0,1)$ instead of $\alpha\in(1/2,1)$ as in our case. Their approach, however, requires more precise control over the sequence $(A_n)_{n\in\N}$, leading to stronger  assumptions for the choice sequence compared to Assumption~$\xref$. As a result, our theorems apply to a broader variety of choice sequences (when $\alpha\in(1/2,1)$). 

We conclude with discussing \emph{random} choice sequences. We provide examples of random choice sequences almost surely being supported by the classes in Examples~\ref{ex:bsl_freeze} and~\ref{ex:poly_freeze}.

\begin{example}[Random choice sequences]
All the results stated in this section also apply to random sequences $\mathbf X$ in $\{-1,1\}^\N$ that satisfy Assumption~$\xref$ (and any additional result-specific assumptions) almost surely. Here, Part~\ref{item:Jn} can be omitted if $\mathbf X$ contains finitely many $-1$'s almost surely. Suppose that $\mathbf X=(\mathbf X_1,\mathbf X_2,\ldots)$ satisfies that $\mathbf X_i$ is independent of $\mathbf X_j$ for all $i\neq j$ with $\P{\mathbf X_i=1}\eqcolon p_i$. The following examples satisfy Assumption~$\xref$:
\begin{itemize}
\item $\liminf_{i\to\infty}p_i>1/2$, and conditionally on the event $\{\tau(\mathbf X)>n-1 \text{ for all }n\in\N\}$. 
\item $-\frac12 <\liminf_{i\to\infty} \log(p_i-1/2)/\log i \leq \limsup_{i\to\infty} \log(p_i-1/2)/\log i\leq \alpha-1$ for some $\alpha\in (1/2,1)$, and conditionally on the event $\{\tau(\mathbf X)>n-1 \text{ for all }n\in\N\}$.
\end{itemize}
\end{example}

\textbf{Structure of the paper. } In Section~\ref{sec:king} we introduce an alternative `time-reversed' construction of the URD model with freezing, called the Kingman coalescent construction. We use this construction throughout the remainder of the paper. Section~\ref{sec:degs} provides the proofs of the results presented in Section~\ref{sec:degresults}, where Theorem~\ref{theorem:joint_degree_distribution} is proved in Section~\ref{subsec:asymptotic_joint_degree_distribution} and Corollary~\ref{cor:LLN} and Theorems~\ref{theorem:poisson_convergence}, \ref{theorem:maximum_degree}, and~\ref{theorem:asymptotic_normality_smaller_degrees} are proved in Section~\ref{subsec:large_degree_vertices}. Finally, Section~\ref{sec:singlevertex} provides the proof of Theorem~\ref{theorem:label_multiple_vertices_conditional_degrees} in the case of a single vertex (i.e.\ $k=1$) and this result is extended to arbitrarily many active vertices in Section~\ref{sec:multiplevertices}.

\section{The Kingman coalescent for the URD model with freezing}\label{sec:king}

In this section we introduce the Kingman coalescent construction of the URD model with freezing. This is a generalised version of the Kingman coalescent construction of the random recursive tree with freezing, i.e.\ when $m=1$. This construction has proved fruitful in understanding properties of the RRT model (see~\cite{Addario.Eslava.2018,Eslava.2021,Eslava.2017,Lodewijks.2023}) and the RRT model with freezing (see~\cite{Bellin.2023,Bellin.2024,Brandenberger.2025}) and here we adapt it to the URD model with freezing.

Section~\ref{chap:kingman_coalescent} presents the Kingman coalescent construction, where we show the construction yields a directed acyclic graph with the correct distribution. Section~\ref{subsec:selection_and_connection_sets} introduces several concepts related to the Kingman coalescent that we use in the analysis of the degrees, labels, and greedy longest paths of vertices.

\subsection{The Kingman coalescent}
\label{chap:kingman_coalescent}

Before we formally introduce the Kingman coalescent construction, let us introduce the following terminology for directed acyclic graphs. 

\begin{definition}
\label{def:root+in_comp}
Let $G$ be a locally finite directed acyclic graph. A vertex $v \in G$ is called a $\emph{root}$ if its out-degree equals zero. A root $v$, together with all the vertices connected to it via directed paths in $G$, forms the $\emph{in-component}$ of $v$. See Figure~\ref{fig:components_roots} for an example.
\end{definition}

\begin{figure}[h]
\centering
\begin{tikzpicture}[
vertex/.style={circle,draw,minimum size=14pt,inner sep=1pt},
vertexorange/.style={vertex, draw=orange, text=orange},
vertexpurple/.style={vertex, draw=green, text=green},
>=stealth
]

\node[vertexorange] (v1) at (-1.4,1.6) {1};
\node[vertexpurple] (v2) at (1.4,1.6) {2};
\node[vertex]       (v3) at (0,0)     {3};
\node[vertex]       (v4) at (-1.4,0)   {4};
\node[vertex]       (v5) at (1.4,0)    {5};

\draw[->] (v3) -- (v1);
\draw[->] (v3) -- (v2);
\draw[->] (v3) -- (v4);
\draw[->] (v5) -- (v3);
\draw[->] (v4) -- (v1);
\draw[->] (v5) -- (v2);

\draw[draw=green]
(-0.9,-0.35) -- (1.9,2.7) -- (1.9,-0.35) -- cycle;

\draw[draw=orange]
(-1.9,-0.5) -- (-1.9,2.6) -- (3.0,-0.5) -- cycle;

\end{tikzpicture}
\caption{An example of roots and their in-components in a directed acyclic graph. Vertices 1 and 2 are roots and the  coloured triangular boxes mark their corresponding in-components.}
\label{fig:components_roots}
\end{figure}
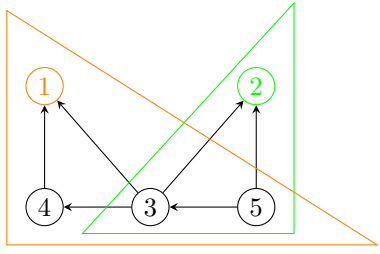
Fix a choice sequence $\mathbf x$ such that Assumption~$\xref$\ref{item:geq1} is satisfied and recall the sequence $(A_n)_{n\in\N}$ from~\eqref{eq:C}. Let us now present the Kingman coalescent construction for URDs with freezing.

\begin{definition}
\label{def:coalescent}
Fix $m,n\in\N$ and fix a choice sequence $\mathbf x$ such that Assumption~$\xref$\ref{item:geq1} is satisfied. The \emph{Kingman $(m,n)$-coalescent with freezing} is the sequence of random graphs $(G_n,\ldots,G_1)$, constructed as follows.\\
We initialise $G_n$ as the graph without edges consisting of $A_n$ \emph{active} vertices labelled $a_1,\ldots,a_{A_n}$ and $F_n$ \emph{frozen} vertices labelled $f_1,\ldots,f_{F_n}$. For $2\leq i\leq n$, construct $G_{i-1}$ from $G_i$ in the following way: If $x_i = -1$, select a frozen vertex uniformly at random and \emph{activate} it by relabelling it as $i$. This vertex is now considered active. If $x_i=1$, select $(m+1)\wedge A_i=(m\wedge A_{i-1})+1$ active roots uniformly at random. Then, independently of everything else, choose one of these $(m+1)\wedge A_i$ roots uniformly at random and connect it via directed edges to each of the other selected roots.
\end{definition}

\begin{remark}
The Kingman $(m,n)$-coalescent with freezing is a generalised version of the $m=1$ case of the Kingman $n$-coalescent with freezing introduced in~\cite{Bellin.2023}.  Setting $m=1$ and $\mathbf x=(1,1,\ldots)$ recovers the Kingman $n$-coalescent for the random recursive tree model.\ensymboldremark
\end{remark}

\begin{remark}
Beyond the generalisation of the  Kingman coalescent construction to the case $m>1$, compared to the construction introduced by Bellin et.\ al for uniform attachment trees with freezing~\cite{Bellin.2023} our definition of this construction has other differences as well. For example, we keep track of all frozen vertices throughout the process and label active vertices differently, whereas the construction in~\cite{Bellin.2023} simply adds an active vertex to the process whenever $x_i=-1$ and labels all active vertices $a_v$ for some $v\in[n]$. We need to keep track of this additional information throughout the construction to be able to analyse the vertex statistics of interest. \ensymboldremark
\end{remark}

See an example of the Kingman $(m,n)$-coalescent for $m=2$ and $n=6$ with $\mathbf x=(1,1,1,-1,1)$ in Figure~\ref{fig:coalescent_process}.
\begin{figure}[h]
\begin{center}
\begin{tikzpicture}[
vertex/.style={
circle,
draw,
minimum size=12pt,
inner sep=1pt,
text height=1.5ex,
text depth=.25ex,
text width=1em,
align=center},
vertexred/.style={vertex, draw=red, text=red},
vertexblue/.style={vertex, draw=blue!80!black, text=blue!80!black},
up/.style={->,>=stealth, thick},
down/.style={<-,>=stealth, thick},
upx/.style={->,>=stealth, thick, dotted},
downx/.style={<-,>=stealth, thick, dotted},
label/.style={inner sep=0pt}
]


\begin{scope}[yshift=0cm]
\node[vertexred] (a3) at (-.45,.3) {$a_3$};
\node[vertexred] (a2) [left=.65cm of a3] {$a_2$};
\node[vertexred] (a1) [left=.65cm of a2] {$a_1$};
\node[vertexred] (a4) [right=.65cm of a3] {$a_4$}
edge [upx, in=60, out=120] (a2)
edge [upx, in=60, out=120] (a1);
\node[vertexblue] (f1) [right=.65cm of a4] {$f_1$};

\node[label] at (-3.6,.3) {$g_6$:};
\end{scope}

\begin{scope}[yshift=-2.3cm]
\node[vertexred] (b3) at (0,.3) {$a_3$};
\node[vertexred] (b2) [left=.9cm of b3] {$a_2$}
edge [upx, in=120, out=60] (b3);
\node[vertex] (b5) [below left=.95cm and .4cm of b2] {$a_4$}
edge [up] (b2);
\node[vertexred] (b1) [left=.9cm of b2] {$a_1$}
edge [down] (b5)
edge [downx, in=120, out=60] (b2);
\node[vertexblue] (b4) [right=.9cm of b3] {$f_1$};

\node[label] at (-3.6,.3) {$g_5$:};
\end{scope}

\begin{scope}[yshift=-4.9cm]
\node[vertexred] (c4) at (0.6,.3) {$4$};
\node[vertexred] (c3) [left=.9cm of c4] {$a_3$};
\node[vertexred] (c1) [left=.9cm of c3] {$a_1$};
\node[vertex] (c5) [below=.95cm of c1] {$a_4$}
edge [up] (c1);
\node[vertex] (c2) [below=.95cm of c3] {$a_2$}
edge [up] (c3)
edge [up] (c1)
edge [down] (c5);

\node[label] at (-3.6,.3) {$g_4$:};
\end{scope}


\begin{scope}[xshift=6.2cm,yshift=0cm]
\node[vertexred] (d4) at (.75,.3) {$4$};
\node[vertexred] (d3) [left=.9cm of d4] {$a_3$}
edge [upx, in=120, out=60] (d4);
\node[vertexred] (d1) [left=.9cm of d3] {$a_1$}
edge [downx, in=120, out=60] (d3);
\node[vertex] (d5) [below=.95cm of d1] {$a_4$}
edge [up] (d1);
\node[vertex] (d2) [below=.95cm of d3] {$a_2$}
edge [up] (d3)
edge [up] (d1)
edge [down] (d5);

\node[label] at (-3,.3) {$g_3$:};
\end{scope}

\begin{scope}[xshift=6.4cm,yshift=-2.3cm]
\node[vertexred] (e4) at (-.1,.1) {$4$};
\node[vertexred] (e1) [left=.9cm of e4] {$a_1$}
edge [upx, in=120, out=60] (e4);
\node[vertex] (e2) [below=.95cm of e1] {$a_2$}
edge [up] (e1);
\node[vertex] (e3) [below=.95cm of e4] {$a_3$}
edge [up] (e4)
edge [up] (e1)
edge [down] (e2);
\node[vertex] (e5) [left=.9cm of e2] {$a_4$}
edge [up] (e1)
edge [up] (e2);

\node[label] at (-3.2,.1) {$g_2$:};
\end{scope}

\begin{scope}[xshift=6.2cm,yshift=-5.6cm]    
\node[vertex] (f2) at (.4,0) {$a_3$};
\node[vertex] (f4) [above left=0.75cm and 0.65 of f2] {$4$}
edge [down] (f2);
\node[vertex] (f1) [left=1.5cm of f2] {$a_1$}
edge [down] (f2)
edge [up] (f4);
\node[vertex] (f5) [below left=0.75cm and 0.65cm of f2] {$a_2$}
edge [up] (f1)
edge [up] (f2);
\node[vertex] (f3) [left=1.5cm of f5] {$a_4$}
edge [up] (f5)
edge [up] (f1);
\node[label] at (-3,1.05) {$g_1$:};
\end{scope}

\end{tikzpicture}
\end{center}
\caption{The sequence $(g_6,\ldots,g_1)$ is a possible realisation of the Kingman $(2,6)$-coalescent $(G_6,\ldots,G_1)$ with $\mathbf x=(1,1,-1,1,1)$. In each step, the active roots are red and the frozen vertices blue. The dotted lines represent the added edges.}
\label{fig:coalescent_process}
\end{figure}
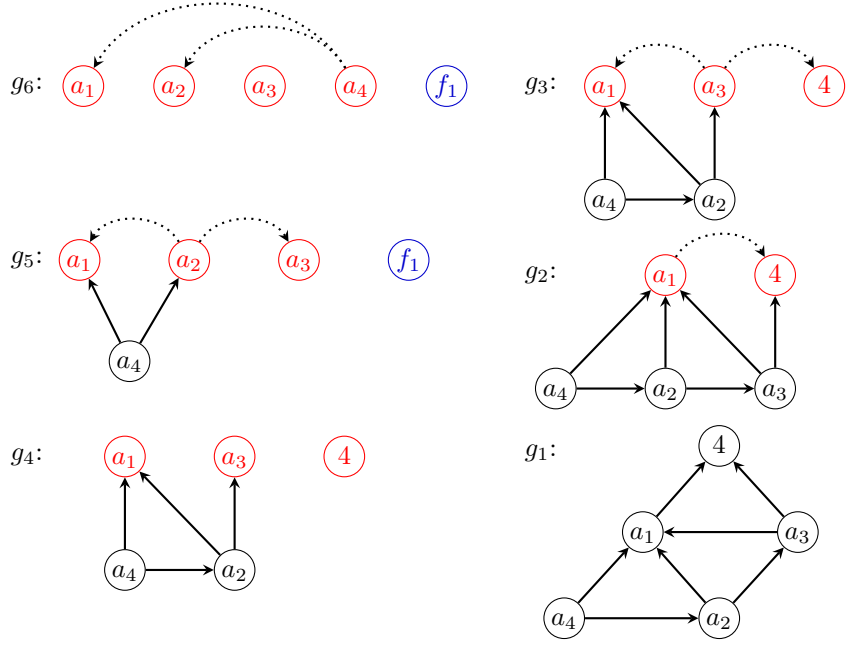
We informally describe the construction process using dice rolls. In each step $i$, we either activate a frozen root if $x_i=-1$ or, if $x_i=1$, select $(m+1) \wedge A_i$ distinct active roots in $G_i$ uniformly at random to participate in a dice roll. We call these the $\emph{selected}$ roots of step $i$ and order them arbitrarily. We roll a fair $((m+1) \land A_i)$-sided dice and say the $k^{\mathrm{th}}$ selected root  \emph{loses} the dice roll if the dice comes up $k$. The other selected roots \emph{win} the dice roll. The root that lost is connected by directed edges to the winning roots. Note that by the definition of a root, the root that lost is no longer a root in the resulting graph (as its out-degree is now non-zero), whilst the roots that won continue to be roots. Consequently, there are $A_i$ active roots in $G_i$ for each $i$, and the number of active roots in the coalescent decreases by one if $x_i=1$ and increases by one if $x_i=-1$ when constructing $G_{i-1}$. 

In the remainder of this section, we prove that there exists a function $\Phi$ that maps $G_1$ onto a directed acyclic graph with relabelled vertices, in such a way that $\Phi(G_1)\overset \dd= G^{(n)}$, where we recall that $G^{(n)}$ is a uniform recursive URD with freezing on $n$ vertices. This is split among Lemmas~\ref{prop:cardinality_increasing_dags} and~\ref{lemma:coalescent_yields_random_chain} and Proposition~\ref{prop:relabelled_coalescent}. Finally, Corollary~\ref{corollary:exchangeability_degrees} described the degrees, labels, and lengths of greedy longest paths of uniform vertices in $G^{(n)}$ to those of fixed vertices in the final graph $G_1$ in the Kingman $(m,n)$-coalescent.

Fix a choice sequence $\mathbf x$ such that Assumption~$\xref$\ref{item:geq1} is satisfied. For $m,n\in\N$, let $\cI_n^{(m)}(\mathbf x)$ be the set of all possible URDs with freezing $G^{(n)}$ that can be constructed. The following result shows that $G^{(n)}$ is a uniform element of $\cI_n^{(m)}(\mathbf x)$.

\begin{lemma}
\label{prop:cardinality_increasing_dags}
Fix a choice sequence $\mathbf x$ such that Assumption~$\xref$\ref{item:geq1} is satisfied. For $m,n\in\N$, we have
\begin{equation}
|\CI_n^{(m)}(\mathbf x)| =  \prod_{\substack{2\leq i\leq n \\ x_i=1}} {A_{i-1} \choose m\wedge A_{i-1}} \prod_{\substack{2\leq i\leq n \\ x_i=-1}} A_{i-1},
\end{equation}
and the $\mathrm{URD}$ $G^{(n)}$ with choice sequence $\mathbf x$ is a uniform element of $\CI_n^{(m)}(\mathbf x)$.
\proof
In the construction of an element of $\CI_n^{(m)}(\mathbf x)$, for each step $i\in\cA_n\setminus\{1\}$, there are $A_{i-1} \choose m\wedge A_{i-1}$ possible choices of active vertices to which the newly added vertex can be connected. Additionally, in each step $j\in\cF_n$, we freeze exactly one of the $A_{j-1}$ active vertices. As the second label entry for frozen vertices is their freezing time (rather than the default value $a$ for active vertices), each frozen vertex's freezing time is uniquely identifiable in the final graph. Hence,
\begin{equation}
|\cI_n^{(m)}(\mathbf x)|= \prod_{\substack{2\leq i\leq n \\ x_i=1}} {A_{i-1} \choose m\wedge A_{i-1}} \prod_{\substack{2\leq j\leq n \\ x_j=-1}} A_{j-1}. 
\end{equation}
Since for a URD, independently at each step $i$, if $x_i=1$, the choices of the distinct vertices are uniform among the $A_{i-1}$ active vertices, and if $x_i=-1$, a uniform random vertex among the $A_{i-1}$ active vertices is frozen, a URD is a uniform element of $\cI_n^{(m)}(\mathbf x)$.\qed
\end{lemma}

We let $\CJ_n^{(m)}(\mathbf x)$ denote the set of all possible outcomes $(g_n,\ldots,g_1)$ of the Kingman $(m,n)$-coalescent $(G_n,\ldots,G_1)$ with choice sequence $\mathbf x$.

\begin{lemma}
\label{lemma:coalescent_yields_random_chain}
Fix a choice sequence $\mathbf x$ such that Assumption~$\xref$\ref{item:geq1} is satisfied. For $m,n \in \N$, we have $|\CJ_n^{(m)}(\mathbf x)| = A_n!F_n!|\CI_n^{(m)}(\mathbf x)|$ and the Kingman $(m,n)$-coalescent with freezing yields a uniform element of $\CJ_n^{(m)}(\mathbf x)$.
\proof
Let $(g_n,\ldots,g_1)$ be an element of $\CJ_n^{(m)}(\mathbf x)$. Here, $g_n$ is always the graph consisting of $A_n$ vertices labelled $a_1\ldots,a_{A_n}$ and $F_n$ vertices labelled $f_1,\ldots,f_{F_n}$ with no edges. At each step $i\geq 2$ such that $x_i=1$, there are $\binom{A_i}{(m+1)\wedge A_i}$ many possibilities to select $(m+1)\wedge A_i$ roots out of the $A_i$ active roots. Then, we choose one out of these $(m+1)\wedge A_i$ selected roots to be the losing root. As we connect the losing root to the other selected roots by directed edges, each choice of these $(m+1)\wedge A_i$ roots and a loser amongst them leads to a distinct realisation of the sequence $(g_n,\ldots,g_1)$. On the other hand, for each step $i\geq 2$ such that $x_i=-1$, we activate a uniform frozen vertex and relabel it $i$. As this relabelling does not depend on which frozen vertex we selected, the sequence $(g_n,\ldots,g_1)$ is only affected by the order in which the vertices $f_1,\ldots,f_{F_n}$ are chosen to be activated. Consequently, we have
\begin{equation}
\label{eq:coalescent_uniform_chain}
|\CJ_n^{(m)}(\mathbf x)| = F_n!\prod_{\substack{2\leq i\leq n\\ x_i=1}}((m+1)\wedge A_i)\binom{A_i}{(m+1)\wedge A_i}=F_n!\prod_{\substack{2\leq i\leq n\\ x_i=1}}A_i \binom{A_{i-1}}{m\wedge A_{i-1}},
\end{equation}
where the last step uses that $\ell \binom k\ell =k\binom{k-1}{\ell-1}$ for all $k\geq \ell\geq 1$ and that $A_i=A_{i-1}+1$ when $x_i=1$. By Lemma~\ref{prop:cardinality_increasing_dags}, to prove that $|\cJ_n^{(m)}(\mathbf x)|=A_n!F_n!|\cI_n^{(m)}(\mathbf x)|$, it suffices to show that
\begin{equation}
\label{eq:induction_card_URD_coalescent}
\prod_{\substack{2\leq i\leq n\\ x_i=1}}A_i  = A_n!\prod_{\substack{2\leq j\leq n \\ x_j=-1}} A_{j-1}=A_n!\prod_{\substack{2\leq j\leq n\\ x_j=-1}}(A_j+1).
\end{equation}
This readily follows from the fact that every step $j$ such that $x_j=-1$ can be uniquely paired with the largest step $i<j$ such that $x_i=1$ and $A_i=A_j+1$, and that the remaining $A_n$ many steps $i$ such that $x_i=1$ satisfy that the values of $A_i$ are unique and in $\{1,2,\ldots, A_n\}$. Since all selections in the Kingman $(m,n)$-coalescent are uniform, it follows that $(G_n,\ldots, G_1)$ is a uniform element of $\cJ_n^{(m)}(\mathbf x)$.\qed
\end{lemma}
Fix a choice sequence $\mathbf x$ such that Assumption~$\xref$\ref{item:geq1} is satisfied. For a graph $G$, we let $V(G)$ and $E(G)$ denote its vertex and edge set, respectively.  There is a natural mapping between $\CJ_n^{(m)}(\mathbf x)$ and $\CI_n^{(m)}(\mathbf x)$: Given $C \coloneq (g_n,\ldots,g_1)\in\CJ_n^{(m)}(\mathbf x)$, we define an edge labelling function on $g_1$ that assigns each edge the step of its addition by
\begin{equation}\label{eq:Lminus}
L_C^-\colon E(g_1) \to \cA_n\setminus\{1\},\qquad e \mapsto \min\{i \in \{2,\ldots,n\}: e \notin E(g_i)\}.
\end{equation}
Now, we define a vertex labelling function $L_C\colon V(g_1)\to \cA_n \times ({\{a\}\cup\cF_n})$ as $L_C(u) \coloneq (x,y)$, where
\be
\label{eq:relabelling_function_L_C}
x\coloneq 
\begin{cases*}
L_C^-((u,v)) \text{ for any edge } (u,v) \in E(g_1)&\mbox{if outdeg$(u)>0,$}\\
1 &\mbox{otherwise},
\end{cases*}
\ee 
where $\text{outdeg}(u)$ denotes the out-degree of the vertex $u$ in $g_1$, and
\be \label{eq:yrelabel}
y\coloneq
\begin{cases*}
u &\mbox{if $u \in \cF_n$},\\
a &\mbox{otherwise}.
\end{cases*}
\ee 
The first element of the tuple $(x,y)$ that $L_C$ assigns to $u$ is the step it lost its first dice roll, or number 1 for the unique vertex that never lost a dice roll in the coalescent process. The second element is either its activation time for a vertex that was initially frozen and has been activated during the  coalescent, or $a$ for any other vertex that was already active in $g_n$. Note that the vertex labelling $L_C(u)$ is well-defined, as the outgoing edges of a fixed vertex $u$ have all been added in the same step, so that for a fixed vertex $u$, the value $L^-((u,v))$ is the same for any edge $(u,v)\in E(g_1)$. If we consider the edges along a directed path, the edge labelling function $L_C^-$ is decreasing by construction. The first entry of the vertex labelling $L_C$ is thus also decreasing along directed paths. Hence, the relabelling $L_C$ of the vertices of the final graph in $C$ constructed by the Kingman $(m,n)$-coalescent with freezing yields a graph in $\CI_n^{(m)}(\mathbf x)$. As an example, Figure~\ref{fig:coalescent_relabelling} shows this relabelling by $L_C$ based on the realisation of Kingman's $(2,6)$-coalescent from Figure~\ref{fig:coalescent_process}.
\begin{figure}
\begin{tikzpicture}[
vertex/.style={
circle,
draw,
minimum size=12pt,
inner sep=1pt,
text height=1.5ex,
text depth=.25ex,
text width=1.5em,
align=center},
up/.style={->,>=stealth, thick},
down/.style={<-,>=stealth, thick},
label/.style={inner sep=0pt}
]


\begin{scope}[xshift=0cm]
\node[vertex] (f2) at (-.3,0) {$a_3$};
\node[vertex] (f4) [above left=0.75cm and 0.65 of f2] {$4$}
edge [down] node[auto] {3} (f2);
\node[vertex] (f1) [left=1.5cm of f2] {$a_1$}
edge [down] node[auto] {3} (f2)
edge [up] node[auto] {2} (f4);
\node[vertex] (f5) [below left=0.75cm and 0.65cm of f2] {$a_2$}
edge [up] node[auto,swap] {5} (f1)
edge [up] node[auto,swap] {5} (f2);
\node[vertex] (f3) [left=1.5cm of f5] {$a_4$}
edge [up] node[auto] {6} (f5)
edge [up] node[auto] {6} (f1);

\node[label] at (-3.8,1.1) {$g_1$:};

\end{scope}


\begin{scope}[xshift=6cm]

\node[vertex] (f2) at (-.8,0) {$3,a$};
\node[vertex] (f4) [above left=0.75cm and 0.65 of f2] {$1,4$}
edge [down] (f2);
\node[vertex] (f1) [left=1.5cm of f2] {$2,a$}
edge [down] (f2)
edge [up] (f4);
\node[vertex] (f5) [below left=0.75cm and 0.65cm of f2] {$5,a$}
edge [up] (f1)
edge [up] (f2);
\node[vertex] (f3) [left=1.5cm of f5] {$6,a$}
edge [up] (f5)
edge [up] (f1);

\node[label] at (-3.8,1.1) {$g_1$ by $L_C$:};

\end{scope}

\end{tikzpicture}
\caption{An example for the vertex relabelling $L_C$ based on the graph $g_1$ from Figure~\ref{fig:coalescent_process}. The left graph shows $g_1$ with its edges labelled by $L_C^-$ and the right graph $g_1$ relabelled by $L_C$.}
\label{fig:coalescent_relabelling}
\end{figure}
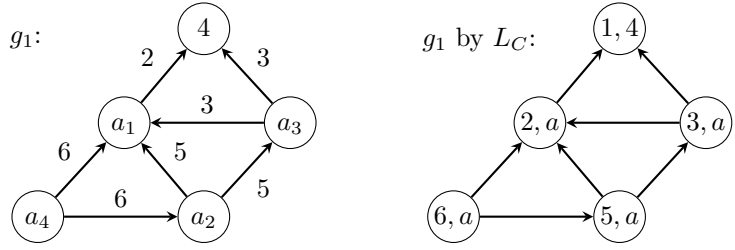

Furthermore, we define the mapping $\Phi\colon \CJ_n^{(m)}(\mathbf x) \to \CI_n^{(m)}(\mathbf x)$ as $C=(g_n,\ldots,g_1) \mapsto G$, where  $G$ equals $g_1$ with its vertices relabelled by $L_C$. In the following proposition we show that relabelling a uniformly random element of $\CJ_n^{(m)}(\mathbf x)$ by $\Phi$ leads to a uniformly random element of $\CI_n^{(m)}(\mathbf x)$, and together with Proposition~\ref{prop:cardinality_increasing_dags} we can deduce that the relabelled Kingman coalescent has the law of a URD.
\begin{proposition}
\label{prop:relabelled_coalescent}
Fix a choice sequence $\mathbf x$ such that Assumption~$\xref$\ref{item:geq1} is satisfied. For $m,n\in\N$ fixed, we have $\Phi((G_n,\ldots,G_1)) \overset \dd= G^{(n)}$.
\proof
We prove that $\Phi$ is surjective and an $A_n!F_n!$-to-$1$ mapping. Lemmas~\ref{prop:cardinality_increasing_dags} and~\ref{lemma:coalescent_yields_random_chain} then yield the desired result. 

Fix $G \in \CI_n^{(m)}(\mathbf x)$. To prove the surjectivity of $\Phi$, we construct a sequence $C\coloneq(g_n,\ldots, g_1)\in\cJ_n^{(m)}(\mathbf x)$ such that $\Phi(C)=G$. To this end, we first construct a different sequence $C'=(g_n',\ldots, g_1')$ of graphs. Here, $g_n'$ consists of $A_n+F_n$ many isolated vertices, and the vertices are labelled by the elements of $\cA_n$. Note that $C'\notin \cJ_n^{(m)}(\mathbf x)$  due to the different labelling. Using elements of $\cA_n$ as labels is, however, more convenient for the construction, and we apply a relabelling later to turn $C'$ into $C\in\cJ_n^{(m)}(\mathbf x)$. 

To construct $C'$, we create $g'_{i-1}$ from $g'_i$ for $i\in\cA_n$ by adding the $m\wedge A_{i-1}$ edges that correspond to the added edges in the construction of $G$ at step $i$, and for $i\in\cF_n$ by relabelling the vertex that corresponds to the vertex that was frozen in $G$ at step $i$, by assigning it the label $i$. First, we define the sets of edges that are added in each step $i\in\cA_n$ of the construction process of $C'$. To this end, for each $(i,x)\in V(G)$, where $i\in\cA_n$ and $x=a$ or $x\in \cF_n$, define the vertex set 
\be 
V_i \coloneq \{v\colon ((i,x),(v,y)) \in E(G) \text{ for some } (v,y)\in\cA_n \times ({\{a\}\cup\cF_n})\}.
\ee 
Here, we keep only $v$ from the tuple $(v,y)$, as the vertices in the sequence $C'$ have a number as their label (from either $\cA_n$ or $\cF_n$), whereas the vertices in the construction of $G$ have tuples as labels. Note that $|V_i|=m\wedge A_{i-1}$ for all $i \in \cA_n\setminus\{1\}$. We then set $E_i \coloneq \{(i,v)\colon v \in V_i\}$ for $i\in \cA_n\setminus\{1\}$. Similarly, for  $(i,x)\in V(G)$ such that $x\in \cF_n$ (i.e.\ all frozen vertices in $G$), we define the edge set $E_i \coloneq \{(x,v)\colon v \in V_i\}$. We have to define the edge sets for vertices in $C'$ corresponding to frozen vertices in $G$ differently since (as we will see in then next paragraph) they are relabelled in the construction process of $C'$, which corresponds to activating frozen vertices in the Kingman coalescent, and their edges are added only after their relabelling. 

We are now ready to formally write down the construction of $C' \coloneq (g'_n,\ldots,g'_1)$. Recall that $g'_n$ is the graph with $A_n+F_n$ isolated vertices labelled by the elements of $\cA_n$. For $n \geq i \geq 2$, construct $g'_{i-1}$ from $g'_i$ as follows. If $i\in\cA_n$, we construct $g'_{i-1}$ by adding all the edges in $E_{i}$ to $g'_i$. If $i\in\cF_n$, there exists a vertex $(v,i)\in V(G)$ for some $v\in\cA_n$. We then construct $g'_{i-1}$ by relabelling the vertex $v$ in $g'_i$ as $i$. Note that, by this construction, $A_n$ many vertices in $g'_1$ have a label from the set $\cA_n$, and $F_n$ vertices from $g'_1$ have a label from the set $\cF_n$. 

We abuse notation and let $L_{C'}^-$ and $L_{C'}$ denote the mappings introduced in~\eqref{eq:Lminus}, \eqref{eq:relabelling_function_L_C}, and~\eqref{eq:yrelabel}, but for the sequence $C'$ (despite that $C'\notin\cJ_n^{(m)}(\mathbf x)$). For $i\in\cA_n\setminus\{1\}$ we have $L_{C'}^-(e) = i$ for each $e\in E_i$, since the edges of set $E_i$ have all been added in step $i$, that is, when we constructed $g'_{i-1}$ from $g'_i$. Consequently, for $i\in \cA_n$ such that $(i,a)\in V(G)$ (the $A_n$ vertices  in $g_1'$ with labels from $\cA_n$), we have $L_{C'}(i) = (i,a)$, and for $j\in\cF_n$ such that $(v,j)\in V(G)$ for some $v\in \cA_n$ (the other $F_n$ vertices in $g'_1$, which have labels from $\cF_n$), we get $L_{C'}(j) = (v,j)$. Hence, the relabelling $\Phi(C')$ yields $G$. 

Now, we create an element $C\in\CJ_n^{(m)}(\mathbf x)$ from $C'$ by relabelling the vertices of $C'$. There are $A_n$ many vertices in $g_1'$ with a label from $\cA_n$. Relabel these vertices as $a_1,\ldots, a_{A_n}$ in an arbitrary order, and relabel them in $g_2', \ldots, g_n'$ with the same labels as well. There are $F_n$ many vertices in $g_1'$ with a label from $\cF_n$, say the labels $\{j_1,\ldots, j_{F_n}\}=\cF_n$, which are in some arbitrary order. The vertex in $g_1'$ with label $j_i$ for some $i\in[F_n]$ is relabelled to $f_i$  \emph{only} in $g_{j_i+k}'$ for $k\in\{1,\ldots, n-j_i\}$. This relabelling yields $C$, and one can verify that $C\in \cJ_n^{(m)}(\mathbf x)$. See Figure~\ref{fig:Cprime_to_C} for an example.

\begin{figure}
\centering
\begin{tikzpicture}[
vertex/.style={
circle,
draw,
minimum size=12pt,
inner sep=1pt,
text height=1.5ex,
text depth=.25ex,
text width=1em,
align=center},
up/.style={->,>=stealth, thick},
down/.style={<-,>=stealth, thick},
upx/.style={->,>=stealth, thick, dotted},
downx/.style={<-,>=stealth, thick, dotted},
label/.style={inner sep=0pt}
]


\begin{scope}[yshift=0cm]
\node[vertex] (a3) at (-.45,.3) {$3$};
\node[vertex] (a2) [left=.65cm of a3] {$5$};
\node[vertex] (a1) [left=.65cm of a2] {$2$};
\node[vertex] (a4) [right=.65cm of a3] {$6$};
\node[vertex] (f1) [right=.65cm of a4] {$1$};

\node[label] at (-3.6,.3) {$g_6'$:};
\end{scope}

\begin{scope}[yshift=-1.0cm]
\node[vertex] (b3) at (0,.3) {$3$};
\node[vertex] (b2) [left=.9cm of b3] {$5$};
\node[vertex] (b5) [below left=.95cm and .4cm of b2] {$6$}
edge [up] (b2);
\node[vertex] (b1) [left=.9cm of b2] {$2$}
edge [down] (b5);
\node[vertex] (b4) [right=.9cm of b3] {$1$};

\node[label] at (-3.6,.3) {$g_5'$:};
\end{scope}

\begin{scope}[yshift=-3cm]
\node[vertex] (c4) at (1.1,.3) {$4$};
\node[vertex] (c3) [left=.9cm of c4] {$3$};
\node[vertex] (c1) [left=.9cm of c3] {$2$};
\node[vertex] (c5) [below=.95cm of c1] {$6$}
edge [up] (c1);
\node[vertex] (c2) [below=.95cm of c3] {$5$}
edge [up] (c3)
edge [up] (c1)
edge [down] (c5);

\node[label] at (-3.6,.3) {$g_4'$:};
\end{scope}


\begin{scope}[xshift=6.6cm,yshift=0cm]
\node[vertex] (a3) at (-.45,.3) {$a_3$};
\node[vertex] (a2) [left=.65cm of a3] {$a_2$};
\node[vertex] (a1) [left=.65cm of a2] {$a_1$};
\node[vertex] (a4) [right=.65cm of a3] {$a_4$};
\node[vertex] (f1) [right=.65cm of a4] {$f_1$};

\node[label] at (-3.6,.3) {$g_6$:};
\end{scope}

\begin{scope}[xshift=6.6cm,yshift=-1.0cm]
\node[vertex] (b3) at (0,.3) {$a_3$};
\node[vertex] (b2) [left=.9cm of b3] {$a_2$};
\node[vertex] (b5) [below left=.95cm and .4cm of b2] {$a_4$}
edge [up] (b2);
\node[vertex] (b1) [left=.9cm of b2] {$a_1$}
edge [down] (b5);
\node[vertex] (b4) [right=.9cm of b3] {$f_1$};

\node[label] at (-3.6,.3) {$g_5$:};
\end{scope}

\begin{scope}[xshift=6.6cm,yshift=-3cm]
\node[vertex] (c4) at (1.1,.3) {$4$};
\node[vertex] (c3) [left=.9cm of c4] {$a_3$};
\node[vertex] (c1) [left=.9cm of c3] {$a_1$};
\node[vertex] (c5) [below=.95cm of c1] {$a_4$}
edge [up] (c1);
\node[vertex] (c2) [below=.95cm of c3] {$a_2$}
edge [up] (c3)
edge [up] (c1)
edge [down] (c5);

\node[label] at (-3.6,.3) {$g_4$:};
\end{scope}
\end{tikzpicture}
\caption{An example of how to construct $C\in\CJ_6^{(2)}(\mathbf x)$ from $C'$ with $C=(g_6,\ldots,g_1)$ and $\x=(1,1,-1,1,1)$ as in Figure~\ref{fig:coalescent_process} (restricted to steps $6,5$, and $4$). Note that in $g_6'$ and $g_5'$ the vertices have labels from the set $\mathcal{A}_6$ and that the in step $4$  activated vertex with label $4$ in $g_4'$ has a different label in $g_6$ and $g_5$ (namely $f_1$)  and has the same label in $g_4$.}
\label{fig:Cprime_to_C}
\end{figure}
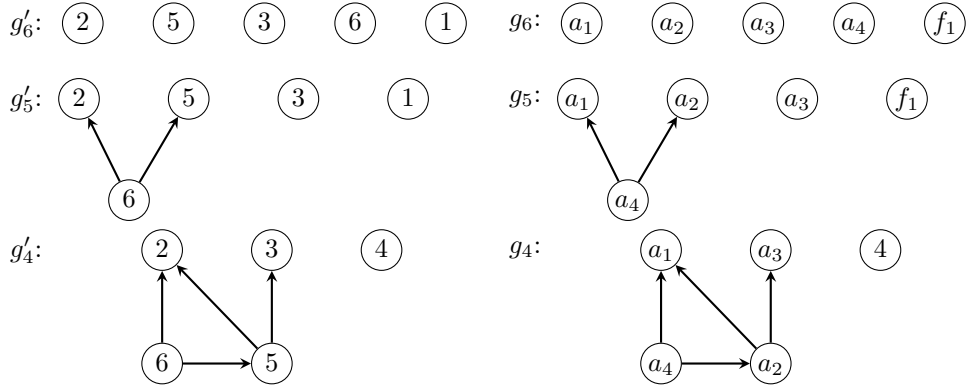

Finally, we observe that the relabelling $\Phi$ applied to $C$ still yields $G$, as it does not depend on the labels $a_1,\ldots, a_{A_n}$ and $f_1,\ldots, f_{F_n}$, but only on the edge labelling $L_C^-$, as in~\eqref{eq:Lminus}, and the vertex relabelling in~\eqref{eq:yrelabel}. This proves the surjectivity of $\Phi$. 

The fact that the labels $a_1,\ldots, a_{A_n}$ and $f_1,\ldots, f_{F_n}$ of the vertices in $g_n$ do not influence the output $\Phi(C)$ implies that for any $C_{\sigma,\tau}$, obtained from $C$ by permuting the vertices in $C$ by the permutations $\sigma\colon \{a_1,\ldots,a_{A_n}\} \to \{a_1,\ldots,a_{A_n}\}$ and $\tau\colon \{f_1,\ldots,f_{F_n}\} \to \{f_1,\ldots,f_{F_n}\}$, we still have $\Phi(C_{\sigma,\tau})=\Phi(C)=G$. With $A_n!$ possible choices for $\sigma$ and $F_n!$ possible choices for $\tau$, we know that there are at least $A_n!F_n!$ preimages under $\Phi$ for each $G \in \CI_n^{(m)}(\mathbf x)$. We conclude that $|\Phi^{-1}(G)|=A_n!F_n!$ due to the first part of Lemma~\ref{lemma:coalescent_yields_random_chain}, so that $\Phi$ is indeed $A_n!F_n!$-to-$1$. Each element of $\CI_n^{(m)}(\mathbf x)$ having the same amount of preimages means that a uniform distribution on $\CJ_n^{(m)}(\mathbf x)$ is preserved under $\Phi$, concluding the proof.\qed
\end{proposition}

Let $m,n \in \N$. Consider the URD $G^{(n)}$ and the Kingman $(m,n)$-coalescent $C=(G_n,\ldots,G_1)$. Recall that the in-degree of a vertex in $G^{(n)}$ is denoted by $\deg_n(v)$. We let $\deg_{G_1}(v')$ denote the in-degree of $v'$, for a vertex $v'$ in $ G_1$. We let $\ell_{G_1}(v')$ denote the first element of the relabelling of $v$ after applying $\Phi$ to $C$. If we write $L_C(v')=(L_C^{(1)}(v'),L_C^{(2)}(v'))$, then 
\be 
\ell_{G_1}(v')\coloneq L_C^{(1)}(v')\qquad\text{for }v'\in V(G_1).
\ee
That is, $\ell_{G_1}(v')$ denotes the timestep at which we added the vertex to $G^{(n)}$ that corresponds to $v'$ in $G_1$. Furthermore, recall that $u_n(v)$ denotes the greedy longest path between $v$ and $1$ in $G^{(n)}$. In the Kingman coalescent, we use an analogous definition of this quantity. Namely, we set 
\be 
u_{G_1}(v')\coloneq u_n(\ell_{G_1}(v'))\qquad \text{for }v'\in V(G_1),
\ee 
where $u_n(\ell_{G_1}(v'))$ is to be interpreted with respect to the graph $\Phi(C)$.

Since the URD and the relabelled Kingman coalescent have the same law, we directly have the following corollary.
\begin{corollary}
\label{corollary:exchangeability_degrees}
For $m,n \in \N$, let $G^{(n)}$ be a URD and let $G_1$ be the resulting graph in the Kingman $(m,n)$-coalescent. Recall that $\mathbb A_n$ and $\mathbb F_n$ are the sets of active and frozen vertices of $G^{(n)}$, respectively. With $\sigma_a$ and $\sigma_f$ uniform permutations of the element of $\mathbb A_n$ and $\mathbb F_n$, respectively,
\be \ba 
\big({}&(\deg_{G_1}(a_v), \ell_{G_1}(a_v), u_{G_1}(a_v))_{v\in [A_n]},(\deg_{G_1}(f_w), \ell_{G_1}(f_w), u_{G_1}(f_w))_{w\in [F_n]}\big)\\
&\overset \dd = \big((\deg_n(\sigma_a(o)),\sigma_a(o),u_n(\sigma_a(o)))_{o\in \mathbb A_n},((\deg_n(\sigma_f(u)),\sigma_f(u),u_n(\sigma_f(u)))_{u\in \mathbb F_n})\big).
\ea \ee 
And, jointly for $i,j\in\N$, $A\subseteq [A_n]$, and $F\subseteq [F_n]$, 
\be \ba
({}&|\{v\in A: \deg_{G_1}(a_v)=i\}|,|\{w\in F: \deg_{G_1}(f_w)=j\}|)\\
&\overset \dd= (|\{o\in \mathbb A_n: \deg_n(\sigma_a(o))=i\}|,|\{u\in \mathbb F_n: \deg_n(\sigma_f(u))=j\}|).
\ea\ee
\end{corollary}

Due to Corollary~\ref{corollary:exchangeability_degrees}, it is equivalent to work with the Kingman coalescent from now on, rather than with the URD model as in Definition~\ref{def:dag}. Consequently, the results in Section~\ref{subsec:statement_results}, where we consider vertices $(V_v)_{v\in[k]}$ selected uniformly at random from $\mathbb A_n$ and $(W_w)_{w\in[\ell]}$ selected uniformly at random from $\mathbb F_n$, can now be thought of as results for fixed vertices $a_1,\ldots, a_k$ and $f_1,\ldots, f_\ell$ in the Kingman $(m,n)$-coalescent.

For ease of writing and to make dependencies clear, but abusing notation, we replace the subscript $G_1$ with $n$ in the following. That is, for a vertex $v$ in the Kingman coalescent, we let $\deg_n(v), \ell_n(v)$, and $u_n(v)$ denote its in-degree, first element of its label (after relabelling by $\Phi$), and greedy longest path, respectively. Also, for simplicity, but being slightly informal, we refer to $\ell_n(v)$ as the \emph{label} of vertex $v$.

\subsection{Selection sets and connection sets}
\label{subsec:selection_and_connection_sets}

We conclude this section by describing the statistics of interest, that is, the degree, label, and length of the greedy longest path, of a vertex in terms of the Kingman coalescent construction, which we use in the forthcoming sections. 

Let $m,n \in \N$ and $(g_n,\ldots, g_1)\in\cJ_n^{(m)}(\mathbf x)$, and recall the definition of in-components from Definition \ref{def:root+in_comp}. For each $i \in [n]$, the graph $g_i$ contains $A_i$ many in-components of active vertices that we denote by  $g_i^{(1)},\ldots,g_i^{(A_i)}$. Note that an active vertex $a_v$ can be contained in multiple in-components. We order the in-components in the following manner. For each $g_i^{(j)}$, list its vertices in increasing order, i.e.\ $(a_{\ell_{1,j}},a_{\ell_{2,j}}, \ldots, a_{\ell_{k,j}},m_{1,j},\ldots, m_{s,j})$ for some $k,s\in\N$ and indices $\ell_{1,j}<\ell_{2,j}<\cdots <\ell_{k,j}$ and $m_{1,j}<\cdots <m_{s,j}$. Here, the indices $m_{t,j}\in\cF_n$ for $t\in[s]$ denote the labels of vertices that were initially frozen in $g_n$ but have been activated, whereas the $a_{\ell_{t,j}}$ for $t\in[k]$ denote the labels of vertices that are active in $g_n$ (and thus stay active throughout the coalescent process). Then, order the components in lexicographical order of their list of vertices, where $a_i\leq a_j$ when $i\leq j$ and $a_\ell<m_s$ for any (originally) active $a_\ell$ and activated vertex $m_s$. For $i\in[n]$ and $v \in \{a_1,\ldots, a_{A_n}\}\cup ([i+1,n]\cap \cF_n) $, let $g_i(v)$ denote the in-component $g_i^{(j)}$ that contains vertex $v$, where $j$ is minimal among all indices $t\in[A_i]$ such that $g_i^{(t)}$ contains $v$. 

Let $s_{v,i}$ be the indicator that $g_i(v)$ is one of the $(m+1)\land A_i$ active in-components (each belonging to a unique active root that is) selected to coalesce at step $i$ if $x_i=1$.  When $s_{v,i}=1$, we say that vertex $v$ is $\emph{selected}$ at step $i$. Note that this definition of a vertex being selected is broader than our definition of a root being selected in Section~\ref{chap:kingman_coalescent}, as now all vertices in the same in-component are considered selected when its root is selected. As the selection of in-components to be merged in each step is independent and uniformly distributed, the variables $(s_{v,i})_{i \in \{2,\ldots,n\}}$ are independent Bernoulli random variables for each vertex $v$ with  $\prob(s_{v,i}=1) =\ind_{\{x_i=1\}} ((m+1)\wedge A_i)/A_i$. We call the set of steps in which vertex $v \in [n]$ is selected the $\emph{selection set}$ of $v$, defined by
\begin{equation}
\label{eq:definition_selection_set}
\cS_n(v) \coloneq \{i \in \{2,\ldots,n\}\colon s_{v,i} = 1\},
\end{equation}
listed as $\cS_n(v) = \{i_{v,1},\ldots,i_{v,S_n(v)}\}$ with $i_{v,1}>i_{v,2}>\ldots>i_{v,S_n(v)}$ and where $S_n(v) \coloneq |\cS_n(v)|$.
To express the degree and the label of a vertex in terms of selection sets, we introduce, for a vertex $v$ and each $i\in \cS_n(v)$, the random variable $r_{v,i}\sim \text{Ber}(1/((m+1)\wedge A_i))$, which determines whether, when $v$ (and thus also the root of $g_i(v)$) is selected, the root of $g_i(v)$ wins or loses the dice roll associated with step $i$. Here, $r_{v,i}=1$ when the root of $g_i(v)$ loses the dice roll and $r_{v,i}=0$ when the root of $g_i(v)$ wins the dice roll.  Similar to $v$ being selected when we select the root of $g_i(v)$, we then also say that vertex $v$ has won/lost the dice roll associated to step $i$ when the root of $g_i(v)$ has won/lost the dice roll. For a vertex $v$, the sequence of random variables $(r_{v,i})_{i \in \{2,\ldots,n\}}$ are mutually independent and they are also independent of $\cS_n(v)$. Using these indicator random variables, the degree of a vertex $v \in [n]$ can be written as
\begin{equation}
\label{eq:express_degree_via_selection_set}
\deg_n(v) = \max\{d \in \{0,\ldots,S_n(v)\}\colon r_{v,i_{v,1}}=\ldots=r_{v,i_{v,d}}=0\}.
\end{equation}
The degree of vertex $v$ thus equals the length of its first winning streak when selected, i.e.\ the length of the first streak of zeros of the indicators $(r_{v,i_{v,\ell}})_{\ell \in [S_n(v)]}$. Similarly, we can express the label of $v$ as the first step in which $v$ is selected and loses the associated dice roll. That is, 
\begin{equation}
\label{eq:express_label_in_terms_selection_sets}
\ell_n(v) = \max\{i \in \cS_n(v)\colon r_{v,i}=1\}=\max\{i \in \{1\}\cup \cS_n(v)\colon r_{v,i}=1\},
\end{equation}
where we set $r_{v,1}\coloneq 1$ for all vertices $v$ to avoid the technicality that the sets in~\eqref{eq:express_label_in_terms_selection_sets} are empty (which happens for the unique root vertex in $G_1$, which never loses a dice roll). Recall that exactly one root loses in each step $i$ such that $x_i=1$ (and none when $x_i=-1$), so that we have $\ell_n(v)\neq\ell_n(v^\prime)$ whenever $v\neq v^\prime$. To summarise, a vertex $v$ is selected at the steps in the set $\cS_n(v)$ and every time it is selected it participates in a dice roll, the outcome of which is determined by $r_{v,i}$ for $i\in \cS_n(v)$. The degree $\deg_n(v)$ equals the number of uninterrupted wins and the label $\ell_n(v)$ equals the step at which $v$ loses the first time, at which time its degree is also determined.  

To characterise the greedy longest path associated to a vertex $v$, we need some additional concepts and notation. We construct the greedy longest path in the coalescent step by step in the following way.  Consider a root $v$ in $g_n$ (active or frozen). At step $\ell_n(v)$ it is selected and loses the associated dice roll for the first time. At this step, it sends outgoing edges to $m\wedge (A_{\ell_n(v)}-1)$ active roots, say $u_1,\ldots, u_{m\wedge (A_{\ell_n(v)}-1)}$. The greedy longest path from $v$ to the vertex that receives label $1$ when relabelling the graph $g_1$ follows one of these $m\wedge (A_{\ell_n(v)}-1)$ edges. Namely, it follows the edge to the vertex $u_j$ that maximises $\ell_n(u_j)$, i.e.\ the vertex with the largest label after relabelling $g_1$. Equivalently, $u_j$ is the first of the roots $u_1,\ldots, u_{m\wedge (A_{\ell_n(v)}-1)}$ to lose a dice roll. Now, again, once this root has lost, it connects itself to $m\wedge (A_{\ell_n(u_j)}-1)$ roots by directed edges, and we again wait for the first loss among these $m\wedge (A_{\ell_n(u_j)}-1)$ roots. We repeat this process until we have constructed $g_1$, and the path we have followed has reached the unique root vertex in $g_1$ that receives label $1$ in the relabelling. 

Let us describe the construction of the greedy longest path (and its length) more formally. Recall the notation from Definition~\ref{def:coalescent}, and let $v\in\{a_1,\ldots, a_{A_n},f_1,\ldots, f_{F_n}\}$. For a vertex $v$ and a step $i\in\cA_n\cap [\ell_n(v)]$, we define the \emph{connection set} $\cC^{(i)}_n(v)$ of $v$ at step $i$ as follows. For each step $i$ such that $x_i=1$, let $u_{1,i},\ldots, u_{m\wedge (A_i-1),i}$ denote the $m \wedge (A_i-1)$ many active roots that are selected and win the associated dice roll in step $i$, and let $u^*_i$ denote the active root that is selected and loses the associated dice roll in step $i$. We initialise $\cC^{(\ell_n(v))}_n(v)\coloneq \{u_{1,\ell_n(v)},\ldots, u_{m\wedge( A_{\ell_n(v)}-1),\ell_n(v)}\}$. Then, given $\cC_n^{(i)}(v)$ for some $i\in \cA_n\cap [\ell_n(v)]$, we define $\cC_n^{(j)}(v)$ for $j= \max\cA_n\cap [i-1]$ as 
\be 
\cC_n^{(j)}(v)\coloneq \begin{cases}
\cC_n^{(i)}(v)&\mbox{if } u^*_j\not\in \cC_n^{(i)}(v), \\
\{u_{1,j},\ldots, u_{m\wedge (A_j-1),j}\}&\mbox{if }u^*_j\in\cC_n^{(i)}(v).
\end{cases}
\ee 
At each step $i$ in $\cA_n\cap[\ell_n(v)]$, the connection set $\cC_n^{(i)}(v)$ consists of all active roots that $v$ is connected to via directed paths, one of which will be part of the greedy longest path from $v$ to the unique root in $g_1$. As a result, we can express the length of the greedy longest path $u_n(v)$ as 
\be \label{eq:unv}
u_n(v)=1+\sum_{i=2}^{\ell_n(v)-1}\ind_{\{x_i=1\}}\sum_{w\in \cC_n^{(i)}(v)}r_{w,i}
\ee 
if $\ell_n(v)>1$ and we have $u_n(v)=0$ if $\ell_n(v)=1$. Here, we observe that the inner sum equals either $0$ or $1$ by definition, since exactly one active root loses the associated dice role in each step.

\section{The degree distribution and large degrees}\label{sec:degs}

In this chapter, we prove Theorems~\ref{theorem:joint_degree_distribution},~\ref{theorem:poisson_convergence},~\ref{theorem:maximum_degree}, and~\ref{theorem:asymptotic_normality_smaller_degrees}, and Corollary~\ref{cor:LLN}, related to the degree distribution and large degrees in the URD model with freezing. We utilise the Kingman coalescent construction introduced in the previous section. Note that these results are generalisations from results obtained in \cite{Addario.Eslava.2018} for RRTs, that is, the case $m=1$ without freezing. Section~\ref{subsec:asymptotic_joint_degree_distribution} contains the proof of   Theorem~\ref{theorem:joint_degree_distribution} and Section~\ref{subsec:large_degree_vertices} presents the proofs of Corollary~\ref{cor:LLN} and Theorems~\ref{theorem:poisson_convergence},~\ref{theorem:maximum_degree}, and~\ref{theorem:asymptotic_normality_smaller_degrees}.

\subsection{Asymptotic joint degree distribution}

\label{subsec:asymptotic_joint_degree_distribution}

To prove Theorem~\ref{theorem:joint_degree_distribution}, we investigate the joint degree distribution of the $k\in\N_0$ active vertices $a_1,\ldots,a_k$ and $\ell\in\N_0$ frozen vertices $f_1,\ldots,f_\ell$ in Kingman's coalescent (where $k+\ell\geq 1$). We start by establishing an upper and a lower bound on the tail of the joint distribution of vertex degrees.
\begin{lemma}
\label{lemma:upper_bound_tail_vertex_degrees}
Fix a choice sequence $\mathbf x$ such that Assumption~$\xref$\ref{item:geq1} is satisfied. Let $n\in\N$ and $d_{a_1},\ldots,d_{a_k},d_{f_1},\ldots,d_{f_\ell} \in [n-1]$, and recall $\theta$ from~\eqref{eq:theta}. Then,
\begin{equation}
\P{\deg_n(a_v)\geq d_{a_v}\text{ for all }v\in[k],\, \deg_n(f_w)\geq d_{f_w}\text{ for all }w\in[\ell]} \leq \theta^{-\sum_{v=1}^k d_{a_v}-\sum_{w=1}^\ell d_{f_w}}.
\end{equation}
\end{lemma}

\begin{proof} 
Let us define for a vertex $v$ in the Kingman coalescent the events 
\be 
\cE(v)\coloneq  \{|\cS_n(v)|\geq d_v\}\qquad \text{and}\qquad \cW(v)\coloneq \{v\text{ wins its first }d_v\text{ dice rolls}\}.
\ee 
Here, $v$ can initially be either active or frozen. Further, we write for $A$ a subset of active and frozen vertices, 
\be 
\cE(A)\coloneq \bigcap_{v\in A}\cE(v)\qquad \text{and}\qquad \cW(A)\coloneq \bigcap_{v\in A}\cW(v). 
\ee 
By \eqref{eq:express_degree_via_selection_set}, we have that $\{\deg_n(v)\geq d_v\}=\cE(v)\cap \cW(v)$. As a result, 
\begin{equation}
\label{eq:degree_upper_bound_lhs}
\begin{aligned}
\mathbb P({}&\deg_n(v)\geq d_{v}\text{ for all }v\in\{a_1,\ldots, a_k, f_1,\ldots, f_\ell\})\\
&=\mathbb P(\cE(\{a_1,\ldots, a_k,f_1,\ldots, f_\ell\})\cap \cW(\{a_1,\ldots, a_k, f_1,\ldots, f_\ell\}))\\
&= \E{\ind_{\cE(\{a_1,\ldots, a_k, f_1,\ldots, f_\ell\})}\mathbb P\big( \cW(\{a_1,\ldots, a_k, f_1,\ldots, f_\ell\})\,\big|\, \cS_n(a_v)_{v\in[k]},\cS_n(f_w)_{w\in[\ell]}\big)}.
\end{aligned}
\end{equation}
If the vertices $(a_v)_{v\in[k]}$ and $(f_w)_{w\in[\ell]}$ are never jointly selected in the coalescent, then the probability in the expected value is at most \be\label{eq:degree_upper_bound_all_disjoint}
\theta^{-\sum_{v=1}^k d_{a_v}-\sum_{w=1}^\ell d_{f_w}},
\ee 
as a given selected vertex wins a dice roll in step $i$ with probability $((A_i-1)\wedge m)/(A_i\wedge (m+1))\leq 1/\theta$, and the dice rolls that are associated with different steps in the coalescent are independent. Bounding the indicator random variables from above by $1$ thus establishes the desired upper bound in this case. When the selection sets $(\cS_n(a_v))_{v\in[k]}$ and $(\cS_n(f_w))_{w\in[\ell]}$ are not disjoint, we can, again, establish \eqref{eq:degree_upper_bound_all_disjoint} as an upper bound. Indeed, suppose that $v_1, v_2,\ldots, v_r$ are all jointly selected at some step $i\in \{2,\ldots,n\}$. The probability that all these vertices win the dice roll equals
\begin{equation}
\max\left\{\frac{((A_i-1)\wedge m)+1-r}{A_i \wedge (m+1)},0\right\}\leq \Big(\frac{m}{m+1}\Big)^r=\theta^{-r}.
\end{equation}
This is a direct result of the fact that left-hand side is increasing in $A_i$ and the inequality
\begin{equation}
\label{eq:product_sum_inequality}
1-\sum_{j=1}^r y_j\leq \prod_{j=1}^r (1-y_j),
\end{equation}
with $y_j\coloneq1/((A_i\wedge m)+1)$ for all $j\in\{1,\ldots,r\}$. Overall, we thus obtain 
\be  
\P{\deg_n(v)\geq d_{v}\text{ for all }v\in\{a_1,\ldots,a_k,f_1,\ldots,f_\ell\}}\leq \theta^{-\sum_{v=1}^k d_{a_v}-\sum_{w=1}^\ell d_{f_w}},
\ee 
as desired.
\end{proof}

For the lower bound, we follow the same procedure as in the proof of Lemma~\ref{lemma:upper_bound_tail_vertex_degrees}, except we introduce the restriction that the selection sets of vertices $a_1,\ldots,a_k,f_1,\ldots,f_\ell$ are disjoint for a sufficiently long time. To facilitate this, we introduce the random variable 
\begin{equation}
\label{eq:definition_tau_k}
\tau_{k,\ell}\coloneq\max\{2\leq i\leq n: s_{v,i}=s_{w,i}=1 \text{ for distinct } v,w\in\{a_1,\ldots,a_k,f_1,\ldots,f_\ell\}\},
\end{equation}
as the first step at which two vertices $v,w \in \{a_1,\ldots,a_k,f_1,\ldots,f_\ell\}$ are selected simultaneously (with $\tau_{k,\ell}\coloneq1$ if $k+\ell = 1$).
We are now ready to formulate the following lower bound on the tail of the degree distribution of vertices $a_1,\ldots,a_k,f_1,\ldots,f_\ell$.
\begin{lemma}
\label{lemma:lower_bound_tail_vertex_degrees}
Let $n\in\N$, recall $\theta$ from~\eqref{eq:theta}, and suppose the choice sequence $\mathbf x$ satisfies Assumption~$\xref$\ref{item:geq1} and $A_i\geq m+1$ for $i\in\{I,\ldots,n\}$ with $I\in[n]$. For any $d_{a_1},\ldots,d_{a_k},d_{f_1},\ldots,d_{f_\ell} \in [n-1]$,
\begin{equation}\begin{aligned}
&\P{\deg_n(a_v)\geq d_{a_v}\text{ for all }v\in[k],\, \deg_n(f_w)\geq d_{f_w}\text{ for all }w\in[\ell]} \\
&\geq \theta^{-\sum_{v=1}^k d_{a_v}-\sum_{w=1}^\ell d_{f_w}}\mathbb P(|\cS_n(v)\cap [I,n]|\geq d_{v}\text{ for all }v\in\{a_1,\ldots, a_k,f_1,\ldots, f_\ell\},\tau_{k,\ell}<I).
\end{aligned}\end{equation}
\end{lemma} 

\begin{proof}
As in the proof of Lemma \ref{lemma:upper_bound_tail_vertex_degrees}, with $C_{k,\ell}\coloneq \{a_1,\ldots, a_k, f_1,\ldots, f_\ell\}$, the desired probability equals
\be \ba 
\mathbb P{}&(|\cS_n(v)|\geq d_{v}\text{ and }v\text{ wins its first }d_{v}\text{ dice rolls, for all }v\in C_{k,\ell})\\
&\geq 
\mathbb P(|\cS_n(v)\cap [I,n]|\geq d_{v}\text{ and }v\text{ wins its first }d_{v}\text{ dice rolls, for all }v\in C_{k,\ell}, \tau_{k,\ell}<I).
\ea\ee  
The event $\{\tau_{k,\ell}<I\}$ implies that all the sets $\cS_n(a_v)\cap [I,n]$ for $v\in[k]$ and $\cS_n(f_w)\cap [I,n]$ for $w\in[\ell]$ are disjoint. As a result, the dice rolls that are associated with each selection are thus independent for all vertices considered. A vertex, when selected, wins a dice roll with probability $1/\theta$ for any step $i\in\{I,\ldots, n\}$ since $A_i\geq m+1$. We can thus write this probability as 
\begin{equation}\begin{aligned} 
\theta^{-\sum_{v=1}^k d_{a_v}-\sum_{w=1}^\ell d_{f_w}}\mathbb P({}&|\cS_n(v)\cap [I,n]|\geq d_{v}\text{ for all }v\in\{a_1,\ldots, a_k, f_1,\ldots, f_\ell\},\tau_{k,\ell}<I),
\end{aligned}\end{equation} 
as desired.
\end{proof}
To make use of Lemma~\ref{lemma:lower_bound_tail_vertex_degrees}, we need tail bounds for the events $\{\tau_{k,\ell}<I\}$, $\{|\cS_n(a_v)\cap [I,n]|<d_{a_v}\}$, $\{|\cS_n(f_w)\cap [I,n]|< d_{f_w}\}$ for $v\in[k],w\in[\ell]$ and a suitable $I\geq 2$. The latter is provided by the next lemma.
\begin{lemma}
\label{lemma:tail_bound_truncated_selection_set}
Fix $c \in (0,m+1)$ and assume the choice sequence $\mathbf x$ satisfies Assumption~$\xref$\ref{item:geq1},~\ref{item:hI} with $\eps<1-c/(m+1)$, and $A_i\geq m+1$ for all $i\in\{I_{\xref},\ldots,n\}$. Then, there exists  $\beta>0$ such that
\begin{equation}
\prob(|\cS_n(a_1)\cap\{I_{\xref},\ldots,n\}|< ch_n^+) = o(n^{-\beta}).
\end{equation}
Additionally, suppose that $\mathbf x$ satisfies Assumption~$\xref$\ref{item:Jn} for some $\eta\in(0,1-c/(m+1))$. Then, there exists $\xi\in(0,1)$ such that
\be 
\P{|\cS_n(f_1)\cap \{I_{\xref},\ldots, n\}|<c\log F_n}=\cO(F_n^{-\xi}).
\ee 
\end{lemma}

To prove Lemma~\ref{lemma:tail_bound_truncated_selection_set}, we need a generalisation of the lower bound for $h_n^+$ presented in~\eqref{eq:hn+lb}, which is the content of the following lemma. 

\begin{lemma}\label{lemma:hnlb}
Fix a choice sequence $\mathbf x$ such that Assumption~$\xref$\ref{item:geq1} is satisfied and integers $1\leq a<b<\infty$. Then, 
\be 
h_b^+(\mathbf x)-h^+_{a-1}(\mathbf x)=\sum_{i=a}^b \ind_{\{x_i=1\}}\frac{1}{A_i(\mathbf x)} \geq \log\Big(\frac{b}{2a}\Big). 
\ee 
In particular, with $a=1$ and $b=n$ we obtain~\eqref{eq:hn+lb}.
\end{lemma}

\begin{proof} 
The proof is similar to that of~\cite[Lemma 16]{Bellin.2023}, which is the case $a=1$ and $b=n$ (up to an additive constant $1$). Let us first assume that there exists $i\in\{a+1,\ldots, b\}$ such that $x_{i-1}=-1$ and $x_i=1$. We construct the 
choice sequence $\mathbf x'$ 
\be 
\mathbf x' =(x_1', x_2',x_3',\ldots)= (x_1, x_2,\ldots, x_{i-2},1,-1,x_{i+1},\ldots). 
\ee 
It is then clear that 
\be 
\sum_{i=a}^b \ind_{\{x_i=1\}}\frac{1}{A_i(\mathbf x)}\geq \sum_{i=a}^b \ind_{\{x_i'=1\}}\frac{1}{A_i(\mathbf x')} .
\ee 
As a result, we iteratively switch entries $-1,1$ to $1,-1$ to obtain the choice sequence 
\be 
\mathbf x'' = (x_1'', x_2'',x_3'',\ldots)=(x_1, x_2,\ldots, x_{a-1}, 1,\ldots, 1,-1,\ldots, -1, x_{b+1},\ldots),
\ee 
where there are
\be 
p_{a,b}(\mathbf{x})\coloneq |\{i\in \{a,\ldots, b\}\colon x_i=1\}| 
\ee
many consecutive $1$ entries and $b-(a-1)-p_{a,b}(\mathbf x)$ many consecutive $-1$ entries. Note that $p_{a,b}(\mathbf x)=p_{a,b}(\mathbf x'')$. We thus arrive at  the lower bound 
\be 
\sum_{i=a}^b \ind_{\{x_i=1\}}\frac{1}{A_i(\mathbf x)}\geq \sum_{i=a}^b \ind_{\{x_i''=1\}}\frac{1}{A_i(\mathbf x'')}=\sum_{i=1}^{p(a,b)(\mathbf x)}\frac{1}{A_{a-1}(\mathbf x)+i} \geq \log\bigg(\frac{A_{a-1}(\mathbf x)+p_{a,b}(\mathbf x)+1}{A_{a-1}(\mathbf x)+1}\bigg).
\ee 
If there does not exist an $i\in\{a+1,\ldots, b\}$ such that $x_{i-1}=-1$ and $x_i=1$, then we observe that $\mathbf x=\mathbf x''$, so that we obtain the same lower bound. Moreover, we derive
\be 
p_{a,b}(\mathbf x)\geq \frac12\max\{0,b-a-A_{a-1}(\mathbf x)\}. 
\ee 
Since $A_i$ is positive for all $i\in[b]$, in particular $A_b>0$. Viewing $(A_i(\mathbf x))_{i\in[b]}$ as a random walk, then at step $a-1$ we are at position $A_{a-1}>0$. To be positive after $b$ steps, one can first make $A_{a-1}-1$ steps downwards, and in the remaining $(b-a-A_{a-1})$ many steps  you need to make at least half of them upwards. This is a necessary condition for $A_b$ to be positive, which yields the lower bound on $p_{a,b}(\mathbf x)$. We thus conclude that 
\be 
\sum_{i=a}^b \ind_{\{x_i=1\}}\frac{1}{A_i(\mathbf x)}\geq \max\bigg\{0,\log\bigg(\frac{b-a+A_{a-1}(\mathbf x) +1}{2(A_{a-1}(\mathbf{x})+1)}\bigg)\bigg\} \geq \log\Big(\frac{b}{2a}\Big),
\ee 
where the final inequality uses that $A_{a-1}(\mathbf x)\leq a-1$ for any choice sequence $\mathbf x$. 
\end{proof} 

We then prove Lemma~\ref{lemma:tail_bound_truncated_selection_set}.

\begin{proof}[Proof of Lemma~\ref{lemma:tail_bound_truncated_selection_set}]
First, we consider the active vertex $a_1$ (this is the same as considering any $a_v$ by exchangeability) and set $Q_n\coloneq |\cS_n(a_1)\cap [I_{\xref},n]|$.  By recalling the definition of the selection sets in~\eqref{eq:definition_selection_set}, we can represent $Q_n$ as a sum of independent Bernoulli random variables, so that
\begin{equation} 
Q_n=\sum_{j=I_{\xref}}^n \ind_{\{x_j=1\}}s_j, 
\end{equation} 
where $s_j\sim \text{Ber}((m+1)/A_j)$ as we assume that $A_j\geq m+1$ for all $I_{\xref}\leq j\leq n$. We have 
\begin{equation} \label{eq:qnlb}
\E{Q_n}=\sum_{j=I_{\xref}}^n \ind_{\{x_j=1\}}\frac{m+1}{A_j}=(m+1)(h_n^+-h_{I_{\xref}-1}^+).
\end{equation} 
By using Assumption~$\xref$\ref{item:hI} with $\eps<1-c/(m+1)$ (so that $(m+1)(1-\eps)>c$), we have  $\E{Q_n}\geq (m+1)(1-\eps)h_n^+>ch_n^+$ for all large $n$. We now apply Bernstein's inequality (see e.g.\ Exercise $5.2.1$ in~\cite{Klenke.2020}) to obtain
\begin{equation}
\label{eq:bern}
\P{Q_n<ch_n^+}\leq \exp\bigg(-\frac12 \frac{(\E{Q_n}-ch_n^+)^2}{\E{Q_n}}\bigg).
\end{equation} 
As the mapping $x\mapsto (x-s)^2/x$ is increasing on $[s,\infty)$ and $\E{Q_n}\geq (m+1)(1-\eps)h_n^+$ for all large $n$, we thus arrive at
\begin{equation} 
\P{Q_n<ch_n^+}\leq \exp\bigg(-\frac{((m+1)(1-\eps)-c)^2}{2(m+1)(1-\eps)}h_n^+(1+o(1))\bigg)=o(n^{-\beta}), 
\end{equation} 
when we choose $\beta\in(0,((m+1)(1-\eps)-c)^2/(2(m+1)(1-\eps)))$ and use \eqref{eq:hn+lb} (or Lemma~\ref{lemma:hnlb} with $a=1,b=n$) in the last step.

Now, consider the frozen vertex $f_1$ and set $Q_n\coloneq |\cS_n(f_1)\cap [I_{\xref},n]|$. We cannot directly apply the same bounds as for the active vertex $a_1$, due to the random step $U$ at which the frozen vertex $f_1$ is activated. We define the quantity
\be 
P_j\coloneq \sum_{i=I_{\xref}}^{j-1}\ind_{\{x_i=1\}}\frac{m+1}{A_i}=(m+1)(h_{j-1}^+-h_{I_{\xref}-1}^+).
\ee
Since $A_i\geq m+1$ for $i\in\{I_{\xref},\ldots,n\}$, the quantity $P_j$ equals the expected number of times vertex $f_1$ is selected in $[I_{\xref},n]$, given that it is activated at step $j$. Fix $C\in(c,(m+1)(1-\eta))$, which is possible by the choice of $\eta$. We now bound
\be \ba 
\mathbb P{}&(|\cS_n(f_1)\cap \{I_{\xref},\ldots n\}|<c\log F_n)\\
&=\E{\P{|\cS_n(f_1)\cap [I_{\xref}, U)|<c\log F_n\,|\, U}}\\
&\leq \E{\ind_{\{P_{U}< C\log F_n\}}+\ind_{\{P_{U}\geq C\log F_n\}}\P{|\cS_n(f_1)\cap  [I_{\xref}, U)|<c\log F_n\,|\, U}}.
\ea \ee 
We can bound the conditional probability from above by using Chernoff's inequality. For $\lambda>0$,
\be \ba 
\mathbb P(|\cS_n(f_1)\cap  [I_{\xref}, U)|<c\log F_n\,|\, U)&\leq \exp(\lambda c\log F_n)\prod_{i=I_{\xref}}^{U-1}\bigg(1+(\e^{-\lambda}-1) \ind_{\{x_i=1\}}\frac{\binom{A_i-1}{m\wedge (A_i-1)}}{\binom{A_i}{(m+1)\wedge A_i}}\bigg)\\
&\leq \exp\bigg(\lambda c\log F_n +(\e^{-\lambda}-1)P_{U}\bigg),
\ea \ee 
where we use that $1+x\leq \e^x$ and that the fraction on the right-hand side equals $(m+1)/A_i$ when $A_i\geq m+1$ (which holds for $i\in\{I_{\xref},\ldots,n\}$ by assumption) to obtain the final inequality. We then define
\be 
J_n(C)\coloneq \min\{j\in \cF_n\colon P_j\geq C\log F_n\}. 
\ee
By noting that $P_j$ is increasing in $j$, we can thus bound 
\be \ba
\mathbb E{}&\big[\ind_{\{P_{U}< C\log F_n\}}+\ind_{\{P_{U}\geq C\log F_n\}}\P{|\cS_n(f_1)\cap  [I_{\xref}, U)|<c\log F_n\,|\, U}\big]\\
&\leq \frac{1}{F_n}\sum_{j\in \cF_n} \Big(\ind_{\{P_j< C\log F_n\}}+\ind_{\{P_j\geq C\log F_n\}}\exp\big(\lambda c\log F_n+(\e^{-\lambda}-1)P_j\big)\Big)\\
&\leq \frac{F_{J_n(C)}}{F_n}+\frac{|\cF_n\cap [J_n(C),n]|}{F_n}\exp\big(\lambda c\log F_n+(\e^{-\lambda}-1)P_{J_n(C)}\big).
\ea \ee
As $\e^{-\lambda}-1\leq -\lambda+\lambda^2/2$ for $\lambda\geq 0$, we can set $\lambda=(P_{J_n(C)}-c\log F_n))/P_{J_n(C)}$ to obtain the upper bound
\be 
\frac{F_{J_n(C)}}{F_n}+\exp\Big(-\frac{(P_{J_n(C)}-  c\log F_n)^2}{2P_{J_n(C)}}\Big).
\ee
As the mapping $x\mapsto (x-s)^2/x$ is increasing on $(s,\infty)$, the definition of $J_n(C)$ thus yields the upper bound
\be
\label{eq:final_ub_freez_sel_sets}
\frac{F_{J_n(C)}}{F_n}+\exp\Big(-\frac{ (C-c)^2}{2C}\log F_n\Big)=\frac{F_{J_n(C)}}{F_n}+\cO(F_n^{-(C-c)^2(2C)^{-1}}).
\ee
It remains to bound the fraction. Fix $\zeta\in(\eta,1)$ close enough to $1$ such that $(m+1)(\zeta-\eta)>C$. Using that $I_{\xref}\leq F_n^\eta$ by Assumption~$\xref$\ref{item:Jn} and applying Lemma~\ref{lemma:hnlb}, we can bound 
\be 
(m+1)(h_{F_n^\zeta}^+-h_{I_{\xref}-1}^+)\geq (m+1)(h_{F_n^\zeta}^+-h_{F_n^\eta-1}^+)\geq (m+1)  \log\big(F_n^{\zeta-\eta}/2\big)  > C\log F_n,
\ee
where the final step holds for all large $n$. It follows that $J_n(C)\leq F_n^\zeta$ for all large $n$. Combined with the fact that $F_n<n/2$ by Assumption$\xref$\ref{item:geq1} and  with~\eqref{eq:final_ub_freez_sel_sets}, there exists $\xi\in(0,1)$ such that 
\be
\P{|\cS_n(f_1)\cap [I_{\xref},n]|<c\log F_n}\leq F_n^{-\xi}, 
\ee 
which concludes the proof.   
\end{proof}

For the tail bound on $\tau_{k,\ell}$ we need to prove that the probability of two vertices $v,w$ among $a_1,\ldots,a_k,f_1,\ldots,f_\ell$ being selected simultaneously early in the coalescent process is sufficiently small, which is made precise in the following lemma.
\begin{lemma}
\label{lemma:tightness_tau_k}
Fix a choice sequence $\mathbf x$ that satisfies Assumption~$\xref$\ref{item:geq1} and~\ref{item:lb}. Then, for any sequence $(I_n)_{n\in\N}$ with $I_n\geq I_{\xref}$ for all large $n$, there exists $\zeta\in(0,1)$ such that
\begin{equation}
\prob(\tau_{k,\ell}\geq I_n)= o(I_n^{-\zeta}).
\end{equation}
\proof
Let $(U_n^w)_{w\in[\ell]}$ be $\ell$ distinct uniform elements of $\cF_n$. Here, $U_n^w$ corresponds to the step at which the frozen vertex $f_w$ is activated. By conditioning on $(U_n^w)_{w\in[\ell]}$, we obtain 
\begin{equation} 
\P{\tau_{k,\ell}<I_n}=\E{\P{\tau_{k,\ell}<I_n\,|\, (U_n^{w})_{w\in[\ell]}}}=\E{\P{\tau_{k,\ell}<I_n\,\Big|\, ( U_n^{(w)})_{w\in[\ell]}}},
\end{equation} 
where the $ U_n^{(w)}$ are the order-statistics of the $U_n^w$. That is, $ U_n^{(1)}>\cdots > U_n^{(\ell)}$. Conditionally on these random variables, the event $\{\tau_{k,\ell}<I_n\}$ holds when at each step $i\in[I_n,n]$ such that $x_i=1$ and $ U_n^{(w-1)}>i> U_n^{(w)}$ for some $w\in[\ell+1]$ (where we set $ U_n^{(0)}=n+1$ and $ U_n^{(\ell+1)}=1$), the $(m+1)\wedge A_i$ many roots that are selected contain at most one of the $k+w-1$ many active roots of interest (the active roots labelled $a_1,\ldots,a_k$ and the first $w-1$ roots that are activated among $f_1,\ldots,f_\ell$). As a result, we have 
\begin{equation}\begin{aligned}
\label{eq:ptaukl}
\mathbb P\big(\tau_{k,\ell}<I_n\,\big|\, ( U_n^{(w)})_{w\in[\ell]}\big)= \prod_{w=1}^{\ell+1} \prod_{\substack{i=\max\{ U_n^{(w)}+1,I_n\}\\ x_i=1}}^{ U_n^{(w-1)}-1} \!\!\!\!\!\!\!\!\!\!\!\!\!\!\!\!\P{\text{At most one root among $k+w-1$ chosen in step $i$}}
\end{aligned}\end{equation} 
By Assumption~$\xref$\ref{item:lb},  we have $A_i\geq m+1$ for all $i\geq I_n$ and all large $n$. For $w\in[\ell+1]$, if $ U_n^{(w-1)}-1>I_n$ and $i\in[\max\{ U_n^{(w)}+1,I_n\}, U_n^{(w-1)}-1]$, we then have 
\begin{equation}\begin{aligned} 
\mathbb P({}&\text{At most one root among $k+w-1$ chosen in step $i$})\\ 
&=\frac{1}{\binom{A_i}{m+1}}\bigg(\binom{A_i-(k+(w-1))}{m+1}+(k+(w-1))\binom{A_i-(k+(w-1))}{m}\bigg)\\ 
&=1-\frac{1}{\binom{A_i}{m+1}}\bigg(\binom{A_i}{m+1}-\binom{A_i-(k+(w-1))}{m+1}-(k+(w-1))\binom{A_i-(k+(w-1))}{m}\bigg).
\end{aligned}\end{equation} 
By repeatedly applying Pascal's rule, we find that, for integers $s,t,r\geq 0$ such that $t-s\geq r$,
\begin{equation}\label{eq:pascal} 
\binom{t}{r}=\binom{t-1}{r-1}+\binom{t-1}{r}=\cdots=\binom{t-s}{r}+\sum_{\ell=1}^s \binom{t-\ell}{r-1}.
\end{equation} 
Applying this equality with $(1)$: $t=A_i$, $r=m+1$, and $s=k+(w-1)$ and $(2)$: $t=A_i-\ell$, $r=m+1$, and $s=k+(w-1)-\ell$, where $\ell\in[k+(w-1)-1]$, we arrive at 
\begin{equation}\begin{aligned} 
\binom{A_i}{m+1}{}&-\binom{A_i-(k+(w-1))}{m+1}-(k+(w-1))\binom{A_i-(k+(w-1))}{m}\\ 
={}&\sum_{\ell=1}^{k+(w-1)-1}\sum_{x=\ell+1}^{k+(w-1)}\binom{A_i-x}{m-1}\leq (k+(w-1))^2\binom{A_i-2}{m-1}\leq (k+(w-1))^2 A_i^{m-1}.
\end{aligned}\end{equation} 
As a result, we have the lower bound
\begin{equation} 
\mathbb P(\text{At most one root among $k+w-1$ chosen in step $i$})\geq 1-\frac{(k+(w-1))^2 A_i^{m-1}}{\binom{A_i}{m+1}}\geq 1-C A_i^{-2}, 
\end{equation} 
where the constant $C>0$ does not depend on $i$ and $w$. We use this in~\eqref{eq:ptaukl} to obtain 
\begin{equation} 
\P{\tau_{k,\ell}<I_n}\geq \mathbb E\Bigg[\prod_{w=1}^{\ell-1}\prod_{\substack{i=\max\{ U_n^{(w)}+1,I_n\}\\ x_i=1}}^{ U_n^{(w)-1}}\big(1-C A_i^{-2}\big)\Bigg]\geq 1-C\sum_{w=1}^{\ell-1}\mathbb E\Bigg[\sum_{\substack{i=\max\{ U_n^{(w)}+1,I_n\}\\ x_i=1}}^{ U_n^{(w)-1}} A_i^{-2}\Bigg].
\end{equation} 
We can further bound the right-hand side from below by $1-C\sum_{i=I_n}^n A_i^{-2}$. By using Assumption~$\xref$\ref{item:lb}, we finally obtain
\begin{equation} 
\P{\tau_{k,\ell}<I_n}=1-\cO(I_n^{-2\delta}).
\end{equation} 
Choosing $\zeta\in(0,2\delta)$ concludes the proof.\qed
\end{lemma}

We are now ready to prove Theorem~\ref{theorem:joint_degree_distribution}.

\begin{proof}[Proof of Theorem~\ref{theorem:joint_degree_distribution}]
We directly obtain the upper bound
\begin{equation}
\P{\deg_n(a_v)\geq d_{a_v}\text{ for all }v\in[k], \,\deg_n(f_w)\geq d_{f_w}\text{ for all }w\in[\ell]} \leq \theta^{-\sum_{v=1}^k d_{a_v}-\sum_{w=1}^\ell d_{f_w}}
\end{equation}
from Lemma~\ref{lemma:upper_bound_tail_vertex_degrees}. For a lower bound we recall $I_{\xref}$ from Assumption~$\xref$ and  apply Lemma~\ref{lemma:lower_bound_tail_vertex_degrees}  combined with a union bound  to arrive at
\begin{equation}
\label{eq:lower_bound_degree_distribution_proof}
\begin{aligned}
\mathbb P{}&(\deg_n(a_v)\geq d_{a_v}\text{ for all }v\in[k],\,\deg_n(f_w)\geq d_{f_w}\text{ for all }w\in[\ell]) \\
&\geq \theta^{-\sum_{v=1}^k d_{a_v}-\sum_{w=1}^\ell d_{f_w}}\bigg(1-\P{\tau_{k,\ell}\geq I_{\xref}}-\!\!\!\!\sum_{v\in\{a_1,\ldots, a_k,f_1\ldots, f_\ell\}}\!\!\!\!\!\!\!\! \P{|\cS_n(v)\cap [I_{\xref},n]|<d_{v}}\bigg).  
\end{aligned}
\end{equation}
Using Lemma~\ref{lemma:tightness_tau_k} and Lemma~\ref{lemma:tail_bound_truncated_selection_set} combined with Assumption~$\xref$\ref{item:Jn}, we bound the right-hand side of~\eqref{eq:lower_bound_degree_distribution_proof} from below by
\begin{equation}
\theta^{-\sum_{v=1}^k d_{a_v}-\sum_{w=1}^\ell d_{f_w}}\Big(1+\cO(n^{-\beta}+I_{\xref}^{-\zeta}+F_n^{-\xi})\Big).
\end{equation}
As $F_n\geq I_{\xref}^{1/\eta}$ for all large $n$ by Assumption~$\xref$\ref{item:Jn} and $I_{\xref}\leq n$, we can take $\alpha\in(0,\min\{\beta,\zeta,\xi/\eta\})\subset(0,1)$ to write the error term as $1+o(I_{\xref}^{-\alpha})$, which concludes the proof.
\end{proof}

\subsection{Moment estimates on the vertex count of a certain degree}
\label{subsec:moment_estimates_vertex_count_degree}

We continue with a moment estimate on the number of vertices of a fixed degree, necessary for the proofs of Theorems~\ref{theorem:poisson_convergence} and~\ref{theorem:asymptotic_normality_smaller_degrees}. Recall the random variables $X_i^{(n)},X_{\geq i}^{(n)},Y_i^{(n)}$, and $Y_{\geq i}^{(n)}$ from~\eqref{eq:Xdef} and~\eqref{eq:Ydef}. To make use of Theorem~\ref{theorem:joint_degree_distribution}, we need the following lemma for the calculation of factorial moments, which is a straightforward adaptation of Lemma 5.1 in~\cite{Addario.Eslava.2018}.
\begin{lemma}[Lemma 5.1, \cite{Addario.Eslava.2018}]
\label{lemma:inclusion_exclusion_degree}
Fix $n\in\N$ and $k,\ell\in[n]$. Fix integers $d_{a_1},\ldots ,d_{a_k}$ and $d_{f_1},\ldots, d_{f_\ell}$. Then,
\be \ba 
\mathbb P{}&\big(\deg_n(a_v)=d_{a_v} \text{ for all }v\in[k], \deg_n(f_w)=d_{f_w}\text{ for all }w\in[\ell]\big)\\
&= \sum_{R\subseteq [k]}\sum_{S\subseteq[\ell]}\!\!(-1)^{|R|+|S|}\P{\deg_n(a_v)\geq d_{a_v}+\ind_{\{v\in R\}},v\in[k], \deg_n(f_w)\geq d_{f_w}+\ind_{\{w\in S\}},w\in[\ell]}.
\ea \ee 
In particular, for integers  $k\geq k'\geq 1$ and $\ell\geq\ell'\geq 1$  and $d_{a_1},\ldots, d_{a_{k}}$ and $d_{f_1},\ldots, d_{f_{\ell}}$, 
\be\ba
\mathbb P({}&\deg_n(a_o)=d_{a_o}, \deg_n(a_v)\geq d_{a_v} \text{ for } 1\leq o\leq k'<v\leq k,\\
&\deg_n(f_u)=d_{f_u}, \deg_n(f_w)\geq d_{f_w}\text{ for }1\leq u\leq \ell'<w\leq \ell)\\
={}&\sum_{R\subseteq [k']}\sum_{S\subseteq[\ell']}\!\!(-1)^{|R|+|S|}\mathbb P(\deg_n(a_v)\geq d_{a_v}+\ind_{\{v\in R\}},v\in[k], \deg_n(f_w)\geq d_{f_w}+\ind_{\{w\in S\}},w\in[\ell]).
\ea\ee 
\end{lemma}
For $r\in\R$ and $a\in\N$, write $(r)_a\coloneq r(r-1)\ldots(r-a+1)$ and $(r)_0\coloneq 1$, and recall the definition of $\eps_n^\aa,\eps_n^\mathrm{f}$ from \eqref{eq:epsn}. We then have the following result.

\begin{proposition}
\label{prop:factorial_moments}
Fix $m\in\N$, $c\in(0,m+1)$, and $K,L\in\N_0$ such that $K+L\geq 1$. Let $\mathbf x$ be a choice sequence that satisfies Assumption~$\xref$ for some $\eps<1-c/(m+1)$ in Part~\ref{item:hI} and some $\eta<1-c/(m+1)$ in Part~\ref{item:Jn}.  Let $i=i(n),i'=i'(n)$ and $j=j(n),j'=j'(n)$ be integers such that $i<i'$ and $j<j'$ for all $n$, $i+\lfloor\log_\theta (A_n)\rfloor\geq 0$ and $j+\lfloor\log_\theta F_n\rfloor\geq 0$ for all $n$, and 
\be 
\limsup_{n\to\infty} \frac{\log_\theta (A_n)+i'}{h_n^+}<c\qquad\text{and}\qquad \limsup_{n\to\infty} \frac{\log_\theta(F_n)+j'}{\log F_n}<c.
\ee 
Then, there exists $\alpha\in(0,1)$ such that for any non-negative integers $b_i,\ldots, b_{i'}$ and $c_j,\ldots, c_{j'}$ such that $b_i+\cdots+b_{i'}=K$ and $c_j+\cdots+c_{j'}=L$, we have 
\be\ba 
\mathbb E{}&\Bigg[\big(X^{(n)}_{\geq i'}\big)_{b_{i'}}\big(Y^{(n)}_{\geq j'}\big)_{c_{j'}}\prod_{i\leq k<i'}\big(X^{(n)}_k\big)_{b_k}\prod_{j\leq \ell<j'}\big(Y^{(n)}_\ell\big)_{c_\ell}\Bigg]\\
&=\Big(\big(\theta^{-i'+\eps_n^\aa}\big)^{b_{i'}}\!\!\prod_{i\leq k<i'}\!\!\big((1-\theta^{-1})\theta^{-k+\eps_n^\aa}\big)^{b_k}\Big)\Big(\big(\theta^{-j'+\eps_n^\mathrm{f}}\big)^{c_{j'}}\!\!\prod_{j\leq \ell <j'}\!\!\big((1-\theta^{-1})\theta^{-\ell+\eps_n^\mathrm{f}}\big)^{c_\ell}\Big)(1+\cO(\delta_n)),
\ea\ee 
where $\delta_n\coloneq \min\{A_n,F_n,I_{\xref}^\alpha\}^{-1}$.
\end{proposition}

\begin{remark}
When $L=0$,  Assumption~$\xref$\ref{item:Jn} can be omitted. When $K+L=1$, Assumption~$\xref$\ref{item:lb} and Assumption~$\xref$\ref{item:Jn} can be weakened to $\lim_{n\to\infty }A_n=\infty$ and  $\lim_{n\to\infty}F_n=\infty$, respectively.
\end{remark}

\begin{proof}
For each $k\in \{i,\ldots, i'\}$ and each integer $v$ with $\sum_{t=i}^{k-1}b_t< v\leq \sum_{t=i}^k b_t$ we set $d_{a_v}\coloneq \lfloor \log_\theta A_n\rfloor +k$. Similarly, for each $\ell\in\{j,\ldots, j'\}$ and each integer $w$ with $\sum_{t=j}^{\ell-1}c_t<w\leq \sum_{t=j}^\ell c_t$ we set $d_{f_w}\coloneq \lfloor \log F_n\rfloor +\ell$. We also define $K'\coloneq K-b_{i'}$ and $L'\coloneq L-c_{j'}$. Then, using \cite[Theorem 2.7]{Hofstad.2017}, we have for the desired factorial moments the representation
\be\ba \label{eq:fact_rep}
\mathbb E{}&\Bigg[\big(X^{(n)}_{\geq i'}\big)_{b_{i'}}\big(Y^{(n)}_{\geq j'}\big)_{c_{j'}}\prod_{i\leq k<i'}\big(X^{(n)}_k\big)_{b_k}\prod_{j\leq \ell<j'}\big(Y^{(n)}_\ell\big)_{c_\ell}\Bigg]\\
&=(A_n)_K (F_n)_L \mathbb P(\deg_n(a_v)=d_{a_v},\deg_n(a_{v'})\geq d_{a_{v'}}\text{ for }1\leq v\leq K'<v'\leq K,\\
&\hphantom{=(A_n)_K (F_n)_L \mathbb P(\, } \deg_n(f_w)=d_{f_w},\deg_n(f_{w'})\geq d_{f_{w'}}\text{ for }1\leq w\leq L'<w'\leq L).
\ea\ee 
Applying Lemma~\ref{lemma:inclusion_exclusion_degree} to the probability on the right-hand side, we can write it as 
\be\ba 
\sum_{R\subseteq [K']}\sum_{S\subseteq[L']}\!\!(-1)^{|R|+|S|}\mathbb P({}&\deg_n(a_v)\geq d_{a_v}+\ind_{\{v\in R\}},v\in [K],\deg_n(f_w)\geq d_{f_w}+\ind_{\{w\in S\}},w\in[L]).
\ea \ee 
Since $d_{a_v}\geq 0$ satisfies $\limsup_{n\to\infty} d_{a_v}/h_n^+<c<m+1$ for all $v\in[K]$ and, analogously, $d_{f_w}\geq 0$ satisfies $\limsup_{n\to\infty} d_{f_w}/\log F_n<c<m+1$ for all $w\in [L]$, we can apply Theorem~\ref{theorem:joint_degree_distribution} to each of the probabilities. As a result, we obtain 
\be \ba 
\sum_{R\subseteq [K']}\sum_{S\subseteq[L']}\!\!(-1)^{|R|+|S|} \theta^{-(|R|+|S|)-\sum_{v=1}^K d_{a_v}-\sum_{v=1}^L d_{f_w}}(1+o(I_{\xref}^{-\alpha})), 
\ea\ee 
for some $\alpha>0$, where the little $o$ term does not depend on $R$ and $S$. We  can hence use the binomial expansion to simplify this expression, to arrive at 
\be 
(1-\theta^{-1})^{K'+L'}\theta^{-\sum_{v=1}^K d_{a_v}-\sum_{v=1}^L d_{f_w}}(1+o(I_{\xref}^{-\alpha})).
\ee 
Using this in~\eqref{eq:fact_rep} thus yields
\be\ba \label{eq:fact_rep_2}
\mathbb E{}&\Bigg[\big(X^{(n)}_{\geq i'}\big)_{b_{i'}}\big(Y^{(n)}_{\geq j'}\big)_{c_{j'}}\prod_{i\leq k<i'}\big(X^{(n)}_k\big)_{b_k}\prod_{j\leq \ell<j'}\big(Y^{(n)}_\ell\big)_{c_\ell}\Bigg]\\
&=(A_n)_K(F_n)_L(1-\theta^{-1})^{K'+L'}\theta^{-\sum_{v=1}^K d_{a_v}-\sum_{v=1}^L d_{f_w}}(1+o(I_{\xref}^{-\alpha})).
\ea \ee 
We now note that $(A_n)_K=A_n^K(1+\cO(A_n^{-1}))$ and $(F_n)_L=(F_n)^L(1+\cO(F_n^{-1}))$. Furthermore, by the definition of $K'$ and $L'$, we can write 
\be 
(1-\theta^{-1})^{K'+L'}=\prod_{i\leq k<i'}(1-\theta^{-1})^{b_k}\prod_{j\leq \ell<j'}(1-\theta^{-1})^{c_\ell}.
\ee 
Finally, recalling that $\eps_n^\aa=\log_\theta A_n-\lfloor \log_\theta A_n\rfloor $ and $\eps_n^\mathrm{f}=\log_\theta  F_n-\lfloor \log_\theta F_n\rfloor$, we can write 
\begin{align} 
K\log_\theta A_n -\sum_{v=1}^K d_{a_v}&=-\sum_{v=K'+1}^K (d_{a_v}-\log_\theta A_n)-\sum_{v=1}^{K'} (d_{a_v}-\log_\theta A_n)\\
&=-(i'-\eps_n^\aa)b_{i'}-\sum_{k=i}^{i'-1} (k-\eps_n^\aa)b_k, 
\intertext{and}
L\log_\theta F_n -\sum_{w=1}^L d_{f_w}&=-\sum_{w=L'+1}^L (d_{f_w}-\log_\theta  F_n)-\sum_{w=1}^{L'} (d_{f_w}-\log_\theta  F_n)\\
&=-(j'-\eps_n^\mathrm{f})c_{j'}-\sum_{\ell=j}^{j'-1} (\ell-\eps_n^\mathrm{f})c_\ell.
\end{align} 
Combining these three observations in~\eqref{eq:fact_rep_2}, recalling that $\delta_n=\min\{A_n,F_n,I_{\xref}^\alpha\}^{-1}$, we arrive at 
\be \ba 
\mathbb E{}&\Bigg[\big(X^{(n)}_{\geq i'}\big)_{b_{i'}}\big(Y^{(n)}_{\geq j'}\big)_{c_{j'}}\prod_{i\leq k<i'}\big(X^{(n)}_k\big)_{b_k}\prod_{j\leq \ell<j'}\big(Y^{(n)}_\ell\big)_{c_\ell}\Bigg]\\
&=(1-\theta^{-1})^{K'+L'}\theta^{K\log_\theta A_n+L\log_\theta F_n -\sum_{v=1}^K d_{a_v}-\sum_{v=1}^L d_{f_w}}(1+\cO(\delta_n))\\
&=\Big(\big(\theta^{-i'+\eps_n^\aa}\big)^{b_{i'}}\!\!\prod_{i\leq k<i'}\!\!\big((1-\theta^{-1})\theta^{-k+\eps_n^\aa}\big)^{b_k}\Big)\Big(\big(\theta^{-j'+\eps_n^\mathrm{f}}\big)^{c_{j'}}\!\!\prod_{j\leq \ell <j'}\!\!\big((1-\theta^{-1})\theta^{-\ell+\eps_n^\mathrm{f}}\big)^{c_\ell}\Big)(1+\cO(\delta_n)),
\ea\ee 
as desired.
\end{proof}

\subsection{Large degree vertices}
\label{subsec:large_degree_vertices}

In this section, we prove Corollary~\ref{cor:LLN} and Theorems~\ref{theorem:poisson_convergence},~\ref{theorem:maximum_degree} and~\ref{theorem:asymptotic_normality_smaller_degrees}, starting with the corollary.

\begin{proof}[Proof of Corollary~\ref{cor:LLN}]
We prove the result for $N_i^{(n)}$, the result for the other three random variables follows in an analogous way. We observe that $N_i^{(n)}=X^{(n)}_{i-\lfloor \log_\theta A_n\rfloor}$, so that, by the conditions on $i=i(n)$ in the statement of Corollary~\ref{cor:LLN}, we can use Proposition~\ref{prop:factorial_moments} to obtain that 
\be 
\mathbb E\big[N_i^{(n)}\big]=\mathbb E\big[X^{(n)}_{i-\lfloor \log_\theta A_n\rfloor}\big]=(1+o(1))(1-\theta^{-1})\theta^{-(i-\lfloor \log_\theta A_n\rfloor)+\eps_n^\aa}=(1+o(1))(1-\theta^{-1})\theta^{-i}A_n,
\ee 
and, similarly,
\be 
\mathbb E\big[\big(N_i^{(n)}\big)_2\big]=\mathbb E\big[\big(X^{(n)}_{i-\lfloor \log_\theta A_n\rfloor}\big)_2\big]=(1+o(1))\big((1-\theta^{-1})\theta^{-i}A_n\big)^2.
\ee
As a result, since $\mathbb E\big[N_i^{(n)}\big]$ tends to infinity by the conditions on $i=i(n)$ (see Remark~\ref{rem:degdistr}$(iii)$) and the two equations imply that $\Var(N_i^{(n)})=o(\mathbb E[N_i^{(n)}]^2)$, the second moment method via Chebyshev's inequality yields the desired result.
\end{proof} 

Theorems~\ref{theorem:poisson_convergence},~\ref{theorem:maximum_degree} and~\ref{theorem:asymptotic_normality_smaller_degrees} are generalisations of three of the main results in \cite{Addario.Eslava.2018} for the random recursive tree, i.e.\ the case $m=1$ and $\mathbf x=(1,1,\ldots)$. Equipped with Proposition~\ref{prop:factorial_moments}, the proofs of the theorems are similar to those in \cite{Addario.Eslava.2018}.

\begin{proof}[Proof of Theorem~\ref{theorem:poisson_convergence}]
Recall that $\cP_\square\overset\dd=  \cP^{\eps^\square}$ for $\square \in \{\aa,\ff\}$ and for some $\eps^\aa,\eps^\ff\in[0,1]$.
As the Poisson distribution is determined by its moments (see for example Corollary 15.33 in \cite{Klenke.2020}), we can imply the convergence in distribution in~\eqref{eq:poisson_convergence_by_continuity_sets_recovered_from_FDDs} along subsequences $(n_\ell)_{\ell\in\N}$ such that $\eps_{n_\ell}^\aa\to\eps^\aa$ and $\eps_{n_\ell}^\ff\to \eps^\ff$ by showing joint convergence of the factorial moments of the random variables $X_i^{(n_\ell)},\ldots,X_{i^\prime-1}^{(n_\ell)},X_{\geq i^\prime}^{(n_\ell)}$ and $Y_j^{(n_\ell)},\ldots,Y_{j^\prime-1}^{(n_\ell)},Y_{\geq j^\prime}^{(n_\ell)}$ to the joint factorial moments of the random variables $\cP_\mathrm{a}\{i\},\ldots, \cP_\aa\{i'-1\},\cP_\aa[i',\infty)$ and $\cP_\mathrm{f}\{j\},\ldots, \cP_\ff\{j'-1\},\cP_\ff[j',\infty)$ for any fixed $i,i',j,j'\in\Z$. For any non-negative integers $b_i,\ldots,b_{i^\prime}$ and $c_j,\ldots, c_{j'}$, by applying Proposition~\ref{prop:factorial_moments} with $K=b_i+\cdots +b_{i'}$ and $L=c_j+\cdots +c_{j'}$, and since $A_{n}$, $F_n$, and $I_{\xref}$ tend to infinity with $n$, 
\be\ba 
\mathbb E{}&\Bigg[\big(X^{(n_\ell)}_{\geq i'}\big)_{b_{i'}}\big(Y^{(n_\ell)}_{\geq j'}\big)_{c_{j'}}\prod_{i\leq k<i'}\big(X^{(n_\ell)}_k\big)_{b_k}\prod_{j\leq \ell<j'}\big(Y^{(n_\ell)}_\ell\big)_{c_\ell}\Bigg]\\
&=\Big(\big(\theta^{-i'+\eps_{n_\ell}^\aa}\big)^{b_{i'}}\!\!\prod_{i\leq k<i'}\!\!\big((1-\theta^{-1})\theta^{-k+\eps_{n_\ell}^\aa}\big)^{b_k}\Big)\Big(\big(\theta^{-j'+\eps_{n_\ell}^\mathrm{f}}\big)^{c_{j'}}\!\!\prod_{j\leq \ell <j'}\!\!\big((1-\theta^{-1})\theta^{-\ell+\eps_{n_\ell}^\mathrm{f}}\big)^{c_\ell}\Big)(1+o(1))\\
&\to \Big(\big(\theta^{-i'+\eps^\aa}\big)^{b_{i'}}\!\!\prod_{i\leq k<i'}\!\!\big((1-\theta^{-1})\theta^{-k+\eps^\aa}\big)^{b_k}\Big)\Big(\big(\theta^{-j'+\eps^\mathrm{f}}\big)^{c_{j'}}\!\!\prod_{j\leq \ell <j'}\!\!\big((1-\theta^{-1})\theta^{-\ell+\eps^\mathrm{f}}\big)^{c_\ell}\Big),
\ea\ee 
as $\ell\to\infty$. On the other hand, the factorial moments of independent Poisson  random variables directly implies that
\be\ba 
\mathbb E\bigg[{}&(\mathcal{P}_\mathrm{a}[i^\prime,\infty))_{b_{i^\prime}}(\mathcal{P}_\mathrm{f}[j^\prime,\infty))_{c_{i^\prime}}\prod_{i\leq k<i^\prime}(\mathcal{P}_\aa\{k\})_{b_k}\prod_{j\leq \ell<j^\prime}(\mathcal{P}_\ff\{\ell\})_{c_k}\bigg]\\
&=\big(\theta^{-i^\prime+\varepsilon^\aa}\big)^{b_{i\prime}}\prod_{i\leq k<i^\prime}\big((1-\theta^{-1})\theta^{-k+\varepsilon^\aa}\big)^{b_k}\big(\theta^{-j^\prime+\varepsilon^\ff}\big)^{c_{i\prime}}\prod_{j\leq \ell<j^\prime}\big((1-\theta^{-1})\theta^{-\ell+\varepsilon^\ff}\big)^{c_k}.
\ea\ee 
Applying Theorem 6.10 of \cite{Janson.Luczak.Rucinski.2000} for convergence in distribution to Poisson random variables via convergence of the factorial moments then yields the joint weak convergence of $\cP_\aa^{(n_\ell)}$ and $\cP_\ff^{(n_\ell)}$ to $\cP_\aa$ and $\cP_\ff$, which concludes the proof.
\end{proof}
\begin{proof}[Proof of Theorem~\ref{theorem:maximum_degree}]
Since 
\be 
\{\Delta_n^\mathrm{a}\geq \floor{\log_\theta A_n}+i^\mathrm{a}_n\}\cap \{\Delta_n^\mathrm{f}\geq \floor{\log_\theta(F_n)}+i^\mathrm{f}_n\} = \{X_{\geq i_n^\mathrm{a}}^{(n)}>0\}\cap \{Y^{(n)}_{\geq i_n^\mathrm{f}}>0\}, 
\ee 
and both random variables on the right-hand side are non-negative, it suffices  to estimate the probability $\prob(X_{\geq i_n^\mathrm{a}}^{(n)} Y_{\geq i_n^\mathrm{f}}^{(n)}>0)$. Let us set $Z_n\coloneq X_{\geq i_n^\mathrm{a}}^{(n)}  Y_{\geq i_n^\mathrm{f}}^{(n)}$ for ease of writing. 
We split the proof into two cases.\\
\textbf{Case 1:} $i_n^\mathrm{a},i_n^\mathrm{f}=\cO(1)$. In this case, 
\be 
(1-\exp(-\theta^{-i_n^\mathrm{a}+\eps_n^\mathrm{a}}))(1-\exp\big(-\theta^{-i_n^\mathrm{f}+\eps_n^\mathrm{f}}))=\Theta(1), 
\ee 
so that showing that
\begin{equation}
\lim_{n\to\infty}(1-\exp(-\theta^{-i_n^\mathrm{a}+\eps_n^\mathrm{a}}))(1-\exp\big(-\theta^{-i_n^\mathrm{f}+\eps_n^\mathrm{f}}))-\prob(Z_n>0)=0
\end{equation}
yields the desired result. We argue by a proof by contradiction. Suppose that there exists $\delta>0$ and a subsequence $(n_k)_{k\in\N}$ for which
\be\ba 
\label{eq:delta_n_theorem_assumption}
\inf_{k\in\N}\big|(1-\exp(-\theta^{-i_{n_k}^\mathrm{a}+\eps_{n_k}^\mathrm{a}}))(1-\exp\big(-\theta^{-i_{n_k}^\mathrm{f}+\eps_{n_k}^\mathrm{f}}))-\prob(Z_{n_k}>0)\big |>\delta.
\ea\ee
Since $(\varepsilon_{n_k}^\square)_{k\in\N}$ is bounded for $\square\in\{\mathrm a,\mathrm f\}$, there exists, due to the Bolzano-Weierstrass theorem, a subsubsequence $(n_{k_\ell})_{\ell\in\N}$ such that $\varepsilon_{n_{k_\ell}}^\square\to\varepsilon^\square$ for some $\varepsilon^\square\in[0,1]$ and both $\square=\mathrm{a}$ and $\square=\mathrm f$. We stress that the subsubsequence $(n_{k_{\ell}})_{\ell\in\N}$ is the \emph{same} in both limits. Then, by Theorem~\ref{theorem:poisson_convergence},  
\be\ba 
\lim_{\ell\to\infty}(1-\exp(-\theta^{-i_{n_{k_\ell}}^\mathrm{a}+\eps_{n_{k_\ell}}^\mathrm{a}}))(1-\exp\big(-\theta^{-i_{n_{k_\ell}}^\mathrm{f}+\eps_{n_{k_\ell}}^\mathrm{f}}))-\prob(Z_{n_{k_\ell}}>0)=0.
\ea\ee
However, this contradicts assumption~\eqref{eq:delta_n_theorem_assumption}.	\\ 
\textbf{Case 2:} Both $i_n^\mathrm{a}\to\infty$ with  $n$ such that $\limsup_{n\to\infty} (i_n^\mathrm{a}+\log_\theta A_n)/h_n^+ <m+1$ and $i_n^\mathrm{f}\to\infty$ with $n$ such that $\limsup_{n\to\infty} (i_n^\mathrm{f}+\log_\theta F_n)/\log F_n<m+1$. For $\prob(Z_n>0)$ we obtain the bounds
\begin{equation}
\label{eq:delta_n_theorem_upper_lower_bound}
\frac{\E{Z_n}^2}{\E{Z_n^2}} \leq \prob(Z_n>0) \leq \E{Z_n},
\end{equation}
where the upper bound follows from Markov's inequality and the lower bound is due to Paley-Zygmund's inequality (see for example Exercise 5.1.1 in \cite{Klenke.2020}). By the conditions on $i_n^\square$ for either $\square=\mathrm{a}$ or $\square=\mathrm{f}$, we deduce from Proposition~\ref{prop:factorial_moments}, 
\begin{equation}
\E{Z_n} = \theta^{-i_n^\mathrm{a}+\eps_n^\mathrm{a}-i_n^\mathrm{f}+\eps_n^\mathrm{f}}(1+o(1))\qquad\text{and}\qquad 
\E{Z_n^2} = (1+o(1))\prod_{\square\in\{\mathrm a,\mathrm f\}}\theta^{-i_n^\square+\eps_n^\square}\big(1+\theta^{-i_n^\square+\eps_n^\square}\big).
\end{equation}
Substituting this into~\eqref{eq:delta_n_theorem_upper_lower_bound} and using that  $i_n^\mathrm{a}\to\infty$ and $i_n^\mathrm{f}\to\infty$, we conclude
\begin{equation}
\prob(Z_n>0) = \theta^{-i_n^\mathrm{a}+\eps_n^\mathrm{a}-i_n^\mathrm{f}+\eps_n^\mathrm{f}}(1+o(1)).
\end{equation}
Then, the result follows from the observation that, since $i_n^\mathrm{a}\to\infty$ and $i_n^\mathrm{f}\to\infty$ and $\theta>1$,
\begin{equation}
\theta^{-i_n^\mathrm{a}+\eps_n^\mathrm{a}-i_n^\mathrm{f}+\eps_n^\mathrm{f}} = \big(1-\exp\big(-\theta^{-i_n^\mathrm{a}+\eps_n^\mathrm{a}-i_n^\mathrm{f}+\eps_n^\mathrm{f}}\big)\big)(1+o(1)),
\end{equation}
as desired. Cases where $i_n^\aa$ is bounded and $i_n^\ff$ tends to infinity (or vice versa), or where $i_n^\square$ is bounded along subsequences only, can be dealt with using a combination of the two cases outlined above. 
\end{proof}

Before we prove Theorem~\ref{theorem:asymptotic_normality_smaller_degrees}, we state the following multivariate version of~\cite[Theorem $1.24$]{Bollobas.2001}, whose proof we defer to Appendix~\ref{sec:appendix}. 

\begin{theorem}\label{thrm:multinorm}
Let $(X_n)_{n\in\N}$ and  $(Y_n)_{n\in\N}$ be sequences of real-valued random variables,  and let $(\lambda_n)_{n\in\N}$, $(\mu_n)_{n\in\N}$ sequences of real numbers such that, for any $r,s\in\N$ and any $t\geq r,u\geq s$, 
\be \label{eq:factmeanass}
\E{(X_n)_r(Y_n)_s}-\lambda_n^r\mu_n^s=o(\lambda_n^{-t}\mu_n^{-u}). 
\ee 
Then, with $Z_1,Z_2$ two independent standard normal random variables, 
\be 
\Big(\frac{X_n-\lambda_n}{\sqrt{\lambda_n}},\frac{Y_n-\mu_n}{\sqrt{\mu_n}}\Big)\toindis (Z_1,Z_2).
\ee 
\end{theorem}

\begin{proof}[Proof of Theorem~\ref{theorem:asymptotic_normality_smaller_degrees}]
We apply Theorem~\ref{thrm:multinorm} to the random variables $X^{(n)}_{i_n^\aa}$ and $Y^{(n)}_{i_n^\ff}$ for appropriate sequences $(i_n^\aa)_{n\in\N}$ and $(i_n^\ff)_{n\in\N}$. We first note that, by the definition of $\eps_n^\aa$ and $\eps_n^\ff$ in~\eqref{eq:epsn}, 
\begin{align*}  
(1-\theta^{-1})\theta^{-(\lfloor \log_\theta A_n\rfloor+i_n^\aa)}A_n&=(1-\theta^{-1})\theta^{-i_n^\aa+\eps_n^\mathrm{a}}
\intertext{and}
(1-\theta^{-1})\theta^{-(\lfloor \log_\theta F_n\rfloor+i_n^\ff)}F_n&=(1-\theta^{-1})\theta^{-i_n^\ff+\eps_n^\mathrm{f}}.
\end{align*}
Since $\eps_n^\aa$ and $\eps_n^\ff$ are bounded and $1-\theta^{-1}$ is constant, we are thus required to prove that
\begin{equation}
\label{eq:asymptotic_normality_method_moments_condition}
\E{\left(X_{i_n^\aa}^{(n)}\right)_b\left(Y_{i_n\ff}^{(n)}\right)_c}-\big((1-\theta^{-1})\theta^{-i_n^\aa+\eps_n^\mathrm{a}}\big)^b\big((1-\theta^{-1})\theta^{-i_n^\ff+\eps_n^\ff}\big)^c=o\big(\theta^{si_n^\aa+ti_n^\ff}\big),
\end{equation}
for all fixed integers $s\geq b\geq 1$ and $t\geq c\geq 1$. Here we recall that $i_n^\square\to-\infty$ with $n$, so that the right-hand side indeed vanishes. We first rewrite
\begin{equation}
\label{eq:asymptotic_normality_parameter_estimation}
\theta^{i_n^\square} = n^{\log(\theta)i_n^\square/\log n}\qquad\text{for }\square\in\{\aa,\ff\}.
\end{equation}
We apply Proposition~\ref{prop:factorial_moments} and note that we assumed that $I_{\xref}=\floor{n^{\eps}}$ for some $\eps>0$ and, due to Assumption~$\xref$\ref{item:lb}, $A_n\geq n^{1/2+\delta}$ for large $n$. Since $\eps_n^\square$ is bounded and $i_n^\square=o(\log n)$ for $\square\in\{\aa,\ff\}$, there exists, for any $b,c\geq1$, a constant $\alpha'\in(0,(\alpha\eps)\wedge (1/2+\delta))$ such that
\begin{equation}
\E{\left(X_{i_n^\aa}^{(n)}\right)_b\left(Y_{i_n\ff}^{(n)}\right)_c}-\big((1-\theta^{-1})\theta^{-i_n^\aa+\eps_n^\mathrm{a}}\big)^b\big((1-\theta^{-1})\theta^{-i_n^\ff+\eps_n^\ff}\big)^c= o\big(n^{-\alpha'}\theta^{-b i_n^\aa-ci_n^\ff}\big)= o(n^{-\alpha'+o(1)}),
\end{equation}
where the final step uses~\eqref{eq:asymptotic_normality_parameter_estimation}. Again applying~\eqref{eq:asymptotic_normality_parameter_estimation} and using that $i_n^\square=o(\log n)$, we see that $\theta^{s i_n^\aa+t i_n^\ff}=\omega(n^{-\beta})$ for any $\beta>0$ (so in particular for any $\beta<\alpha'$) and any fixed $s,t\geq1$, from which we obtain the desired result.
\end{proof}

\section{Properties of vertices with a large degrees}\label{sec:singlevertex}

In this section, we present the proof of Theorem~\ref{theorem:label_multiple_vertices_conditional_degrees} in the case of $k=1$, which is a generalisation of a result from \cite{Lodewijks.2023} for RRTs, again using the Kingman coalescent with $m \in \N$. The general approach of the proof is based on \cite{Lodewijks.2023}, but certain parts are newly developed, since the analysis of the greedy longest path in a URD must be handled more carefully than the analysis of the depth in an RRT. In Section~\ref{chapter:analysis_ungreedy_depth_vertex}, we perform some preliminary analysis on the greedy longest path of vertices, and Section~\ref{subsec:label_ungreedy_depth_theorem} contains the actual proof of the $k=1$ case, which is the first step towards proving Theorem~\ref{theorem:label_multiple_vertices_conditional_degrees} in full generality.

\subsection{Analysis of the greedy longest path of a vertex}
\label{chapter:analysis_ungreedy_depth_vertex}
Let $v$ be an active vertex in the Kingman coalescent and recall its connection sets $\CC_n^{(i)}(v)$ for $i\in\{2,\ldots, n\}$, degree $\deg_n(v)$, label $\ell_n(v)$, and the length of its greedy longest path $u_n(v)$. In this section we analyse the label and greedy longest path of $v$, conditionally on the degree of $v$ being large. Since the active vertices in the Kingman coalescent are exchangeable due to Corrolary~\ref{corollary:exchangeability_degrees}, it suffices to think of $v= a_1$.

We recall that the degree of $v$ corresponds to the length of the winning streak of dice rolls of $v$. Its label corresponds to the step at which $v$ loses a dice roll for the first time (thus ending its winning streak). Further, once $v$ has lost a dice roll for the first time, it connects to $m\wedge A_i$ other active roots, given that this first loss occurred at step $i$. We track these $m\wedge A_i$ vertices and follow the edge to the first of them to lose a dice roll, where we denote the step number at which this vertex first loses as $j$. Again, this vertex sends an edge to $m\wedge A_j$ active roots, and we now track these $m\wedge A_j$ vertices. We repeat this process until we reach the unique root in the graph $g_1$ of the Kingman coalescent $(g_n,\ldots, g_1)$. This establishes the greedy longest path of $v$ and its length $u_n(v)$ equals, as in~\eqref{eq:unv},
\be 
u_n(v)=1+\sum_{i=2}^{\ell_n(v)-1}\ind_{\{x_i=1\}}\sum_{w\in \cC_n^{(i)}(v)}r_{w,i}.
\ee 
Here we recall that $r_{w,i}$ equals one when both the indicator $s_{w,i}$ equals one (i.e.\ when $w$ is selected at step $i$) and when $w$ loses the associated dice roll. As exactly one active vertex loses at each step $i$ such that $x_i=1$, at most one of the $r_{w,i}$ in the sum equals one. However, it is possible that $s_{w,i}=1$ for multiple $w\in \cC_n^{(i)}(v)$. To analyse $u_n(v)$, we start with the following lemma, which shows that we only need to focus on steps $i\in [I_{\xref},n]$ at which, out of all $w\in \cC^{(i)}_n(v)$, exactly one of the $s_{w,i}$ equals one. Here, we recall  $I_{\xref}$ from Assumption~$\xref$. To this end, we define for $2\leq i\leq n$ and $j\in[m]$, 
\be \label{eq:Ij}
C_i^{(j)} \coloneq \ind\{\text{At step $i$, exactly $j$ roots in $\mathcal{C}_n^{(i)}(v)$ are selected and one of them loses}\}.
\ee 
We can then rewrite 
\be \label{eq:unrewrite}
u_n(v)=1+\sum_{i=2}^{\ell_n-1}\ind_{\{x_i=1\}}\sum_{j=1}^{|\cC_n^{(i)}(v)|} C_i^{(j)},
\ee 
and introduce
\be \label{eq:wtun}
\wt u_n(v)\coloneq  \sum_{i=I_{\xref}}^{\ell_n-1} \ind_{\{x_i=1\}}C_i^{(1)}. 
\ee
Recalling $h_n^+$ from \eqref{eq:hn+}, we then have the following result.
\begin{lemma}
\label{lemma:ungreedy_depth_only_one_vertex_connections_set}
Let $\mathbf x$ be a choice sequence that satisfies Assumption~$\xref$\ref{item:geq1} is satisfied. We have
\be 
\prob(u_n(v)\leq u,\ell_n(v)\geq\ell\,|\,\deg_n(v)\geq d)\leq \prob(\widetilde{u}_n(v)\leq u,\ell_n(v)\geq\ell\,|\,\deg_n(v)\geq d).
\ee 
Now, fix $b\in(0,m+1)$, and let $d,\ell,u=d(n),\ell(n),u(n)\in\N$ such that $d< bh_n^+$ for $n\in\N$. Suppose that, additionally, the choice sequence $\mathbf x$ satisfies Assumption~$\xref$\ref{item:hI} with $\varepsilon<1-b/(m+1)$ and Part~\ref{item:lb}. Let $(c_n)_{n\in\N}$ be any sequence such that $c_n=\omega(I_{\xref})$. Then, 
\be
\prob(u_n(v)\leq u,\ell_n(v)\geq\ell\,|\,\deg_n(v)\geq d)\geq \prob(\widetilde{u}_n(v)\leq u-c_n,\ell_n(v)\geq\ell\,|\,\deg_n(v)\geq d)+o(1).  
\ee
\end{lemma}
We use this result later in the proof of Proposition~\ref{prop:depthlabelsinglevertex} (which is Theorem~\ref{theorem:label_multiple_vertices_conditional_degrees} in the case of a single vertex), where we show that, under suitable conditions, the upper and lower bound in Lemma~\ref{lemma:ungreedy_depth_only_one_vertex_connections_set} converge to the same non-zero limit. For the proof of the lemma, we need the following result on the negative correlation between the degree and the length of the greedy longest path of a vertex.
\begin{lemma}
\label{lemma:only_one_vertex_connection_set_lower_bound_technical_calculation}
Let $\mathbf x$ be a choice sequence that satisfies Assumption~$\xref$\ref{item:geq1}  and let $d=d(n),k=k(n)\in\N$.  Then,
\begin{equation}
\prob(u_n(v)-\widetilde{u}_n(v)> k,\deg_n(v)\geq d)\leq\theta^{-d} \prob(u_n(v)-\widetilde{u}_n(v)> k).
\end{equation}
\end{lemma}

\begin{proof}
The proof uses a similar idea as for the upper bound presented in Lemma~\ref{lemma:upper_bound_tail_vertex_degrees} (which does not incorporate the event $\{u_n(v)-\wt u_n(v)>k\}$). Vertex $v$ attains degree at least $d$ when it is selected at least $d$ times and $v$ wins the first $d$ dice rolls associated with the first $d$ steps that $v$ is selected.    By conditioning on the selection set of $v$ and using~\eqref{eq:unrewrite} and~\eqref{eq:wtun}, we can thus write
\be \ba \label{eq:fkgprob}
\mathbb P{}&(u_n(v)-\wt u_n(v)>k, \deg_n(v)\geq d)\\
&=\E{\ind_{\{|\cS_n(v)|\geq d\}}\P{u_n(v)-\wt u_n(v)>k, \text{$v$ wins first $d$ dice rolls}\,|\,\cS_n(v)}}.
\ea\ee 
Then, given $\cS_n(v)$ we let $G(v)$ denote the number of dice rolls $v$ performs before it loses for the first time (without the step of its first loss). Clearly, 
\be 
\{v\text{ wins its first $d$ dice rolls}\}=\{G(v)\geq d\} \ \ \text{and}\ \ \ell_n(v)=\max\{i\in[n]\colon |\cS_n(v)\cap [i,n]|\geq G(v)+1\}.
\ee 
As a result, we see that
\be 
u_n(v)-\wt u_n(v)=1+\sum_{i=2}^{I_{\xref}-1}\ind_{\{x_i=1\}}\sum_{j=1}^{|\cC_n^{(i)}(v)|}C_i^{(j)}+\sum_{i=I_{\xref}}^{\ell_n(v)-1}\sum_{j=2}^{|\cC_n^{(i)}(v)|} \ind_{\{x_i=1\}}C_i^{(j)}
\ee 
is decreasing in $G(v)$, since $\ell_n(v)$ is decreasing in $G(v)$. Furthermore, the indicator random variables $C_i^{(j)}$ for $i\leq \ell_n(v)-1$ are independent of $G(v)$, conditionally on $\cS_n(v)$. We thus conclude that the conditional probability in~\eqref{eq:fkgprob} contains a decreasing and an increasing event with respect to $G(v)$. The FKG inequality thus  yields the upper bound
\be 
\E{\ind_{\{|\cS_n(v)|\geq d\}}\P{G(v)\geq d\,|\, \cS_n(v)}\P{u_n(v)-\wt u_n(v)>k\,|\, \cS_n(v)}}.
\ee 
The term $\P{G(v)\geq d\,|\, \cS_n(v)}$ can be bounded from above by $\theta^{-d}$ on the event $\{|\cS_n(v)|\geq d\}$. Indeed, when selected, vertex $v$ wins a dice roll with probability $(m\wedge (A_i-1))/((m+1)\wedge A_i)\leq \theta^{-1}$, so that $G(v)$ is stochastically dominated by $\min\{|\cS_n(v)|,G_\theta\}$, where $G_\theta$ is a geometric random variable with parameter $\theta^{-1}$ and which is independent of $\cS_n(v)$. As a result, we arrive at 
\be 
\P{u_n(v)-\wt u_n(v)>k, \deg_n(v)\geq d}\leq \theta^{-d}\P{u_n(v)-\wt u_n(v)>k}, 
\ee 
as desired.
\end{proof}
We now provide the proof of Lemma~\ref{lemma:ungreedy_depth_only_one_vertex_connections_set}.
\begin{proof}[Proof of Lemma~\ref{lemma:ungreedy_depth_only_one_vertex_connections_set}]
We directly obtain the upper bound as we have $\widetilde{u}_n(v)\leq u_n(v)$ by definition. For the lower bound, we introduce the event $\cE_n\coloneq \{u_n(v)-\wt u_n(v)\leq c_n\}$ and get 
\be
\label{eq:only_one_vertex_connection_set_lower_bound}
\prob(u_n(v)\leq u,\ell_n(v)\geq\ell\,|\,\deg_n(v)\geq d) \geq \prob(\cE_n\cap \{\widetilde{u}_n(v)\leq u-c_n,\ell_n(v)\geq\ell\}\,|\,\deg_n(v)\geq d).
\ee 
To obtain the desired result, we show that 
\be 
\prob(\cE_n\,|\,\deg_n(v)\geq d) = o(1).  
\ee 
We use Lemma~\ref{lemma:only_one_vertex_connection_set_lower_bound_technical_calculation} with $k=c_n$ and Theorem~\ref{theorem:joint_degree_distribution} to bound
\begin{equation}
\prob(\cE_n^c\,|\,\deg_n(v)\geq d) \leq  \frac{ \prob(\cE_n^c\cap\{\deg_n(v)\geq d\})}{\prob(\deg_n(v)\geq d)}\leq \frac{\theta^{-d}\prob(\cE_n^c)}{\prob(\deg_n(v)\geq d)} = \prob(\cE_n^c)(1+o(1)),
\end{equation}
so that it is sufficient to prove $\prob(\cE_n^c) = o(1)$. Recall the indicator random variables $C_i^{(j)}$ in~\eqref{eq:Ij}. On the event $\{i<\ell_n(v)\}$ and for $i$ such that $x_i=1$, we have 
\begin{equation}
\label{eq:order_of_selection_proabability_for_ungreedy_depth}
\ind_{\{i<\ell_n(v)\}}\prob(C_i^{(j)}=1\,|\, \ell_n(v)) =\ind_{\{i<\ell_n(v)\}}\frac{j}{A_i\wedge (m+1)}\frac{\binom{m}{j}\binom{A_i-m}{A_i\wedge(m+1)-j}}{\binom{A_i }{A_i\wedge(m+1)}}\leq C A_i^{-j},
\end{equation}
where $C>0$ is a universal constant that does not depend on $j$, $i$, and $A_i$. As a result, by combining~\eqref{eq:unrewrite} and~\eqref{eq:wtun} with~\eqref{eq:order_of_selection_proabability_for_ungreedy_depth} and Markov's inequality, we obtain 
\begin{equation}
\label{eq:lower_bound_more_than_one_Selected_ungreedy_depth}
\prob(\cE_n^c) =\P{u_n(v)-\wt u_n(v)>c_n}\leq \prob\bigg(I_{\xref}+\sum_{i=I_{\xref}}^{\ell_n(v)-1}\sum_{j=2}^{|\cC_n^{(i)}(v)|}C_i^{(j)}> c_n\bigg)\leq c_n^{-1}\bigg(I_{\xref}+ \sum_{i=2}^n C m A_i^{-2}\bigg),
\end{equation}
where the final step uses that $|\cC_n^{(i)}(v)|\leq m$ for all $i\in [n]$ and all vertices $v$. As we suppose that Assumption~$\xref$\ref{item:lb} is satisfied for some $\delta>0$, we can bound $A_i\geq i^{1/2+\delta}$ for all $i\geq I_{\xref}$. Since $A_i\geq 1$ for all $i<I_{\xref}$ and $I_{\xref}=o(c_n)$, the right-hand side thus tends to zero, which yields the desired result.
\end{proof}

\subsection{Label and greedy longest path of a vertex with given large degree}
\label{subsec:label_ungreedy_depth_theorem}

Lemma~\ref{lemma:ungreedy_depth_only_one_vertex_connections_set} simplifies the analysis of the greedy longest path of an active vertex $v$, since we can focus on the random variables $C_i^{(1)}$ and ignore the $C_i^{(j)}$ for $j\geq 2$. To control the sum over these random variables $C_i^{(1)}$, we need control over the subset $\cS_n^{(1)}(v)$ of steps $i$ at which exactly one root in $\cC_n^{(i)}(v)$ selected. To be able to do so, we focus on $\cS_n^{(1)}(v)\cap [I_{\xref},n]$, where we recall $I_{\xref}$ from Assumption~$\xref$. Here, we assume without loss of generality that $n$ is large enough so that $A_i\geq 2m\geq m+1$ for all $i\in\{I_{\xref},\ldots,n\}$ (which is possible by Assumption~$\xref$\ref{item:lb}). Focussing only on steps $i\in\{I_{\xref},\ldots,n\}$ simplifies the behaviour of $\cS_n^{(1)}(v)$ somewhat, since we then know that $\cC_n^{(i)}(v)$ contains exactly one root (namely $v$) at steps $i$ before $v$ loses its first dice roll, and exactly $m$ vertices at steps $i\geq I_{\xref}$ after $v$ loses its first dice roll. However, since the size of $\cC_n^{(i)}(v)$ changes once $v$ has lost its first dice roll, the analysis of $\cS_n^{(1)}(v)$ is more complicated than of $\cS_n(v)$. Indeed, where $\cS_n(v)$ can be sampled independently of the dice rolls, this is no longer the case for $\cS_n^{(1)}(v)$. To still facilitate the analysis, we first `decouple' $\cS_n^{(1)}(v)$ from the outcome of the dice rolls.   To this end, we introduce for $i\in\{I_{\xref},\ldots,n\}$ the probabilities
\begin{align}\label{eq:def_p_i^-}
p_i^-&\coloneq\frac{\binom{A_i-1}{m}}{\binom{A_i}{m+1}} \ind_{\{x_i=1\}}= \frac{m+1}{A_i}\ind_{\{x_i=1\}}\ind_{\{A_i\geq m+1\}},
\intertext{and} 
p_i^+&\coloneq \frac{\binom{m}{1}\binom{A_i-m}{m}}{\binom{A_i}{m+1}} \ind_{\{x_i=1\}} 
= \Big(\frac{m(m+1)}{A_i}+e_i\Big) \ind_{\{x_i=1\}}\ind_{\{A_i\geq 2m\}},\label{eq:def_pi_+}
\end{align}
with $e_i=\cO(A_i^{-2})$ as $i\to\infty$ (as we assume that $A_i\to\infty$ with $i$). Here, we think of $p_i^-$ as the probability to select exactly one vertex from $\cC_n^{(i)}(v)$ at a step $i$ \emph{before} $v$ has lost its first dice roll, and of $p_i^+$ as the probability to select exactly one vertex from $\cC_n^{(i)}(v)$ at a step $i$ \emph{after} $v$ has lost its first dice roll.  The difference between the two probabilities arises from the fact that $C_n^{(i)}(v)$ contains $1$ (resp.\ $m$) vertices before (resp.\ after) $v$ has lost its first dice roll. Note that here, we use that $A_i\geq 2m$ for all $i\in\{I_{\xref},\ldots,n\}$, so that the indicators in
\eqref{eq:def_p_i^-} and~\eqref{eq:def_pi_+} equal $1$. We then define two independent sequences of random variables $(s^-_{i})_{I_{\xref}\leq i\leq n}$ and $(s^+_{i})_{I_{\xref}\leq i\leq n}$ with $s^-_{i}\sim\text{Ber}(p_i^-)$ and $s^+_{i}\sim\text{Ber}(p_i^+)$. Define furthermore, analogous to~\eqref{eq:definition_selection_set}, the sets
\begin{equation}
\label{eq:definition_selection_sets_ungreedy_depth_single_vertex}
\cS^-_n(v) \coloneq \{i \in \{I_{\xref},\ldots,n\}\colon s^-_{i} = 1\}\qquad \text{and}\qquad \cS^+_n(v) \coloneq \{i \in \{I_{\xref},\ldots,n\}\colon s^+_{i} = 1\}.
\end{equation}
We intuitively think of $\cS^-_n(v)$ and $\cS_n^+(v)$ as the set of all steps $i$ at which exactly one vertex in $\cC_n^{(i)}(v)$ is selected  before  and after $v$ has lost its first dice roll, respectively. This is not exactly true, of course, since $s_i^-=s_i^+=1$ could hold for some $i$, or $s_i^-=1$ despite $v$ having lost its first dice roll already, or $s_i^+=1$ despite not $v$ having lost its first dice roll, yet.  Let us make this intuition more precise by coupling $\cS_n^{(1)}(v)$ to $\cS_n^-(v)$ and $\cS_n^+(v)$. For each $i\in\cS_n^-(v)$, let $r_{v,i}^-\sim \text{Ber}(1/(m+1))$ be a Bernoulli random variable. The $(r_{v,i}^-)_{i\in\cS_n^-(v)}$ are mutually independent and are also independent of the sets $\cS_n^-(v)$ and $\cS_n^+(v)$. These $r_{v,i}^-$ correspond to $v$ losing a dice roll at step $i$. Then, for all $i\in \{I_{\xref},\ldots,n\}$, we set
\be\label{eq:Sncoupling} 
i\in \cS_n^{(1)}(v) \text{ when}\begin{cases} 
i\in \cS_n^-(v) \text{ and } r_{v,j}^-=0\text{ for all }j\in [i+1,n]\cap \cS_n^-(v),\text{ or} \\
i\in \cS_n^+(v) \text{ and } r_{v,j}^-=1\text{ for some }j\in [i+1,n]\cap \cS_n^-(v).
\end{cases}
\ee 
Furthermore, we set $r_{v,i}=r_{v,i}^-$ for every $i\in \cS_n^{(1)}(v)\cap \cS_n^-(v)$, where we recall the random variables $(r_{v,i})_{i\in \cS_n(v)}$ from Section~\ref{subsec:selection_and_connection_sets}. This also immediately extends the coupling, in the sense that we have now coupled $\deg_n(v)$ and $\ell_n(v)$ to $\cS_n^-(v),\cS_n^+(v)$, and $(r_{v,i}^-)_{i\in \cS_n^-(v)}$ as well. This latter part holds only if $\ell_n(v)\geq I_{\xref}$, otherwise we say the coupling \emph{fails}. As we will see, the probability of the coupling failing  tends to zero with $n$.  See a graphical representation of this coupling in Figure~\ref{fig:couple}.  It is relatively straightforward to check that this coupling yields the desired distribution for the set $\cS_n^{(1)}(v)$, and hence also for $\deg_n(v)$ and $\ell_n(v)$. The advantage is that we have `decoupled` the events  $\{i\in \cS_n^{(1)}(v)\}$ and $\{i\geq \ell_n(v)\}$, since we can independently sample the sets $\cS_n^-(v)$ and $\cS_n^+(v)$ and construct $\cS_n^{(1)}(v)$ by first taking elements from $\cS_n^-(v)$ and `switching' to taking elements from $\cS_n^+(v)$ once $v$ loses its first dice roll (i.e.\ when the value of $\ell_n(v)$ is determined).

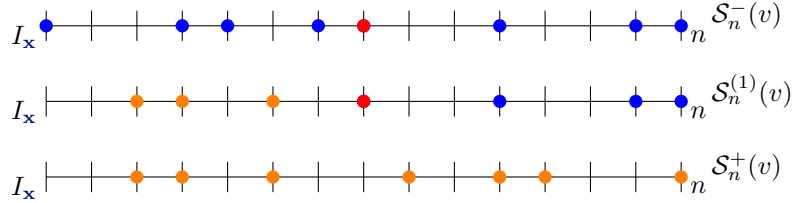
\begin{figure}[h] 
\centering 
\begin{tikzpicture}

\def\n{15}      
\def\spacing{0.6}

\foreach \y in {0,1,2} {

\draw (0,\y) -- ({(\n-1)*\spacing},\y);

\ifnum\y=2 
\foreach \i in {0,...,14} {
\draw (\i*\spacing,\y-0.2) -- (\i*\spacing,\y+0.2);
}
\foreach \i in {0,3,4,6,7,10,13,14} { 
\fill[blue] (\i*\spacing,\y) circle (2.5pt);
\ifnum\i=7 
\fill[red] (\i*\spacing,\y) circle (2.5pt);
\fi 
}
\fi 
\ifnum\y=1
\foreach \i in {0,...,14} {
\draw (\i*\spacing,\y-0.2) -- (\i*\spacing,\y+0.2);
}
\foreach \i in {7,10,13,14} { 
\fill[blue] (\i*\spacing,\y) circle (2.5pt);
\ifnum\i=7 
\fill[red] (\i*\spacing,\y) circle (2.5pt);
\fi 
}
\foreach \i in {2,3,5} {
\fill[orange] (\i*\spacing,\y) circle (2.5pt);
}
\fi 
\ifnum\y=0
\foreach \i in {0,...,14} {
\draw (\i*\spacing,\y-0.2) -- (\i*\spacing,\y+0.2);
}
\foreach \i in {2,3,5,8,10,11,14} {
\fill[orange] (\i*\spacing,\y) circle (2.5pt);
}
\fi 
\node[left,  yshift=-4pt] at (0,\y) {$I_{\xref}$};
\node[right, yshift=-4pt] at ({(\n-1)*\spacing},\y) {$n$};

\ifnum\y=2
\node[right=8pt,yshift=4pt] at ({(\n-1)*\spacing},\y) {$\cS_n^-(v)$};
\fi
\ifnum\y=1
\node[right=8pt,yshift=4pt] at ({(\n-1)*\spacing},\y) {$\cS_n^{(1)}(v)$};
\fi
\ifnum\y=0
\node[right=8pt,yshift=4pt] at ({(\n-1)*\spacing},\y) {$\cS_n^+(v)$};
\fi
}

\end{tikzpicture}
\caption{The coupling of $\cS_n^{(1)}(v)\cap [I_{\xref},n]$ with $\cS_n^-(v)$ and $\cS_n^+(v)$. The top line represents the elements of $\cS_n^-(v)$ and the red dot is the first element $i\in\cS_n^-(v)$ such that $r_{v,i}^-=1$. The bottom line represents the elements of $\cS_n^+(v)$. The set $\cS_n^{(1)}(v)$, represented by the middle line, thus contains the red dot, all blue dots in $\cS_n^-(v)$ to the right of the red dot, and all orange dots in $\cS_n^+(v)$ to the left of the (position of the) red dot in $\cS_n^-(v)$. The degree  $\deg_n(v)$ of $v$ equals the number of blue dots in $\cS_n^{(1)}(v)$, and the label $\ell_n(v)$ of $v$ equals the step $i$ such that the dot in $\cS_n^{(1)}(v)$ is red.} \label{fig:couple}
\end{figure}

Equipped with this notation and the tools developed in the previous subsection, the aim of this section is to prove  Theorem~\ref{theorem:label_multiple_vertices_conditional_degrees} in the case of a single vertex. We (re)state this result here for completeness. 
\begin{proposition}\label{prop:depthlabelsinglevertex}
Fix $m\in\N$ and $b\in[0,m+1)$. Let $d=d(n)\in\N_0$ tend to infinity with $n$ such that $\lim_{n\to\infty}d/h_n^+=b$. Let $\mathbf x$ be a choice sequence that satisfies Assumption~$\xref$\ref{item:geq1},\ref{item:hI},\ref{item:lb} for some $\varepsilon<1-b/(m+1)$ in Part~\ref{item:hI} and where $I_{\xref}=o(\sqrt{h_n^+})$. Let $M,N$ be two independent standard normal random variables.
Then, conditionally on the event $\{\deg_n(a_1)\geq d(n)\}$
\be \ba 
\bigg({}&\frac{u_n(a_1)-(mh_n^+-\frac{m}{m+1}d(n))}{\sqrt{mh_n^+-\frac{m}{(m+1)^2}d(n)}}, \frac{h_{\ell_n(a_1)}^+-(h_n^+-\frac{1}{m+1}d(n))}{\sqrt{\frac{1}{(m+1)^2}d(n)}}\bigg)\\
&\toindis \bigg(M\sqrt{\frac{mb}{(m+1)^2-b}}+N\sqrt{1-\frac{mb}{(m+1)^2-b}},M\bigg).
\ea \ee 
\end{proposition}

To prove Proposition~\ref{prop:depthlabelsinglevertex} we need two main ingredients: $(1)$ We need to control the number of times vertex $a_1$ is selected before it loses its first dice roll. $(2)$ We need to finely control the number of times exactly one vertex in the selection set of $a_1$ is selected and loses a dice roll, for all steps after $v$ has lost its first dice roll.  As the step at which $a_1$ loses its first dice roll (which is its label $\ell_n(v)$) is random, these are non-trivial tasks. To this end, let us make the following observation. If we define, for $d,y\in\R$, 
\be \label{eq:Ldef}
L=L(d,y)\coloneq \sup\{\ell\in\N\colon (m+1)(h_n^+-h_{\ell-1}^+)\geq d-y\sqrt d\}, 
\ee 
then 
\be \label{eq:equiv}
\frac{h_{\ell_n(a_1)}^+-(h_n^+-\frac{1}{m+1}d(n))}{\sqrt{\frac{1}{(m+1)^2}d(n)}}>y \quad \Leftrightarrow \quad \ell_n(a_1)\geq L(d(n),y).
\ee 
In accordance with Proposition~\ref{prop:depthlabelsinglevertex}, we thus obtain that $\ell_n(a_1)$ should be close to $L(d,y)$ (with positive probability). So, we can instead approximate ingredients $(1)$ and $(2)$ by substituting $L(d(n),y)$ for $\ell_n(a_1)$. The proof of the proposition then comes down to applying $(1)$ and $(2)$ to an arbitrarily fine partition of $[L(d(n),y),n]$. In the remainder of the section, we omit the argument $(n)$ from $d(n)$ for ease of writing.

Ingredient (1) is summarised in the next lemma. 

\begin{lemma}\label{lemma:CLTSmin}
Consider the same notation and conditions as in Proposition~\ref{prop:depthlabelsinglevertex}. Additionally, recall $\cS_n^-(a_1)$ from~\eqref{eq:definition_selection_sets_ungreedy_depth_single_vertex} and let $\text{Geo}_m\sim \mathrm{Geo}(1/(m+1))$ be independent of $\cS_n^-(a_1)$. Fix $y\in \R$ and define the event 
\be 
\cE_n(d,y)\coloneq \{|\cS_n^-(a_1)\cap [L(d,y),n]|\geq d+\text{Geo}_m\}. 
\ee 
Then, 
\be 
\lim_{n\to\infty}\P{\cE_n(d,y)}=1-\Phi(y),
\ee
where $\Phi$ denotes the cumulative density function of the standard normal distribution.
\end{lemma}

\begin{proof}
We provide an upper and lower bound to the probability that have the same limit. Let $\eps>0$ be arbitrarily small and let $K=K(\eps)\in\N$ be large such that $\P{\text{Geo}_m\leq K}\geq 1-\eps$. We then have  
\begin{align}
\P{\cE_n(d,y)}&\leq \P{|\cS_n^-(a_1)\cap [L(d,y),n]|\geq d}
\intertext{and}
\P{\cE_n(d,y)}&\geq \P{|\cS_n^-(a_1)\cap [L(d,y),n]|\geq d +K}-\eps, 
\end{align}
where we use that $\text{Geo}_m$ is independent of $\cS_n^-(a_1)$ in the lower bound. We only prove that the probability in the upper bound has the desired limit. Proving that the probability in the lower bound has the same limit for any $K=K(\eps)\in\N$ fixed follows in an analogous way, from which the desired result follows, as $\eps$ is arbitrary. 

Since Assumption~$\xref$\ref{item:hI} is satisfied with $\eps<1-b/(m+1)$ and $\lim_{n\to\infty}d/h_n^+=b$, we can take $\xi\in (b,(m+1)(1-\eps))$ and use the definition of $L(d,y)$ in~\eqref{eq:Ldef} to see that for all $n$ large, 
\be\label{eq:I<L(d,y)}
(m+1)(h_n^+-h_{I_{\xref}}^+) \geq (m+1)(1-\eps)h_n^+ > \xi h_n^+ > d-y\sqrt d > (m+1)(h_n^+-h_{L(d,y)}^+).
\ee
It follows that $I_{\xref}<L(d,y)$ for all $n$ sufficiently large, as $h_i$ is non-decreasing in $i$. By using $A_i\geq m+1$ for $i\in\{I_{\xref},\ldots,n\}$ and recalling the definition of $\cS_n^-(a_1)$ in~\eqref{eq:definition_selection_sets_ungreedy_depth_single_vertex}, we obtain 
\begin{align}
\E{|\cS_n^-(a_1)\cap [L(d,y),n]|}&=\sum_{i=L(d,y)}^n \ind_{\{x_i=1\}} \frac{m+1}{A_i},
\intertext{and}
\Var(|\cS_n^-(a_1)\cap [L(d,y),n]|)&=\sum_{i=L(d,y)}^n \ind_{\{x_i=1\}} \frac{m+1}{A_i}\Big(1-\frac{m+1}{A_i}\Big).
\end{align} 
By Assumption~$\xref$\ref{item:lb} it follows that $\Var(|\cS_n^-(a_1)\cap [L(d,y),n]|)=\E{|\cS_n^-(a_1)\cap [L(d,y),n]|}+\cO(1)$. By the definition of $L(d,y)$ in~\eqref{eq:Ldef}, we have that 
\begin{align} 
\sum_{i=L(d,y)}^n \ind_{\{x_i=1\}} \frac{m+1}{A_i}&=(m+1)(h_n^+-h_{L(d,y)-1})\geq d-y\sqrt d
\intertext{and}
\sum_{i=L(d,y)}^n \ind_{\{x_i=1\}} \frac{m+1}{A_i}&=(m+1)(h_n^+-h_{L(d,y) })+\frac{m+1}{A_{L(d,y)}}<d-(y+o(1))\sqrt d,
\end{align}
where the final step in the upper bound follows from the fact that $d$ tends to infinity with $n$ and $(m+1)/A_{L(d,y)}$ is bounded. The Lindeberg central limit theorem thus yields that 
\be
\frac{|\cS_n^-(a_1)\cap [L(d,y),n]|-\E{|\cS_n^-(a_1)\cap [L(d,y),n]|}}{\sqrt{\Var(|\cS_n^-(a_1)\cap [L(d,y),n]|)}}\toindis N, 
\ee  
with $N$ a standard normal random variable. As a result, 
\be\ba 
\lim_{n\to\infty}{}&\mathbb P\big(|\cS_n^-(a_1)\cap [L(d,y),n]|\geq d\big)\\ 
={}&\lim_{n\to\infty}\mathbb P\bigg(\frac{|\cS_n^-(a_1)\cap [L(d,y),n]|-\E{|\cS_n^-(a_1)\cap [L(d,y),n]|}}{\sqrt{\Var(|\cS_n^-(a_1)\cap [L(d,y),n]|)}}\geq \frac{d-\E{|\cS_n^-(a_1)\cap [L(d,y),n]|}}{\sqrt{\Var(|\cS_n^-(a_1)\cap [L(d,y),n]|)}}\bigg)\\
={}&\P{N\geq y}=1-\Phi(y),
\ea \ee 
as desired.
\end{proof}

The following proposition deals with ingredient $(2)$. 

\begin{proposition}\label{prop:Ylim}
Consider the same notation and conditions as in Proposition~\ref{prop:depthlabelsinglevertex}. Additionally, recall $\cS_n^+(a_1)$ from~\eqref{eq:definition_selection_sets_ungreedy_depth_single_vertex}, fix $y\in \R$ and, conditionally on $\cS_n^+(a_1)$, define the random variable
\be 
Y_n(y)\sim \mathrm{Bin}\Big(|[L_n(d,y)-1]\cap \cS_n^+(a_1)|,\frac{1}{m+1}\Big). 
\ee 
Then, with $N$ a standard normal random variable, 
\be 
\frac{Y_n(y)-(mh_n^+-\frac{m}{m+1}d)}{\sqrt{mh_n^+-\frac{m}{(m+1)^2}d}}\toindis  N\sqrt{1-\frac{mb}{(m+1)^2-b}}+y\sqrt{\frac{mb}{(m+1)^2-b}}.
\ee
\end{proposition}

\begin{proof}
Recall that $(s_i^+)_{I_{\xref}\leq i\leq n}$ is a sequence of independent indicator random variables with 
\be 
\P{s_i^+=1}=p_i^+=\Big(\frac{m(m+1)}{A_i}+e_i\Big)\ind_{\{x_i=1\}}\ind_{\{A_i\geq 2m\}}, 
\ee  
with $e_i=\cO(A_i^{-2})$ and that $|\cS_n^+(a_1)|=\sum_{i=I}^n s_i^+$. Furthermore, if we let $(I_i^n)_{i\in[n],n\in\N}$ denote a sequence of i.i.d.\ Bernoulli random variables with parameter $(m+1)^{-1}$, we can define
\be \label{eq:Ynrewrite}
Q_n(y)\coloneq \sum_{i=I_{\xref}}^{L(d,y)-1}s_i^+\quad \text{and write}\quad  Y_n(y)=\sum_{j=1}^{Q_n(y)}I_j^{Q_n(y)}.
\ee 
Since $A_i\geq 2m$ for $i\geq I_{\xref}$ and by the fact that $I_{\xref}=o(\sqrt{h_n^+})$ and $I_{\xref}<L(d,y)$ by~\eqref{eq:I<L(d,y)}, we have
\be 
\E{Q_n(y)}=\!\!\sum_{i=I_{\xref}}^{L(d,y)-1}\! \!\Big(\frac{m(m+1)}{A_i}+e_i\Big)\ind_{\{x_i=1\}}\ind_{\{A_i\geq 2m\}}=\!\!\sum_{i=I_{\xref}}^{L(d,y)-1}\! \!\Big(\frac{m(m+1)}{A_i}+e_i\Big)\ind_{\{x_i=1\}}.
\ee 
By the definition of $L(d,y)$ and since $d=d(n)$ tends to infinity with $n$, we can write this as
\be \ba \label{eq:meanQ}
\E{Q_n(y)}&=m(m+1)(h_n^+-h_{I_{\xref}-1}^+)-m(m+1)(h_n^+-h_{L(d,y)-1})+\sum_{i=I_{\xref}}^{L(d,y)-1} \!\!\!\ind_{\{x_i=1\}}e_i\\
&=m(m+1)(h_n^+-h_{I_{\xref}-1}^+)-m(d-(y+o(1))\sqrt d) +\sum_{i=I_{\xref}}^{L(d,y)-1}\!\!\!\ind_{\{x_i=1\}}e_i.
\ea \ee 
Since $e_i=\cO(A_i^{-2})$, it follows from Assumption~$\xref$\ref{item:lb} that the sum on the right-hand side is $o(1)$, since $I_{\xref}$ tends to infinity with $n$. As $Q_n(y)$ is a sum of independent indicator random variables, similar calculations yield that 
\be \label{eq:varQ}
\Var(Q_n(y))=m(m+1)(h_n^+-h_{I_{\xref}-1}^+)-m(d-(y+o(1))\sqrt d. 
\ee 
As $\lim_{n\to\infty}d/h_n^+=b<m+1$ and $h_{I_{\xref}-1}^+\leq I_{\xref} =o(\sqrt{h_n^+})$, it follows that both the expected value and variance tend to infinity.  Applying the Lindeberg central limit theorem thus yields
\be \label{eq:CLTQ}
\frac{Q_n(y)-\E{Q_n(y)}}{\sqrt{\Var(Q_n(y))}}\toindis N', 
\ee 
with $N'$ a standard normal random variable. Furthermore, as the $(I_i^n)_{i\in[n],n\in\N}$ are i.i.d.\ Bernoulli random variables with success parameter $(m+1)^{-1}$, 
\be \label{eq:CLTind}
\frac{(m+1)\sum_{j=1}^n I_j^n -n}{\sqrt{m n}}\toindis N''
\ee 
with $N''$ a standard normal random variable that is independent of $N'$. We now write 
\be 
\frac{(m+1)\sum_{j=1}^{Q_n(y)}I_j^{Q_n(y)}-(m(m+1)h_n^+-md)}{\sqrt{m(m+1)^2h_n^+-md}}=B_n+C_n+D_n, 
\ee 
where
\be \ba 
B_n&\coloneq \frac{(m+1)\sum_{j=1}^{Q_n(y)}I_j^{Q_n(y)}-Q_n(y)}{\sqrt{mQ_n(y)}}\sqrt{\frac{Q_n(y)}{m(m+1)h_n^+-md}}\sqrt{\frac{m(m+1)h_n^+-md}{(m+1)^2 h_n^+-d}},\\
C_n&\coloneq \frac{Q_n(y)-\E{Q_n(y)}}{\sqrt{\Var(Q_n(y))}}\sqrt{\frac{\Var(Q_n(y))}{m(m+1)h_n^+-md}}\sqrt{\frac{(m+1)h_n^+-d}{(m+1)^2h_n^+-d}}, \\
D_n&\coloneq \frac{\E{Q_n(y)}-(m(m+1)h_n^+-md)}{\sqrt{d}}\sqrt{\frac{d}{m(m+1)^2h_n^+-md}}.
\ea\ee 
As $d/h_n^+\to b$, it is clear from~\eqref{eq:meanQ} that 
\be \label{eq:Cnconv}
\lim_{n\to\infty}D_n= y \sqrt{\frac{mb}{(m+1)^2-b}}.
\ee
By Skorokhod's representation theorem (see Theorem $6.7$ in \cite{Billingsley.1999}), there exists a probability space and an independent coupling of $(Q_n(y))_{n\in\N}$ and $(I_j^n)_{j\in[n],n\in\N}$ such that the convergence in~\eqref{eq:CLTQ} and~\eqref{eq:CLTind} is almost sure rather than in distribution.  In particular, we have 
\be 
\frac{Q_n(y)}{m(m+1)h_n^+-md}\toas1\qquad\text{and}\qquad Q_n(y)\toas\infty. 
\ee 
We can use this to determine that, in this probability space, 
\be 
B_n\toas N'' \sqrt{\frac{m(m+1)-mb}{(m+1)^2-b}}\qquad \text{and}\qquad C_n\toas N'\sqrt{\frac{(m+1)-b}{(m+1)^2-b}}. 
\ee 
Combined with the convergence of $D_n$ in~\eqref{eq:Cnconv}, we thus obtain that 
\be \ba 
\frac{(m+1)\sum_{j=1}^{Q_n(y)}I_j^{Q_n(y)}-(m(m+1)h_n^+-md)}{\sqrt{m(m+1)^2h_n^+-md}}\toindis{}&  N'' \sqrt{\frac{m(m+1)-mb}{(m+1)^2-b}}+ N'\sqrt{\frac{(m+1)-b}{(m+1)^2-b}}\\
&+y \sqrt{\frac{mb}{(m+1)^2-b}}.
\ea \ee 
The  independence of $N'$ and $N''$ and~\eqref{eq:Ynrewrite} thus conclude the proof.
\end{proof}

With both Lemma~\ref{lemma:CLTSmin} and Proposition~\ref{prop:Ylim} at hand, we are now ready to prove Proposition~\ref{prop:depthlabelsinglevertex}.

\begin{proof}[Proof of Proposition~\ref{prop:depthlabelsinglevertex}]
Fix $y,z\in\R$, recall $L(d,y)$  from~\eqref{eq:Ldef}, and set $\cS_n^-\coloneq \cS_n^-(a_1)$, $\cS_n^+\coloneq \cS_n^+(a_1)$, $u_n\coloneq u_n(a_1)$, $\ell_n \coloneq \ell_n(a_1)$, and $\ell\coloneq L(d,y)$ for ease of writing. Set 
\be \label{eq:u}
u=u(d,n,z)\coloneq mh_n^+-\frac{m}{m+1}d+z \sqrt{mh_n^+-\frac{m}{(m+1)^2}d}. 
\ee 
By the equivalence in~\eqref{eq:equiv} and with $M$ and $N$ i.i.d.\ standard normal random variables, we are required to prove that
\begin{equation}
\label{eq:joint_behaviour_ungreedy_label_conditional_degree}
\begin{aligned}
&\lim_{n\to\infty}\prob(u_n\leq u,\ell_n\geq\ell\,|\,\deg_n(a_1)\geq d)\\
&= \prob\left(M\sqrt{\frac{mb}{(m+1)^2-b}}+N\sqrt{1-\frac{mb}{(m+1)^2-b}}\leq z,M>y\right).
\end{aligned}
\end{equation}
We divide the proof of~\eqref{eq:joint_behaviour_ungreedy_label_conditional_degree} into an upper bound and a lower bound. We prove the upper bound first, and can then recover most of the steps for the lower bound.

\textbf{Upper bound. } Recall $  \wt u_n(a_1)$ from~\eqref{eq:wtun}. By an application of the first part of Lemma~\ref{lemma:ungreedy_depth_only_one_vertex_connections_set},
\be\ba \label{eq:begin_upper_bound_theorem_ungreedy_depth}
\prob(u_n\leq u,\ell_n\geq\ell\,|\,\deg_n(a_1)\geq d)& \leq \prob(\widetilde{u}_n\leq u,\ell_n\geq\ell\,|\,\deg_n(a_1)\geq d)\\
&= \frac{\prob(\widetilde{u}_n\leq u,\ell_n\geq\ell,\deg_n(a_1)\geq d)}{\prob(\deg_n(a_1)\geq d)}.
\ea\ee
Recall from the definition of $\cS_n^-$ at the start of this section that we use $\cS_n^-$ to determine the steps at which vertex $a_1$ is selected, prior to losing its first dice roll. The event $\{\deg_n(a_1)\geq d\}$ is equivalent to $\cS_n^-$ containing at least $d$ elements and $a_1$ winning the first $d$ dice rolls  associated with elements in  $\cS_n^-$. Let $\text{Geo}_m$ be a geometric random variable with parameter $(m+1)^{-1}$ independent of everything else. After these $d$ wins of vertex $a_1$, we use $\text{Geo}_m$ to count the number of dice rolls that $a_1$ participates in, until $a_1$ loses for the first time. Since $A_i\geq m+1$ for all $i\in\{I_{\xref},\ldots,n\}$ by Assumption~$\xref$\ref{item:lb} and $L(d,y)\geq I_{\xref}$ by Assumption~$\xref$\ref{item:hI} (see~\eqref{eq:I<L(d,y)}) and we work on the event $\{\ell_n(a_1)\geq \ell\}=\{\ell_n(a_1)\geq L(d,y)\}$, it follows that $\text{Geo}_m$ indeed has the correct distribution.  Now, the event $\{\deg_n(a_1)\geq d ,\ell_n\geq \ell\}$ is equivalent to $\cS_n^-\cap [\ell,n]$ containing at least $d+\text{Geo}_m$ many elements and $a_1$ winning the first $d$ dice rolls associated with the elements in $\cS_n^-$. After $a_1$ has lost its first dice roll, it increases its depth in $G_{\ell_n}$ by one and is connected to $m$ roots by directed edges. For each step $i\geq I_{\xref}$ such that exactly one of these $m$ roots is selected and loses the corresponding dice roll (and connects itself to $m$ new roots by directed edges), we add to the value of $\wt u_n$. The number of times exactly one of these $m$ roots is selected is given by $\cS_n^+\cap [I_{\xref},\ell_n-1]$. On the event $\{\deg_n(a_1)\geq d,\ell_n\geq \ell\}$ we define the set 
\be 
\cR_n\coloneq\{i\in [\ell,n]\colon|\cS^-_n\cap[i,n]|\geq d+\text{Geo}_m\}\cap \cS_n^+
\ee 
and partition $\cS_n^+\cap [I_{\xref},\ell_n-1]$ into 
\be 
\cS_n^+\cap [I_{\xref},\ell_n-1]=\big(\cS_n^+\cap [I_{\xref},\ell)\big)\cup (\cR_n\setminus\{\ell_n\}).
\ee 
We let $X_n$ and $Y_n$ denote the number of losses of dice rolls associated to the steps in $S_n^+\cap [I_{\xref},\ell)$ and $\cR_n\setminus\{\ell_n\}$, respectively, so that $  \wt u_n= X_n+Y_n$. Conditionally on $\cS_n^-,\cS_n^+$, and $\text{Geo}_m$, the random variables $X_n$ and $Y_n$ are independent and distributed as 
\begin{equation}
X_{n}\sim\text{Bin}\left(|\cR_n|-1,\frac{1}{m+1}\right),\quad Y_n\sim\text{Bin}\left(|\cS^+_n\cap[I_{\xref},\ell-1]|,\frac{1}{m+1}\right).
\end{equation}
By using the tower property and the fact that $\{|\cS^-_n\cap[\ell,n]|\geq d+\text{Geo}_m\}$ is measurable with respect to $\cS_n^-$ and $\text{Geo}_m$ yields
\be\ba \label{eq:properties_vertices_condition_tower_property}
\prob({}&\deg_n(a_1)\geq d,\ell_n\geq\ell,\wt u_n\leq u)\\
&=\E{\prob\left(\deg_n(a_1)\geq d, |\cS^-_n\cap[\ell,n]|\geq d+\text{Geo}_m,X_n+Y_n\leq u\,\big|\,\cS_n^-,\cS_n^+,\text{Geo}_m\right)}\\
&=\E{\ind_{\{|\cS^-_n\cap[\ell,n]|\geq d+\text{Geo}_m\}}\prob\left(\deg_n(a_1)\geq d,X_n+Y_n\leq u\,\big|\,\cS_n^-,\cS_n^+,\text{Geo}_m\right)}.
\ea\ee
Conditionally on $\cS_n^-,\cS_n^+$, and $\text{Geo}_m$, and on the event $\{|\cS_n^-\cap [\ell,n]\geq d+\text{Geo}_m\}$, the events $\{\deg_n(a_1)\geq d\}$ and $\{X_n+Y_n\leq u\}$ are independent, since they depend on the outcomes of distinct (and therefore independent) dice rolls. Furthermore, the event $\{\deg_n(a_1)\geq d\}$ has probability $\theta^{-d}$ on the event $\{|\cS_n^-\cap [\ell,n]\geq d+\text{Geo}_m\}$, since $A_i\geq m+1$ for $i\geq \ell=L(d,y)>I_{\xref}$ by Assumption~$\xref$\ref{item:hI},\ref{item:lb} and \eqref{eq:I<L(d,y)}. Hence, we obtain 
\be 
\prob(\deg_n(a_1)\geq d,\ell_n\geq\ell,\wt u_n\leq u)=\theta^{-d}\E{\ind_{\{|\cS^-_n\cap[\ell,n]|\geq d+\text{Geo}_m\}}\prob\left(X_n+Y_n\leq u\,\big|\,\cS_n^-,\cS_n^+,\text{Geo}_m\right)}.
\ee 
By Theorem~\ref{theorem:joint_degree_distribution} we have that $\P{\deg_n(a_1)\geq d}=\theta^{-d}(1+o(1))$. Using the above in~\eqref{eq:begin_upper_bound_theorem_ungreedy_depth} thus yields
\be \ba\label{eq:ubo(1)}
\mathbb P(u_n\leq u,\ell_n\geq \ell\,|\, \deg_n(a_1)\geq d)\leq{}& \mathbb E\big[\ind_{\{|\cS^-_n\cap[\ell,n]|\geq d+\text{Geo}_m\}}\prob\left(X_n+Y_n\leq u\,|\,\cS_n^-,\cS_n^+,\text{Geo}_m\right)\big]\\
&+o(1).
\ea\ee 
To obtain the correct upper bound, it thus remains to show that 
\be \ba 
\limsup_{n\to\infty}{}& \E{\ind_{\{|\cS^-_n\cap[\ell,n]|\geq d+\text{Geo}_m\}}\prob\left(X_n+Y_n\leq u\,\big|\,\cS_n^-,\cS_n^+,\text{Geo}_m\right)}\\
&=\prob\left(M\sqrt{\frac{ma}{(m+1)^2-a}}+N\sqrt{1-\frac{ma}{(m+1)^2-a}}\leq y,M>x\right).
\ea \ee 
To this end, we recall the definition of $L(d,x)$ from~\eqref{eq:Ldef}, where $x\in\R$ is a fixed constant, and define the events 
\be 
\cE_n(x)\coloneq \{|\cS_n^-\cap [L(d,x),n]|\geq d+\text{Geo}_m\}. 
\ee 
We then fix $\eps>0$ arbitrarily small and let $\delta=\delta(\eps)>0$ to be determined and  $C=C(\delta,\eps,y)>0$ be a large constant such that $1-\Phi(y+C)<\eps/2$ and $C/\delta\in\N$. Then, since $\ell=L(d,y)$ and $L(d,x)$ is non-decreasing in $x$, we can partition
\be \label{eq:Epart}
\{|\cS_n^-\cap [\ell,n]|\geq d+\text{Geo}_m\}=\cE_n(y)=\bigg(\bigcup_{j=0}^{C/\delta-1}\cE_n(y+j\delta)\cap \cE_n(y+(j+1)\delta)^c\bigg)\cup\cE_n(y+C).
\ee 
This allows us to write 
\be \ba\label{eq:splitsum}
\mathbb E\Big[{}&\ind_{\{|\cS^-_n\cap[\ell,n]|\geq d+\text{Geo}_m\}}\prob\left(X_n+Y_n\leq u\,\big|\,\cS_n^-,\cS_n^+,\text{Geo}_m\right)\Big]\\
&=\sum_{j=0}^{C/\delta-1} \E{\ind_{\cE_n(y+j\delta)\cap \cE_n(y+(j+1)\delta)^c}\prob\left(X_n+Y_n\leq u\,\big|\,\cS_n^-,\cS_n^+,\text{Geo}_m\right)}\\
&\hphantom{=}\,+\E{\ind_{\cE_n(y+C)}\prob\left(X_n+Y_n\leq u\,|\,\cS_n^-,\cS_n^+,\text{Geo}_m\right)}.
\ea\ee 
For $j\in\{0,1\ldots,C/\delta\}$, and conditionally on $\cS_n^+$, define the independent random variables  
\begin{equation}
\label{eq:definition_X_nj_and_Y_nj}
X_{n,j}\sim\text{Bin}\Big(|[\ell,L(d,y+j\delta))\cap\cS_n^+|,\frac{1}{m+1}\Big),\quad Y_{n,j}\sim\text{Bin}\Big(|[I_{\xref},L(d,y+j\delta))\cap\cS_n^+|,\frac{1}{m+1}\Big),
\end{equation}
On the event $\cE_n(y+j\delta)\cap \cE_n(y+(j+1)\delta)^c$, it follows that 
\be 
[\ell,L(d,y+j\delta)) \subseteq\{i\in[\ell,n]\colon |\cS_n^-\cap [i,n]|\geq d+\text{Geo}_m\}\setminus\{\ell_n\}, 
\ee 
so that 
\be 
[\ell,L(d,y+j\delta))\cap \cS_n^+\subseteq \cR_n\setminus\{\ell_n\}.
\ee 
As a result, $X_n$ stochastically dominates $X_{n,j}$ on the event $\cE_n(y+j\delta)\cap \cE_n(y+(j+1)\delta)^c$, and  we can bound 
\be \ba 
\ind_{\cE_n(y+j\delta)\cap \cE_n(y+(j+1)\delta)^c}{}&\prob\left(X_n+Y_n\leq u\,\big|\,\cS_n^-,\cS_n^+,\text{Geo}_m\right)\\
\leq{}& \ind_{\cE_n(y+j\delta)\cap \cE_n(y+(j+1)\delta)^c}\prob\left(X_{n,j}+Y_n\leq  u\,\big|\,\cS_n^-,\cS_n^+,\text{Geo}_m\right)\\
={}&\ind_{\cE_n(y+j\delta)\cap \cE_n(y+(j+1)\delta)^c}\prob\left(Y_{n,j}\leq  u\,\big|\,\cS_n^-,\cS_n^+,\text{Geo}_m\right).
\ea \ee 
We can now omit $\text{Geo}_m$ and $\cS_n^-$ from the conditioning, as $Y_{n,j}$ does not depend on $\text{Geo}_m$ and $\cS_n^-$. Furthermore, the event $\cE_n(y+j\delta)\cap \cE_n(y+(j+1)\delta))^c$ is measurable with respect to $\cS_n^-$ and therefore independent of $ \cS_n^+$. As a result, we obtain 
\be \ba 
\mathbb E\big[{}&\ind_{\cE_n(y+j\delta)\cap \cE_n(y+(j+1)\delta)^c}\prob\left(X_n+Y_n\leq u \,\big|\,\cS_n^-,\cS_n^+,\text{Geo}_m\right)\big]\\
&\leq \P{\cE_n(y+j\delta)\cap \cE_n(y+(j+1)\delta)^c}\P{Y_{n,j}\leq u }.
\ea\ee 
Using this in~\eqref{eq:splitsum}, we thus arrive at 
\be\ba \label{eq:limsupexp} 
\mathbb E\Big[{}&\ind_{\{|\cS^-_n\cap[\ell,n]|\geq d+\text{Geo}_m\}}\prob\left(X_n+Y_n\leq u\,\big|\,\cS_n^-,\cS_n^+,\text{Geo}_m\right)\Big]\\
&\leq \sum_{j=0}^{C/\delta-1} \P{\cE_n(y+j\delta)\cap \cE_n(y+(j+1)\delta)^c}\P{Y_{n,j}\leq u}+\P{\cE_n(y+C)}.
\ea \ee 
We then recall the definition of $Y_{n,j}$ from~\eqref{eq:definition_X_nj_and_Y_nj} and apply Proposition~\ref{prop:Ylim} to obtain for any $j\in\{0,1,\ldots, C/\delta-1\}$, 
\be \label{eq:Ynjlim}
\lim_{n\to\infty}\P{Y_{n,j}\leq u}=\P{N\sqrt{1-c_{b,m}}\leq z-(y+j\delta)\sqrt{c_{b,m}}},
\ee 
with $N$ a standard normal random variable and $c_{b,m}\coloneq mb/((m+1)^2-mb)$.  We also apply Lemma~\ref{lemma:CLTSmin} to obtain 
\be \label{eq:Enlim}
\lim_{n\to\infty}\P{\cE_n(y+j\delta)\cap \cE_n(y+(j+1)\delta)^c}= \P{M\in (y+j\delta, y+(j+1)\delta)}, 
\ee 
and $\lim_{n\to\infty}\P{\cE_n(y+C)}=1-\Phi(y+C)$, with $M$ a standard normal random variable (independent of $N$). By the choice of $C$ we have $1-\Phi(y+C)<\eps/2$. Combining both limits in~\eqref{eq:limsupexp} finally yields 
\be\ba 
\limsup_{n\to\infty}\mathbb E{}&\Big[\ind_{\{|\cS^-_n\cap[\ell,n]|\geq d+\text{Geo}_m\}}\prob\left(X_n+Y_n\leq u\,\big|\,\cS_n^-,\cS_n^+,\text{Geo}_m\right)\Big]\\
\leq \sum_{j=0}^{C/\delta-1}{}&\P{M\in (y+j\delta,y+(j+1)\delta),N\sqrt{1-c_{b,m}}\leq z-(y+j\delta)\sqrt{c_{b,m}}} +\eps/2\\
\leq  \sum_{j=0}^{C/\delta-1} {}&\P{M\in(y+j\delta,y+(j+1)\delta),N\sqrt{1-c_{b,m}}+M\sqrt{c_{b,m}}\leq z+\delta\sqrt{c_{b,m}}}+\eps/2.
\ea\ee  
As only the first event in the probabilities in the sum depends on $j$, we obtain the upper bound
\be 
\P{M>y, N\sqrt{1-c_{b,m}}+M\sqrt{c_{b,m}}\leq z+\delta\sqrt{c_{b,m}}}+\eps/2
\ee 
By the continuity of the distribution of $M$ and $N$, we can take $\delta=\delta(\eps)$ small enough such that 
\be \ba 
\mathbb P{}&(M>y, N\sqrt{1-c_{b,m}}+M\sqrt{c_{b,m}}\leq z+\delta\sqrt{c_{b,m}})+\eps/2\\
&\leq  \P{M>y, N\sqrt{1-c_{b,m}}+M\sqrt{c_{b,m}}\leq z}+\eps.
\ea \ee 
Since $\eps$ is arbitrary, this yields the desired upper bound for the limsup and concludes the proof for the upper bound, as going back to~\eqref{eq:begin_upper_bound_theorem_ungreedy_depth} yields 
\be \label{eq:limsupub}
\limsup_{n\to\infty} \P{u_n\leq u,\ell_n\geq \ell\,|\, \deg_n(a_1)\geq d}\leq  \P{M>y, N\sqrt{1-c_{b,m}}+M\sqrt{c_{b,m}}\leq z}.
\ee

\textbf{Lower bound. }We reuse the notation introduced for the upper bound. Take $(c_n)_{n\in\N}$ to be some   sequence diverging to infinity, such that $c_n=o(\sqrt{h_n^+})$ and $c_n=\omega(I_{\xref})$ as $n\to\infty$ (note that this is possible, since we assume that  $I_{\xref}=o(\sqrt{h_n^+})$). By applying the second part of Lemma~\ref{lemma:ungreedy_depth_only_one_vertex_connections_set} and as in~\eqref{eq:begin_upper_bound_theorem_ungreedy_depth}, we obtain 
\be 
\P{u_n\leq u, \ell_n\geq \ell\,|\, \deg_n(a_1)\geq d}\geq \frac{\P{\wt u_n\leq u-c_n , \ell_n\geq \ell, \deg_n(a_1)\geq d}}{\P{\deg_n(a_1)\geq d}}.
\ee 
Using the same definitions and notation as in the lower bound, we then obtain, as in~\eqref{eq:ubo(1)}, 
\be \ba 
\mathbb P(u_n{}&\leq u, \ell_n\geq \ell\,|\, \deg_n(a_1)\geq d)\\
&\geq \E{\ind_{\{\cS_n^-\cap [\ell,n]\geq d+\text{Geo}_m\}}\P{X_n+Y_n\leq u-c_n\,\big|\, \cS_n^-,\cS_n^+,\text{Geo}_m}}+o(1).
\ea \ee 
We partition the event $\{|\cS_n^-\cap [\ell,n]|\geq d+\text{Geo}_m\}=\cE_n(y)$ as in~\eqref{eq:Epart} to write the expected value on the right-hand side as a sum, similar to~\eqref{eq:splitsum} (but omitting the last term). This yields the lower bound 
\be 
\sum_{j=0}^{C/\delta-1}\E{\ind_{\{\cE_n(y+j\delta)\cap \cE_n(y+(j+1)\delta)^c\}}\P{X_n+Y_n\leq u-c_n \,\big|\, \cS_n^-,\cS_n^+,\text{Geo}_m}}+o(1).
\ee 
Recall $X_{n,j}$ and $Y_{n,j}$ for $j\in\{0,\ldots, C/\delta\}$ from~\eqref{eq:definition_X_nj_and_Y_nj}. We now use that, on the event $\cE_n(y+j\delta)\cap \cE_n(y+(j+1)\delta)^c$, the random variable $X_n$ is stochastically dominated by $X_{n,j+1}$. As a result, we can bound each term in the sum from below by 
\be \ba 
\mathbb E\big[{}&\ind_{\{\cE_n(y+j\delta)\cap \cE_n(y+(j+1)\delta)^c\}}\P{X_n+Y_n\leq u-c_n \,\big|\, \cS_n^-,\cS_n^+,\text{Geo}_m}\big]\\
&\geq \E{\ind_{\{\cE_n(y+j\delta)\cap \cE_n(y+(j+1)\delta)^c\}}\P{X_{n,j+1}+Y_n\leq u-c_n\,\big|\, \cS_n^-,\cS_n^+,\text{Geo}_m}}\\
&=\P{\cE_n(y+j\delta)\cap \cE_n(y+(j+1)\delta)^c}\P{Y_{n,j+1}\leq u-c_n },
\ea \ee 
where we use, as in the upper bound, that $X_{n,j+1}+Y_n\overset \dd= Y_{n,j+1}$, and that $Y_{n,j+1}$ depends only on $\cS_n^+$ and the event $\cE_n(y+j\delta)\cap \cE_n(y+(j+1)\delta)^c$ depends only on $\cS_n^-$ and are thus independent. As   $d/h_n^+\to b\in [0,m+1)$ and $c_n=o(\sqrt{h_n^+})$, it follows from the definition of $u$ in~\eqref{eq:u} that
\be 
u-c_n=mh_n^+-\frac{m}{m+1}d+(z+o(1))\sqrt{mh_n^+-\frac{m}{(m+1)^2}d}. 
\ee 
We can thus apply Proposition~\ref{prop:Ylim} and Lemma~\ref{lemma:CLTSmin}, as we do in~\eqref{eq:Ynjlim} and~\eqref{eq:Enlim}, respectively, to analogously obtain
\be \ba 
\liminf_{n\to\infty}{}&\P{u_n\geq u,\ell_n\geq \ell\,|\, \deg_n(a_1)\geq d}\\
\geq{}& \sum_{j=0}^{C/\delta-1} \P{M\in(y+j\delta,y+(j+1)\delta),N\sqrt{1-c_{b,m}}\leq z-(y+(j+1)\delta)\sqrt{c_{b,m}}}\\
\geq{}&\P{M>y,N\sqrt{1-c_{b,m}}+M\sqrt{c_{b,m}}\leq z-\delta\sqrt{c_{b,m}}}-\P{M\geq y+ C}.
\ea \ee 
Again, by the choice of $C$ the final term is at most $\eps/2$. We can then choose $\delta=\delta(\eps)$ small enough such to obtain the lower bound
\be 
\liminf_{n\to\infty}\P{u_n\geq u,\ell_n\geq \ell\,|\, \deg_n(a_1)\geq d}\geq \P{M>y, N\sqrt{1-c_{b,m}}+M\sqrt{c_{b,m}}\leq z}-\eps.
\ee 
As $\eps$ is arbitrary, this yields a matching lower bound to the upper bound we established in~\eqref{eq:limsupub}, so that we arrive at~\eqref{eq:joint_behaviour_ungreedy_label_conditional_degree} and thus conclude the proof.
\end{proof}

\section{Properties of multiple vertices with given large degrees}\label{sec:multiplevertices}

The objective of this section is to extend Proposition~\ref{prop:depthlabelsinglevertex} to the setting of multiple active vertices, as in Theorem~\ref{theorem:label_multiple_vertices_conditional_degrees}. This is section is therefore mainly of a technical nature, where we first gather introduce some preliminary results on the selection sets of multiple vertices in Section~\ref{subsec:selection_sets_multiple_vertices}  to deal with the dependencies of the labels and depths of the active vertices $a_1,\ldots, a_k$. We then prove Theorem~\ref{theorem:label_multiple_vertices_conditional_degrees} in Section~\ref{subsec:labels_multiple_vertices}.

\subsection{Selection sets of multiple vertices}
\label{subsec:selection_sets_multiple_vertices}

The approach to extending Proposition~\ref{prop:depthlabelsinglevertex} to Theorem~\ref{theorem:label_multiple_vertices_conditional_degrees} is to `decouple' the active vertices $a_1,\ldots, a_k$. That is, if we can show that 
\be \ba \label{eq:asympindep}
\mathbb P{}&(u_n(a_v)\leq u_v, \ell_n(a_v)\geq \ell_v, d_n(a_v)\geq d_v\text{ for all }v\in[k])\\
&= (1+o(1))\prod_{v=1}^k \P{u_n(a_v)\leq u_v, \ell_n(a_v)\geq \ell_v, d_n(a_v)\geq d_v}, 
\ea \ee 
for an appropriate choice of integers $u_v=u_v(n),\ell_v=\ell_v(n)$, and $d_v=d_v(n)$, then (the proof of) Proposition~\ref{prop:depthlabelsinglevertex} yields the desired result. 

This decoupling follows a similar approach as in Section~\ref{sec:degs}, where, among others, we proved Theorem~\ref{theorem:joint_degree_distribution} regarding degrees of typical vertices. There, we showed that  the active root vertices $a_1,\ldots, a_k$ are not selected at the same step $i$ for all large $i$. That is, the sets $\cS_n(a_v)\cap [i,n]$ with $v\in[k]$ are disjoint with high probability when $i=i(n)$ tends to infinity with $n$. This allowed us to argue that the   degrees of $a_1,\ldots,a_k$ in the Kingman coalescent are asymptotically independent. Here, we use a similar notion, but for the connection sets $(\cC_n^{(i)}(a_v))_{i\in[n], v\in[k]}$ instead. 

We recall that the connection set $\cC_n^{(i)}(a_v)$ of a vertex $a_v$ satisfies that $\cC_n^{(i)}(a_v)=\{a_v\}$ for all steps $i\in\{\ell_n(a_v),\ldots, n\}$. That is, until we reach the step $i$ when $a_v$ loses its first dice roll (and determines its label $\ell_n(a_v)$), the connection set of $a_v$ contains only $a_v$. Afterwards, in steps $i\in\{I_{\xref},\ldots, \ell_n(a_v)-1\}$, where we recall $I_{\xref}$ from Assumption~$\xref$, the connection set $\cC_n^{(i)}(a_v)$ contains $m$ vertices. Here, we use Assumption~$\xref$\ref{item:lb} and that $n$ is large enough so that $A_i\geq m+1$. Which $m$ vertices are in $\cC_n^{(i)}(a_v)$ can change throughout these steps, but there will always be exactly $m$ in the set (when $i\in\{I_{\xref},\ldots, \ell_n(a_v)-1\}$). In Section~\ref{sec:singlevertex} we used this to couple $\cS_n^{(1)}(a_v)$ to $\cS_n^-(a_v)$ and $\cS_n^+(a_v)$. Here, we instead define for $v\in[k]$,
\be 
\cS_n^{\geq 1}(a_v)\coloneq \{i\in \{I_{\xref},\ldots, n\}\colon \text{At least one root in $\cC_n^{(i)}(a_v)$ is selected}\}.
\ee 
This is slightly different from $\cS_n^{(1)}(a_v)$, which is the set of all steps at which \emph{exactly} one root in $\cC_n^{(i)}(a_v)$ is selected. As it turns out, to `decouple' the active vertices $a_1,\ldots, a_k$, it is more convenient to work with $\cS_n^{\geq1} (a_v)$ rather than with $\cS_n^{(1)}(a_v)$.  We change the definition of  $\cS_n^+(a_v)$ accordingly and then couple $\cS_n^-(a_v)$ and $\cS_n^+(a_v)$ to $\cS_n^{\geq 1}(a_v)$. For $i\in\{I_{\xref},\ldots,n\}$, we set 
\be
p_i^-\coloneq \frac{\binom{A_i-1}{m}}{\binom{A_i}{m+1}}\ind_{\{x_i=1\}}=\frac{m+1}{A_i}\ind_{\{x_i=1\}}\ind_{\{A_i\geq m+1\}}, 
\ee
and
\be \label{eq:definition_and_order_of_p_v_i^+}
p_i^+\coloneq \sum_{j=1}^{m}\frac{\binom{m}{j}\binom{A_i-m}{m-(j-1)}}{\binom{A_i}{m+1}}\ind_{\{x_i=1\}}=\Big(\frac{m(m+1)}{A_i}\ind_{\{A_i\geq 2m\}}+\wt e_i\Big)\ind_{\{x_i=1\}},
\ee 
with $\wt e_i=\cO(A_i^{-2})$ as $i\to\infty$. We see that $p_i^-$ as defined here is the same as in~\eqref{eq:def_p_i^-}, whereas $p_i^+$ differs $\cO(A_i^{-2})$ from the definition in~\eqref{eq:def_pi_+}. Then, we let $(s_{v,i}^-)_{v\in[k],I_{\xref}\leq i\leq n}$ and $(s_{v,i}^+)_{v\in[k],I_{\xref}\leq i\leq n}$ be sequences of random variables, where $s_{v,i}^-\sim \text{Ber}(p_i^-)$ and $s_{v,i}^+\sim \text{Ber}(p_i^+)$ for each $v\in[k]$ and $i\in\{I_{\xref},\ldots,n\}$, and define 
\be 
\cS_n^-(a_v)\coloneq \{i\in \{I_{\xref},\ldots, n\}\colon s_{v,i}^-=1\}\quad \text{and}\quad \cS_n^+(a_v)\coloneq \{i\in\{I_{\xref},\ldots, n\}\colon s_{v,i}^+=1\}.
\ee 
Here, $(1)$: $(s_{v,i}^-)_{v\in[k],I_{\xref}\leq i\leq n}$ is independent of $(s_{v,i}^+)_{v\in[k],I_{\xref}\leq i\leq n}$, $(2)$: For a fixed $v\in[k]$ and $\square\in\{-,+\}$, the random variables $(s^\square_{v,i})_{I_{\xref}\leq i\leq n}$ are mutually independent, and $(3)$: For $i\neq j$ and $\square\in\{-,+\}$, the random variables $(s_{v,i}^\square)_{v\in [k]}$ are independent of $(s_{v,j}^\square)_{v\in [k]}$. We intuitively think of $\cS_n^-(a_v)$ and $\cS^+_n(a_v)$ as the set of all steps $i$ at which at least one vertex in $\cC_n^{(i)}(a_v)$ is selected before and after $a_v$ has lost its first dice roll, respectively. Since $\cC_n^{(i)}(a_v)=\{a_v\}$ for all steps $i\in \{\ell_n(a_v),\ldots, n\}$, selecting at least one and exactly one element from $\cC_n^{(i)}(a_v)$ is equivalent. Hence, the definition of $\cS_n^-(a_v)$ here and in~\eqref{eq:definition_selection_sets_ungreedy_depth_single_vertex} is the same. However, selecting at least one and exactly one element from $\cC_n^{(i)}(a_v)$ when this set contains $m$ elements are not the same, hence why the definition of $\cS_n^+(a_v)$ here differs from~\eqref{eq:definition_selection_sets_ungreedy_depth_single_vertex}.

Let us now couple $\cS_n^-(a_v)$ and $\cS_n^+(a_v)$ to $\cS_n^{\geq 1}(a_v)$. This uses the same idea as in Section~\ref{sec:singlevertex}. For each $v\in[k]$ and $i\in\cS_n^-(a_v)$, let $r_{v,i}^-\sim \text{Ber}(1/(m+1))$. For each $v\in[k]$ fixed, the $(r_{v,i}^-)_{i\in\cS_n^-(a_v)}$ are mutually independent and are also independent of the sets $\cS_n^-(a_v)$ and $\cS_n^+(a_v)$. Here,  $r_{v,i}^-=1$ corresponds to $v$ losing a dice roll at step $i$. Then, we for all $i\in\{I_{\xref},\ldots, n\}$, we have 
\be 
i\in \cS_n^{\geq 1}(a_v) \text{ when}\begin{cases} 
i\in \cS_n^-(a_v) \text{ and } r_{v,j}^-=0\text{ for all }j\in [i+1,n]\cap \cS_n^-(a_v),\text{ or} \\
i\in \cS_n^+(a_v) \text{ and } r_{v,j}^-=1\text{ for some }j\in [i+1,n]\cap \cS_n^-(v).
\end{cases} 
\ee 
Furthermore, we set $r_{v,i}=r_{v,i}^-$ for every $i\in \cS_n^{\geq 1}(a_v)\cap \cS_n^-(a_v)$, where we recall the random variables $(r_{v,i})_{i\in \cS_n(a_v)}$ from Section~\ref{subsec:selection_and_connection_sets}. This also immediately extends the coupling, in the sense that we have now coupled $\deg_n(a_v)$ and $\ell_n(a_v)$ to $\cS_n^-(a_v),\cS_n^+(a_v)$, and $(r_{v,i}^-)_{i\in \cS_n^-(a_v)}$ as well. This latter part holds only if $\ell_n(a_v)\geq I_{\xref}$, otherwise we say the coupling \emph{fails}. As we will see, the probability of the coupling failing tends to zero with $n$.

We have not (yet) specified the correlations of the random variables $s^\square_{1,i},\ldots,s^\square_{k,i}$ for $\square\in\{-,+\}$ and the random variables $r^-_{1,i},\ldots, r^-_{k,i}$, and  we have only defined the marginals of $\cS_n^-(a_v)$ and $\cS_n^+(a_v)$ for each $v\in [k]$. In Lemma~\ref{lemma:replace_selection_sets_by_independent_copies}, we shall see that these correlations are sufficiently weak so that we  obtain the asymptotic independence, as in~\eqref{eq:asympindep}, required to prove Theorem~\ref{theorem:label_multiple_vertices_conditional_degrees}. 

Recall the random variable $ \tau_k\coloneq\tau_{k,0}$ from~\eqref{eq:definition_tau_k}. By the definition of $\tau_k$,  the selection sets $\cS_n(a_1),\ldots, \cS_n(a_k)$ of the vertices $a_1,\ldots,a_k$ are disjoint up to step $\tau_k$. Analogously, we define
\begin{equation}
\!\tau_k^* \coloneq \max\{ i\in\{I_{\xref},\ldots, n\}\colon s_{v,i}^*=s_{w,i}^{\dagger}=1 \text{ for some $*,\dagger\in\{-,+\}$ and distinct } v,w \in  [k]\},
\end{equation}
with the convention that we set $\tau_k^*=I_{\xref}-1$ when the maximum is over an empty set. By its definition, it follows that 
\be 
\bigcap_{\square\in\{-,+\}}\bigcap_{v\in[k]}\cS^\square_n(a_v)\cap [\tau_k^*+1,n]=\varnothing.
\ee 
It is for this reason that we couple the set $\cS_n^{\geq 1}(a_v)$ to $\cS_n^-(a_v)$ and $\cS_n^+(a_v)$, rather than the set $\cS_n^{(1)}(a_v)$ (and with $\cS_n^+(a_v)$ defined as in~\eqref{eq:definition_selection_sets_ungreedy_depth_single_vertex}). The definition of $\tau_k^*$ implies that the elements of $\cS^\square_n(a_v)\cap [\tau_k^*+1,n]$ are distinct for different $v\in[k]$. Moreover, conditionally on $\cS^-_n(a_v)\cap [\tau_k^*+1,n]$  and $\cS^+_n(a_v)\cap [\tau_k^*+1,n]$ for $v\in[k]$, this implies that the evolution of the degree, label, and greedy longest path of $a_v$ during steps in the sets $\cS^-_n(a_v)\cap [\tau_k^*+1,n]$  and $\cS^+_n(a_v)\cap [\tau_k^*+1,n]$ (as coupled to $\cS_n^{\geq 1}(a_v)$) is independent among the vertices $a_1,\ldots, a_k$, since these evolutions depend on the independent dice rolls at these steps only, and we know that $\cC_n^{(i)}(a_v)\cap \cC_n^{(i)}(a_w)=\varnothing$ for $v\neq w\in [k]$ and $i>\tau_k^*$.   The same is not (necessarily) true if we would instead use $\tau_k'$, the first step $i$ at which \emph{exactly one} vertex from $\cC_n^{(i)}(a_v)$ and \emph{exactly one} vertex from $\cC_n^{(i)}(a_w)$ (with $v\neq w\in[k]$) is selected. Indeed, it is then still possible that as some earlier step $j>i$, for example  one vertex $u_1$ from $\cC_n^{(j)}(a_1)$ is selected and two vertices (if $m\geq2$) $u_{2,1}$ and $u_{2,2}$ from $\cC_n^{(j)}(a_2)$ are selected. Suppose that $u_{2,1}$ loses and connects to $u_{1}$ and $u_{2,2}$ by directed edges. Now, $\cC_n^{(j-1)}(a_1)\cap \cC_n^{(j-1)}(a_2)\neq \varnothing$, so that the evolution of the greedy longest paths from $a_1$ and $a_2$ are now correlated from step $j-1$ onwards (as both can increase at the same step when $u_1$ is selected and loses the dice roll).

Observe that $\tau_k^*$ is stochastically dominated by $\tau_{mk,0}$ (as in~\eqref{eq:definition_tau_k}), which yields that  $\tau_k^*$ is tight under Assumption~$\xref$\ref{item:lb} by Lemma~\ref{lemma:tightness_tau_k}. That is,  $\prob(\tau_k^*<t_n)=1-o(1)$ if the choice sequence $\x$ satisfies Assumption~$\xref$\ref{item:lb} and $t_n\geq I_{\xref}$ tends to infinity with $n$. As such, for a sequence $(t_n)_{n\in\N}$ that we refer to as the \emph{truncation sequence}, we define the \emph{truncated selection sets}
\begin{equation}
\cS^-_{n,1}(a_v) \coloneq \{i\in\Omega_1\colon s^-_{a_v,i}=1\}, \quad \cS^+_{n,1}(a_v) \coloneq \{i\in\Omega_1\colon s^+_{a_v,i}=1\}
\end{equation}
for $v\in [k]$ with  $\Omega_1=\Omega_1(\mathbf x)\coloneq \{t_n,\ldots, n\}\cap \cA_n$. Clearly, we require $I_{\xref}\leq t_n$ for all $n$ for this definition to make sense. Restricting to the truncated selection sets avoids correlations between the degree, label, and greedy longest path of the active vertices $a_1,\ldots, a_k$,  that occur at later steps of the Kingman coalescent (that is, after step $\tau^*_k$). It is still required to justify why we can ignore all steps $i<t_n$ (for $t_n$ appropriately chosen), which we defer to Lemma~\ref{lemma:bound_u_n_2}. For ease of writing, we set 
\be 
\overline{\cS^-_{n,1}}\coloneq(\cS^-_{n,1}(a_v))_{v\in [k]} \quad \text{and}\quad \overline{\cS^+_{n,1}}\coloneq(\cS^+_{n,1}(a_v))_{v\in [k]}. 
\ee
We also write $\overline{J^-}\coloneq(J_v^-)_{v\in [k]},\overline{J^+}\coloneq(J_v^+)_{v\in [k]}$ with $J^-_v,J^+_v\subseteq \Omega_1$ and $v\in [k]$.  Additionally, we define the \emph{truncated greedy longest path} $u_{n,1}(a_v)$ as the contribution to the greedy longest path of vertex $a_v$ in steps $\ell_n(a_v)\vee t_n,\ldots,n$ and $u_{n,2}(a_v)$ as the contribution to the greedy longest path of vertex $a_v$ in steps $2,\ldots,(\ell_n(a_v)\wedge t_n)-1$. That is, using the notation in~\eqref{eq:unrewrite} and that $|\cC_n^{(i)}(a_v)|\leq m$ for all $i\geq t_n\geq I_{\xref}$,
\be\ba
u_{n,1}(a_v)&\coloneq \ind_{\{\ell_n(a_v)>t_n\}}+\sum_{i=t_n}^{\ell_n(a_v)-1}\ind_{\{x_i=1\}}\sum_{j=1}^m C_i^{(j)},\\
u_{n,2}(a_v)&\coloneq u_n(a_v)-u_{n,1}(a_v) = \ind_{\{\ell_n(a_v)\leq t_n\}}+\sum_{i=2}^{(\ell_n(a_v)\wedge t_n)-1}\ind_{\{x_i=1\}}\sum_{j=1}^m C_i^{(j)},
\ea\ee 
where we use the convention that an empty sum equals $0$. Though we already implicitly considered different choices of truncation sequences in the proofs of Lemma~\ref{lemma:tail_bound_truncated_selection_set} and Theorem~\ref{theorem:joint_degree_distribution}, and some of the results presented here are formulated more generally, it is sufficient for us to consider
\be
\label{eq:def_tngamma}
t_n:=\ceil{(h_n^+)^{\gamma_0}} \quad \text{with} \quad \gamma_0\in(1,\min\{(1/2+\delta)^{-1},(1-2\delta)^{-1}\}),
\ee
where $\delta\in(0,1/2)$ is the constant in Assumption~$\xref$\ref{item:lb}. Notice that, since $\gamma_0<(1-2\delta)^{-1}$, we have that $t_n=\lceil (h_n^+)^{\gamma_0}\rceil=\cO(\max\{I_{\xref}^{\gamma_0},\sqrt n\})$ when Assumption~$\xref$\ref{item:lb} is satisfied. When $I_{\xref}=\cO(\sqrt n)$ is satisfied and since $\gamma_0<2$, we thus obtain that $t_n=o(n)$.

With these definitions at hand, we first show that the contribution of $u_{n,2}(a_v)$ to the greedy longest path of vertex $a_v$ is small. 
\begin{lemma}
\label{lemma:bound_u_n_2}
Fix $k\in\N$. Suppose that the choice sequence $\mathbf x$ satisfies Assumption~$\xref$\ref{item:geq1} and~\ref{item:lb}, where $I_{\xref}=\cO((h_n^+)^{\gamma_0(1/2-\delta)})$ and $\gamma_0$ is as in~\eqref{eq:def_tngamma}. For $v\in [k]$, let $d_v=d_v(n),\ell_v=\ell_v(n)\in\N_0$ such that $\ell_v> t_n$ with $t_n\coloneq\ceil{(h_n^+)^{\gamma_0}}$ and $d_v(n)\to\infty$ with $n$. Then, for any sequence $r_n=\omega((h_n^+)^{\gamma_0(1/2-\delta)})$, as $n\to\infty$,
\begin{equation}
\prob(\{\exists v\in [k]\colon u_{n,2}(a_v)>r_n\}\cap \{\ell_n(a_v)\geq\ell_v,d_n(a_v)\geq d_v\text{ for all } v\in [k]\}) = o\big(\theta^{-\sum_v d_v}\big).
\end{equation}
\end{lemma} 

\begin{remark}\label{rem:rn}
Since $\gamma_0<(1-2\delta)^{-1}$ and $h_n^+\leq n$, we have $(h_n^+)^{\gamma_0(1/2-\delta)}=o(\sqrt{h_n^+})=o(\sqrt n)$, so that, by the argument above Lemma~\ref{lemma:bound_u_n_2},  choosing $I_{\xref}=\cO((h_n^+)^{\gamma_0(1/2-\delta)})$ implies that $t_n=o(n)$. It also implies that we can choose $r_n$ such that $r_n=o(\sqrt{h_n^+})$, so that the  result implies that $u_{n,2}$ is with high probability negligible compared to the Gaussian fluctuations of the greedy longest path (see either Proposition~\ref{prop:depthlabelsinglevertex} or Theorem~\ref{theorem:label_multiple_vertices_conditional_degrees}).\ensymboldremark 
\end{remark}

\begin{proof}
A union bound yields  
\begin{equation}
\begin{aligned}
&\prob(\{\exists v\in [k]\colon u_{n,2}(a_v)>r_n\}\cap \{\ell_n(a_v)\geq\ell_v,d_n(a_v)\geq d_v\text{ for all }v\in [k]\})\\
&\leq \sum_{v\in [k]} \prob(\{u_{n,2}(a_v)>r_n\}\cap\{\ell_n(a_v)\geq\ell_v,d_n(a_v)\geq d_v\text{ for all }v\in [k]\}).
\end{aligned}
\end{equation}
Then, as we have $\ell_n(a_v)\geq \ell_v>t_n$, vertex $a_v$ loses its first dice roll before step $t_n$. Hence, the contribution to $u_{n,2}(a_v)$ comes from steps in $[t_n-1]$, and it is thus necessary that $|\cS_n^{\geq1}(a_v)\cap[t_n-1]|>r_n$. Again using that $a_v$ loses its first dice roll before step $t_n$ and leveraging the coupling of $\cS_n^{\geq1}(a_v)$ with $\cS_n^+(a_v)$ and $\cS_n^-(a_v)$, it follows that it is necessary that $|\cS_n^+(a_v)\cap [t_n-1]|>r_n$.  As a result, we have the upper bound
\be \ba 
\sum_{v\in[k]}{}&\prob(\{|\cS_n^+(a_v)\cap[t_n-1]|>r_n\}\cap \{\ell_n(a_v)\geq\ell_v,d_n(a_v)\geq d_v\text{ for all }v\in [k]\})\\
&= k\prob(|\cS_n^+(a_1)\cap[t_n-1]|>r_n)\P{\ell_n(a_v)\geq\ell_v,d_n(a_v)\geq d_v\text{ for all }v\in [k]}.
\ea \ee 
The final equality is due to independence of $\cS_n^+(a_v)$ and the random variables $(\ell_n(a_v))_{v\in[k]}$, $(d_n(a_v))_{v\in[k]}$, since $(\ell_n(a_v))_{v\in[k]}$ and $(d_n(a_v))_{v\in[k]}$  depend  only on $(\cS_n^-(a_v))_{v\in[k]}$ and the dice rolls, and the fact that $\cS_n^+(a_v)$ has the same distribution for each $v\in[k]$. It then follows from Lemma~\ref{lemma:upper_bound_tail_vertex_degrees} that it suffices to show that the first probability on the final line tends to zero with $n$. By Markov's inequality, we derive
\begin{equation}
\prob(|\cS_n^+(a_1)\cap[t_n-1]|>r_n)\leq \frac{\E{|\cS_n^+(a_1)\cap[t_n-1]|}}{r_n} = \frac{1}{r_n}\bigg(\sum_{i=2}^{I_{\xref}}p^+_{i}+\sum_{i=I_{\xref}+1}^{t_n-1}p^+_{i}\bigg).
\end{equation}
For the first sum, our assumptions on $I_{\xref}$ and $r_n$ directly yield that 
\be 
\sum_{i=2}^{I_{\xref}}p^+_{i}\leq I_{\xref}=o(r_n).
\ee 
For the second sum, we obtain by~\eqref{eq:definition_and_order_of_p_v_i^+} and Assumption~$\xref$\ref{item:lb} that, for some $C>0$,
\begin{equation}
\sum_{i=I_{\xref}+1}^{t_n-1}p^+_{i}\leq 	\sum_{i=I_{\xref}+1}^{t_n-1}\frac{C}{i^{1/2+\delta}}=\cO(t_n^{1/2-\delta})=o(r_n),
\end{equation}
where the final step again follows from the choice of $t_n$ and $r_n$. Hence,  $\prob(|\cS_n^+(a_1)\cap[t_n-1]|>r_n)=o(1)$, which concludes the proof.
\end{proof}
Let $\cP(\cdot)$ denote the power set of a set. We write $\overline J \in \cP(\Omega_1)^k$ to denote that $\overline J=(J_1,\ldots, J_k)$ is a tuple of $k$ elements $J_1,\ldots, J_k$, each of which is a subset of $\Omega_1$.  The following lemma shows that the evolutions of degree, label, and greedy longest path of vertices $a_1,\ldots,a_k$ are independent if the sets in $\overline \cS_{n,1}^\square$ are disjoint, for both $\square\in\{-,+\}$. It is based on a result for random recursive trees without freezing from \cite{Lodewijks.2023}.
\begin{lemma}		\label{lemma:decouple_label_degree_multiple_vertices}
Suppose that Assumption~$\xref$\ref{item:geq1} is satisfied. Fix $k\in\N$ and consider any truncation sequence $(t_n)_{n\in\N}$ such that $t_n\leq n$ for all $n\in\N$. Let $n\in\N$ and, for $v\in [k]$, let $u_v,d_v\in\N_0$, $\ell_v\in\Omega_1$. Furthermore, let $\overline{J^-},\overline{J^+}\in{\cP(\Omega_1)}^k$ such that $J^-_v\cup J^+_v$ and $J^-_w\cup J^+_w$ are disjoint for $v\neq w\in[k]$. Then,
\begin{equation}
\begin{aligned}
&\prob(u_{n,1}(a_v)\leq u_v,\ell_n(a_v)\geq \ell_v,\deg_n(a_v)\geq d_v,v\in [k]\,|\,\overline{\cS^-_{n,1}}=\overline{J^-},\overline{\cS^+_{n,1}}=\overline{J^+})\\
&= \prod_{v=1}^k\prob(u_{n,1}(a_v)\leq u_v,\ell_n(a_v)\geq \ell_v,\deg_n(a_v)\geq d_v\,|\,\cS^-_{n,1}(a_v)=J^-_v,\cS^+_{n,1}(a_v)=J^+_v).
\end{aligned}
\end{equation}
\end{lemma}

\begin{proof}
For each $v\in [k]$, define the event $A_{n,v} \coloneq \{u_{n,1}(a_v)\leq u_v,\ell_n(a_v)\geq \ell_v,\deg_n(a_v)\geq d_v\}$. We rewrite $J^-_v=\{j^-_{v,1},\ldots,j^-_{v,|J_v^-|}\}$ and $J^+_v=\{j^+_{v,1},\ldots,j^+_{v,|J^+_v|}\}$, where $j^-_{v,1}>\cdots>j^-_{v,|J_v^-|}$ and $j^+_{v,1}>\cdots >j^+_{v,|J^+_v|}$ for $v\in [k]$. For each $v\in [k]$ we have that, conditionally on $\{\overline{\cS^-_{n,1}}=\overline{J^-}\}$, the event $\{\deg_n(a_v)\geq d_v\}$ occurs if and only if $|J^-_v|\geq d_v$ and $a_v$ wins the dice rolls at steps $j^-_{v,1},\ldots,j^-_{v,d_v}$, analogous to~\eqref{eq:express_degree_via_selection_set}. Additionally, analogous to~\eqref{eq:express_label_in_terms_selection_sets}, for each $v\in [k]$, conditionally on $\{\overline{\cS^-_{n,1}}=\overline{J^-}\}$ the event $\{\ell_n(a_v)\geq \ell_v\}$ occurs if and only if $a_v$ does not win all dice rolls at steps $\{j^-_{v,i}\colon j^-_{v,i}\geq\ell_v\text{ with }1\leq i \leq |J_v^-|\}$. For each $v\in [k]$, the truncated greedy longest path of $a_v$ is determined by the outcome of the dice rolls at steps $\{j^+_{v,i}\colon j^+_{v,i}\in[t_n, \ell_n(a_v)]\text{ with }1\leq i \leq |J_v^+|\}$. As $\ell_v\geq t_n$ and $J^+_v\subset\Omega_1$, the event $\bigcap_{v\in [k]}A_{n,v}$ is thus fully determined by the steps in $\Omega_1$. In particular, conditionally on $\{\overline{\cS^-_{n,1}}=\overline{J^-},\overline{\cS^+_{n,1}}=\overline{J^+}\}$, the event $\bigcap_{v\in [k]}A_{n,v}$ depends solely on $(J^-_v)_{v\in [k]},(J^+_v)_{v\in [k]}$ and the associated dice rolls. Since the sets $J^-_1\cup J^+_1,\ldots,J^-_k\cup J^+_k$ are pairwise disjoint, the occurrence of the events $(A_{n,v})_{v\in [k]}$ depends on disjoint sets of independent random variables (namely the dice rolls associated with the sets $(J^-_v\cup J^+_v$ for $v\in [k]$). Therefore,
\be \ba
\prob\bigg(\bigcap_{v\in [k]}A_{n,v}\,\bigg|\,\overline{\cS^-_{n,1}}=\overline{J^-},\overline{\cS^+_{n,1}}=\overline{J^+}\bigg) &=\prod_{v=1}^k \prob \Big( A_{n,v}\,\Big|\,\overline{\cS^-_{n,1}}=\overline{J^-},\overline{\cS^+_{n,1}}=\overline{J^+} \Big)\\
&=\prod_{v=1}^k\prob(A_{n,v}\,|\,\cS^-_{n,1}(a_v)=J^-_v,\cS^+_{n,1}(a_v)=J^+_v),
\ea\ee 
where the last step follows from the fact that $A_{n,v}$, conditionally on $\cS_{n,1}^-(a_v)=J_v^-$ and $\cS_{n,1}^+(a_v)=J^+_v$, is independent of the events $\{\cS_{n,1}^-(a_u)=J_u^-\text{ and }\cS_{n,1}^+(a_u)=J^+_u\}$ for all $u\neq v$.
\end{proof}

Next, we state several results about the truncated selection sets that allow us to apply Lemma~\ref{lemma:decouple_label_degree_multiple_vertices} in the analysis of the behaviour of multiple active vertices. For $\overline{d}=(d_{v})_{v\in [k]}\in\N_0^k$, $\alpha\in(0,m+1)$, $D>m(m+1)$, $\beta\in(0,1)$, and $\gamma',\gamma''\in(1,2]$ with $\gamma'<\gamma''$, we define
\begin{equation}
\begin{aligned}
\cD_{\overline{d}}\coloneq\{&(\overline{J^-},\overline{J^+})\in{\cP(\Omega_1)}^{2k}\colon \prob(\overline{\cS^-_{n,1}}=\overline{J^-},\overline{\cS^+_{n,1}}=\overline{J^+},\deg_n(a_v)\geq d_{v} \text{ for all }v\in [k])>0\},\\
\cB_{\alpha,D}\coloneq\{&(\overline{J^-},\overline{J^+})\in{\cP(\Omega_1)}^{2k}\colon J^-_v\cup J^+_v\text{ and $J^-_w\cup J^+_w$ disjoint for $v\neq w\in[k]$, and}\\
&||J^-_v|-(m+1)h_n^+|\leq\alpha h_n^+,|J^+_v|\leq D h_n^+\text{ for all }v\in [k]\},\\
\cG_{\beta,\gamma',\gamma''}\coloneq\{&(\overline{J^-},\overline{J^+})\in{\cP(\Omega_1)}^{2k}\colon|J_v^\square\cap [(h_n^+)^{\gamma'},(h_n^+)^{\gamma''}]|\leq (h_n^+)^\beta\text{ for all }v\in [k],\square\in\{-,+\}\}.
\end{aligned}
\end{equation}
In words, $\cD_{\overline{d}}$ consists of all possible outcomes of the truncated selection sets that enable the event $\{\deg_n(a_v)\geq d_{v},\text{ for all }v\in [k]\}$. Then, $\cB_{\alpha,D}$ contains all truncated selection sets of `typical' sizes which enable the decoupling of the label, degree, and greedy longest path  of the vertices $a_1,\ldots,a_k$, as follows from Lemma~\ref{lemma:decouple_label_degree_multiple_vertices}. Finally, $\cG_{\beta,\gamma',\gamma''}$ contains all truncated selection sets with that do not contain `too many' elements on the scale $[(h_n^+)^{\gamma'},(h_n^+)^{\gamma''}]$. This provides us with a finer control compared to the bounds in the set $\cB_{\alpha,D}$, which we leverage later to bound error probabilities.

We now present some results related to the sets $\cD_{\overline{d}},\,\cB_{\alpha,D},$ and $\cG_{\beta,\gamma',\gamma''}$. In the next lemma, we show that the conditions to be in the set $\cB_{\alpha,D}$ (as well as to be in the set $\cG_{\beta,\gamma',\gamma''}$ for certain $\beta,\gamma',\gamma''$) are met by $\overline{\cS_{n,1}^-}$ and $\overline{\cS_{n,1}^+}$ with high probability.

\begin{lemma}
\label{lemma:truncated_selection_sets_decouple_label_degree_whp}
Fix $k\in\N$, $\alpha\in(0,m+1)$, and $D>m(m+1)$. Let the choice sequence $\mathbf x$ satisfy Assumption~$\xref$\ref{item:geq1} and~$\xref$\ref{item:lb} with $I_{\xref}=\cO((h_n^+)^{\gamma_0(1/2-\delta)})$ and let $(t_n)_{n\in\N}$ be a truncation sequence such that $  I_{\xref}\leq t_n\leq n$ for all $n\in\N$. Let $\beta\in(0,1)$, and $\gamma',\gamma''\in(1,2]$ such that  $\gamma'<\gamma''$ and $\beta>\gamma''(1/2-\delta)$. Then,
\begin{equation}
\prob((\overline{\cS^-_{n,1}},\overline{\cS^+_{n,1}})\in\cB_{\alpha,D}) = 1-o(1),\quad \prob((\overline{\cS^-_{n,1}},\overline{\cS^+_{n,1}})\in \cG_{\beta,\gamma',\gamma''}) = 1-o(1).
\end{equation}
\end{lemma}

\begin{proof}
We first recall that $\mathbf x$ satisfies  Assumption~$\xref$\ref{item:hI}   for any $\eps>0$ and all $n$ large  by the condition that $I_{\xref}=\cO((h_n^+)^{\gamma_0(1/2-\delta)})$ (see Remark~\ref{rem:greedy}$(iii)$). To prove the first statement, we use a union bound to obtain
\begin{equation}
\begin{aligned}
\prob((\overline{\cS^-_{n,1}},\overline{\cS^+_{n,1}})\in\cB_{\alpha,D}) \geq {}&1-\prob(\tau_k^*>t_n)-k\prob(|\cS^-_{n,1}(a_1)-(m+1)h_n^+|>\alpha  h_n^+)\\
&-k\prob(|\cS^+_{n,1}(a_1)|>D h_n^+).
\end{aligned}
\end{equation}
Due to the tightness of $\tau_k^*$ by Lemma~\ref{lemma:tightness_tau_k} and applying a Bernstein type inequality as in the proof of Lemma~\ref{lemma:tail_bound_truncated_selection_set} to the last two probabilities, we obtain a lower bound $1-o(1)$, as desired. For the second statement, we use Chernoff's inequality and Assumption~$\xref$\ref{item:lb} to derive for any constant $C>(\e-1)m(m+1)$ and with $\square\in\{-,+\}$ that 
\begin{equation}
\begin{aligned}
\P{|\cS_{n,1}^\square(a_v)\cap [(h_n^+)^{\gamma'},(h_n^+)^{\gamma''}]|\geq (h_n^+)^\beta}&\leq \exp(-(h_n^+)^\beta)\prod_{\substack{i=(h_n^+)^{\gamma'}\\ x_i=1}}^{(h_n^+)^{\gamma''}}\Big(1+\frac{C}{A_i}\Big)\\
&\leq \exp\bigg(-(h_n^+)^\beta + C\sum_{i=(h_n^+)^{\gamma'}}^{(h_n^+)^{\gamma''}} i^{-(1/2+\delta)}\bigg)\\
&\leq \exp\bigg(-(h_n^+)^\beta +\frac{C}{1/2-\delta} (h_n^+)^{\gamma''(1/2-\delta)}\bigg).
\end{aligned}
\end{equation}
Since $\beta>\gamma''(1/2-\delta)$ and $h_n^+$ tends to infinity with $n$ by~\eqref{eq:hn+lb}, this probability tends to zero as $n\to\infty$. A union bound thus yields the second result.
\end{proof}

The following lemma is a generalisation of a result for random recursive trees without freezing from~\cite{Eslava.2021}, which tells us that for values $d_v$ that are not too large, the set $\cB_{\alpha,D}$ is a subset of $\cD_{\overline d}$.
\begin{lemma}
\label{lemma:disjoint_selection_sets_enable_event}
Fix $\alpha\in(0,m+1), \,D>m(m+1)$, and let the choice sequence $\mathbf x$ satisfy Assumption~$\xref$\ref{item:geq1}. Let $(t_n)_{n\in\N}$ be any truncation sequence such that we have $A_i\geq (m+1)k+1$ for $i\in\{t_n,\ldots,n\}$. If $\overline{d}=(d_v)_{v\in [k]}\in\N_0^k$ satisfies $d_v<(m+1-\alpha) h_n^+$ for all $v\in [k]$, then $\cB_{\alpha,D}\subset\cD_{\overline{d}}$.
\end{lemma} 

\begin{proof}
Let $(\overline{J^-},\overline{J^+})\in\cB_{\alpha,D}$. For $i\in\Omega_1$ and $\square\in\{-,+\}$, let $\sigma_i^\square\coloneq\sum_{v\in [k]}\ind_{\{i\in J_v^\square\}}$, where we observe that $\sigma_i^\square\in\{0,1\}$, since $(\overline{J^-},\overline{J^+})\in\cB_{\alpha,D}$ and thus the sets $\overline J^-_v$ (resp.\ $\overline J^+_v$) are mutually disjoint for $v\in[k]$.  Let $\cE_i^- \coloneq \{s^-_{v,i}=\ind_{\{i\in J_v^-\}}\text{ for all }v\in[k]\}$. Then, we have
\be \label{eq:S-prob}
\prob(\overline{\cS^-_{n,1}}=\overline{J^-}) = 
\prod_{i\in\Omega_1} \prob(\cE_i^-) = \prod_{i\in\Omega_1}  \Bigg[ \frac{\binom{A_i-k}{m+1}}{\binom{A_i}{m+1}}\ind_{\{\sigma_i^-=0\}}+\frac{\binom{A_i-k}{m}}{\binom{A_i}{m+1}}\ind_{\{\sigma_i^-=1\}}\Bigg].
\ee 
Similarly, we have 
\be \label{eq:S+prob} 
\P{\overline{\cS^+_{n,1}}=\overline{J^+}}=\prod_{i\in\Omega_1}\Bigg[ \frac{\binom{A_i-mk}{m+1}}{\binom{A_i}{m+1}}\ind_{\{\sigma_i^+=0\}}+\frac{m\binom{A_i-m(k-1)}{m+1}}{\binom{A_i}{m+1}}\ind_{\{\sigma_i^+=1\}}\Bigg],
\ee 
As $\overline{\cS^-_{n,1}}$ is independent of $\overline{\cS^+_{n,1}}$, we thus obtain that the probability of the event $\{\overline{\cS_{n,1}^-}=\overline{J^-}, \overline{\cS_{n,1}^+}=\overline{J^+}\}$ is strictly positive, as $A_i\geq (m+1)k+1$ for $i\in\{t_n,\ldots, n\}$, so that all terms in the products in~\eqref{eq:S-prob} and~\eqref{eq:S+prob} are positive.  Furthermore, for $v\in[k]$,
\begin{equation}
\prob(\deg_n(a_v)\geq d_v\,|\,\cS^-_{n,1}(a_v)=J^-_v,\cS^+_{n,1}(a_v)=J^+_v)=\prob(\deg_n(a_v)\geq d_v\,|\,\cS^-_{n,1}(a_v)=J^-_v)>0,
\end{equation}
since $|J^-_v|\geq d_v$ and the event $\{\deg_n(a_v)\geq d_v\}$ is independent of $\cS_n^+(a_v)$. Now, by an analogous argument as in (the proof of)  Lemma~\ref{lemma:decouple_label_degree_multiple_vertices}, we derive for $(\overline{J^-},\overline{J^+)}\in \cB_{\alpha,D}$ that
\begin{equation}
\begin{aligned}
&\prob(\deg_n(a_v)\geq d_v\text{ for all }v\in[k]\,|\,\overline{\cS^-_{n,1}}=\overline{J^-},\overline{\cS^+_{n,1}}=\overline{J^+})\\
&= \prod_{v=1}^k\prob(\deg_n(a_v)\geq d_v\,|\,\cS^-_{n,1}(a_v)=J^-_v,\cS^+_{n,1}(a_v)=J^+_v)>0,
\end{aligned}
\end{equation}
and we thus conclude $\cB_{\alpha,D}\subset\cD_{\overline{d}}$.
\end{proof}

We conclude this subsection with a comparison of the sets $\overline{\cS_{n,1}^\square}$ and independent copies of $\cS_{n,1}^\square(a_1)$. Let $\overline{\cR^\square_{n,1}}\coloneq(\cR^\square_{n,1}(1),\ldots,\cR^\square_{n,1}(k))$ be $k$ independent copies of $\cS^\square_{n,1}(a_1)$ for $\square\in\{-,+\}$. Fix $\delta\in(0,1/2)$, and let $\gamma_0$ as in~\eqref{eq:def_tngamma}. For $\ell\in\N$, we define  
\be \label{eq:L}
\gamma_\ell\coloneq \gamma_0(1-2\delta)^{-\ell}\wedge 2\qquad \text{and}\qquad \Lambda=\Lambda(\delta,\gamma_0)\coloneq \inf\big\{\ell\in\N\colon \gamma_\ell\big(\tfrac12+\delta\big)>1\big\}. 
\ee
We then let $(\beta_\ell)_{\ell\in[\Lambda]}$ be a sequence such that 
\be \label{eq:betagammaineq}
\gamma_\ell\big(\tfrac12-\delta\big)<\beta_\ell<\gamma_{\ell-1}\big(\tfrac12+\delta\big)\qquad\text{for all }\ell\in[\Lambda].
\ee
By the definition of $\Lambda$ and the fact that $\gamma_0<(\tfrac12+\delta)^{-1}$, it follows that $\beta_\ell<1$ for all $\ell\in[\Lambda]$. Further, we observe by the definition of $\gamma_\ell$ and the fact that $\delta\in(0,1/2)$ that  the interval $(\gamma_\ell(\frac12-\delta),\gamma_{\ell-1}(\frac12+\delta))$ is non-empty and thus such a  $\beta_\ell$ exists for all $\ell\in[\Lambda]$. The exact choice of $\beta_\ell$ is not relevant, only that the inequalities in~\eqref{eq:betagammaineq} are satisfied. Furthermore, since $\gamma_0<\gamma_\Lambda\leq 2$, it follows from the same argument above Lemma~\ref{lemma:bound_u_n_2} that $(h_n^+)^{\gamma_\Lambda}=o(n)$ when Assumption~$\xref$\ref{item:lb} is satisfied with $I_{\xref}=o(\sqrt n)$. In particular, $I_{\xref}=\cO((h_n^+)^{\gamma_0(1/2-\delta)})$ suffices (see Remark~\ref{rem:rn}).

We then have the following result, which we prove in Appendix~\ref{sec:appendix}.
\begin{lemma}
\label{lemma:replace_selection_sets_by_independent_copies}
Fix $k\in\N$, $\alpha\in(0,m+1)$, $D>m(m+1)$, and set $t_n\coloneq \lceil(h_n^+)^{\gamma_0}\rceil$. Let $\mathbf x$ be a choice sequence that satisfies Assumption~$\xref$\ref{item:geq1} and~\ref{item:lb} for $I_{\xref} \leq t_n$. Uniformly over $(\overline{J^-},\overline{J^+})\in\cB_{\alpha,D}\cap\bigcap_{\ell=1}^\Lambda\cG_{\beta_\ell,\gamma_{\ell-1},\gamma_\ell}$,
\begin{equation}
\prob(\overline{\cS^-_{n,1}}=\overline{J^-})=\prob(\overline{\cR^-_{n,1}}=\overline{J^-})(1+o(1)) \qquad\text{and}\qquad \prob(\overline{\cS^+_{n,1}}=\overline{J^+})=\prob(\overline{\cR^+_{n,1}}=\overline{J^+})(1+o(1)).
\end{equation}
\end{lemma}

\subsection{Label and greedy longest path of active vertices with large degrees}
\label{subsec:labels_multiple_vertices}

Equipped with the preliminary results of Section~\ref{subsec:selection_sets_multiple_vertices}, we are ready to prove Theorem~\ref{theorem:label_multiple_vertices_conditional_degrees}.
\begin{proof}[Proof of Theorem~\ref{theorem:label_multiple_vertices_conditional_degrees}]
The case $k=1$ follows directly from Proposition~\ref{prop:depthlabelsinglevertex}, so we consider only the case $k\geq 2$.  Recall the definitions of $u=u(d,n,z)$ from~\eqref{eq:u} and $L=L(d,y)$ from~\eqref{eq:Ldef}, and that $b_v\coloneq \lim_{n\to\infty} d_v(n)/h_n^+$. For $v\in[k]$ we set $u_v\coloneq u(d_v,n,z_v)$ for $z_v\in\R$ and take $y_v\in \R$.  By the equivalence in~\eqref{eq:equiv}, it suffices to prove that
\begin{equation}
\label{eq:multiple_vertices_proof_rephrasement}
\begin{aligned}        
\prob{}&(u_n(a_v)\leq u_v,\ell_n(a_v)\geq L(d_v,y_v),\deg_n(a_v)\geq d_v\text{ for all }v\in [k])\\
&= (1+o(1))\theta^{-\sum_{v\in [k]} d_v}\!\!\!\prod_{v\in [k]}\prob\Bigg(M\sqrt{\frac{mb_v}{(m+1)^2-b_v}}+N\sqrt{1-\frac{mb_v}{(m+1)^2-b_v}}\leq z_{a_v},M>y_v\Bigg),
\end{aligned}
\end{equation}
since then, by Theorem~\ref{theorem:joint_degree_distribution},
\begin{equation}
\begin{aligned}
\lim_{n\to\infty}{}&\prob(u_n(a_v)\leq  u_v,\ell_n(a_v)\geq L(d_v,y_v)\text{ for all }v\in[k]\,|\,\deg_n(a_v)\geq d_v\text{ for all }v\in[k])\\
&= \prod_{v\in [k]}\prob\left(M\sqrt{\frac{mb_v}{(m+1)^2-b_v}}+N\sqrt{1-\frac{mb_v}{(m+1)^2-b_v}}\leq z_{a_v},M>y_v\right),
\end{aligned}
\end{equation}
which is equivalent to the theorem statement by the choice of $u_v$, by using~\eqref{eq:equiv}, and by Corollary~\ref{corollary:exchangeability_degrees}. We split the proof of~\eqref{eq:multiple_vertices_proof_rephrasement} into an upper and a lower bound and prove the upper bound first. The lower bound follows analogously up to the first steps.\\
\textbf{Upper bound. } By the definition of $u_{n,1}(a_v)$, we have
\be\ba 
\prob{}&(u_n(a_v)\leq u_v,\ell_n(a_v)\geq L(d_v,y_v),\deg_n(a_v)\geq d_v\text{ for all }v\in [k])\\
&\leq \prob(u_{n,1}(a_v)\leq u_v,\ell_n(a_v)\geq L(d_v,y_v),\deg_n(a_v)\geq d_v\text{ for all }v\in [k]).
\ea\ee
Define, for $\overline{J^-},\overline{J^+}\in\cP(\Omega_1)^k$, 
\be\ba f_n{}&(\overline{J^-},\overline{J^+})\\
&\coloneq\prob\big(u_{n,1}(a_v)\leq u_v,\ell_n(a_v)\geq L(d_v,y_v),\deg_n(a_v)\geq d_v\text{ for all }v\in[k]\,\big|\,\overline{\cS^-_{n,1}}=\overline{J^-},\overline{\cS^+_{n,1}}=\overline{J^+}\big),\\
g_n{}&(\overline{J^-},\overline{J^+})\\
&\coloneq\prod_{v\in [k]}\prob(u_{n,1}(a_v)\leq u_v,\ell_n(a_v)\geq L(d_v,y_v),\deg_n(a_v)\geq d_v\,|\,\cS^-_{n,1}(a_v)=J^-_v,\cS^+_{n,1}(a_v)=J^+_v),
\ea\ee
and define $\cH\coloneq \cB_{\alpha,D}\cap\bigcap_{\ell=1}^\Lambda\cG_{\beta_\ell,\gamma_{\ell-1},\gamma_\ell}$, where $\gamma_\ell$ and $\Lambda$ are  as in \eqref{eq:L}. Take $D>m(m+1)$, $c\in(\max_{v\in [k]}b_v,m+1)$, and set $\alpha\coloneq m+1-c$ so that $\cH\subset\cB_{\alpha,D}\subset\cD_{\overline{d}}$ by Lemma~\ref{lemma:disjoint_selection_sets_enable_event}. Using the tower property, we deduce
\begin{equation}
\label{eq:multiple_vertices_decompose_probability_two_expectations}
\begin{aligned}
\prob{}&(u_{n,1}(a_v)\leq u_v,\ell_n(a_v)\geq L(d_v,y_v),\deg_n(a_v)\geq d_v\text{ for all }v\in[k]) \\
&= \E{f_n(\overline{\cS^-_{n,1}},\overline{\cS^+_{n,1}})}\\
&= \E{f_n(\overline{\cS^-_{n,1}},\overline{\cS^+_{n,1}})\ind_{\left\{(\overline{\cS^-_{n,1}},\overline{\cS^+_{n,1}})\in\cH\right\}}}+\E{f_n(\overline{\cS^-_{n,1}},\overline{\cS^+_{n,1}})\ind_{\left\{(\overline{\cS^-_{n,1}},\overline{\cS^+_{n,1}})\in\cD_{\overline{d}}\setminus\cH\right\}}},
\end{aligned}
\end{equation}
where the last step follows from the fact that, for $(\overline{J^-},\overline{J^+})\notin\cD_{\overline{d}}$, we have $f_n(\overline{J^-},\overline{J^+})=0$ or $\prob(\overline{\cS^-_{n,1}}=\overline{J^-},\overline{\cS^+_{n,1}}=\overline{J^+})=0$.  Now, consider the first term on the right-hand side. The truncated selection sets in $\cH$ are disjoint by definition, and therefore we have $f_n(\overline{J^-},\overline{J^+})=g_n(\overline{J^-},\overline{J^+})$ for all $(\overline{J^-},\overline{J^+})\in\cH$ by Lemma~\ref{lemma:decouple_label_degree_multiple_vertices}. Together with Lemma~\ref{lemma:replace_selection_sets_by_independent_copies} and by the independence of $\overline{S_{n,1}^-}$ and $\overline{S_{n,1}^+}$, this yields
\begin{equation}	\label{eq:multiple_vertices_deal_with_expectation_B_nd}
\begin{aligned}
\E{f_n(\overline{\cS^-_{n,1}},\overline{\cS^+_{n,1}})\ind_{\left\{(\overline{\cS^-_{n,1}},\overline{\cS^+_{n,1}})\in\cH\right\}}} &= \E{g_n(\overline{\cS^-_{n,1}},\overline{\cS^+_{n,1}})\ind_{\left\{(\overline{\cS^-_{n,1}},\overline{\cS^+_{n,1}})\in\cH\right\}}}\\
&=\E{g_n(\overline{\cR^-_{n,1}},\overline{\cR^+_{n,1}})\ind_{\left\{(\overline{\cR^-_{n,1}},\overline{\cR^+_{n,1}})\in\cH\right\}}}(1+o(1)).
\end{aligned}
\end{equation}
Moreover, since $f_n(\overline{J^-},\overline{J^+})\leq \theta^{-\sum_{v\in [k]} d_v}$ and $g_n(\overline{J^-},\overline{J^+})\leq\theta^{-\sum_{v\in [k]} d_v}$ by Lemma~\ref{lemma:upper_bound_tail_vertex_degrees}, we can apply Lemmas~\ref{lemma:truncated_selection_sets_decouple_label_degree_whp} and~\ref{lemma:replace_selection_sets_by_independent_copies} to derive
\begin{equation}
\label{eq:multiple_vertices_deal_with_expectation_A_d_B_nd}
\begin{aligned}
&\left|\E{f_n(\overline{\cS^-_{n,1}},\overline{\cS^+_{n,1}})\ind_{\left\{(\overline{\cS^-_{n,1}},\overline{\cS^+_{n,1}})\in\cD_{\overline{d}}\setminus\cH\right\}}} - \E{g_n(\overline{\cR^-_{n,1}},\overline{\cR^+_{n,1}})\ind_{\left\{(\overline{\cR^-_{n,1}},\overline{\cR^+_{n,1}})\in\cD_{\overline{d}}\setminus\cH\right\}}}\right|\\
&\leq\theta^{-\sum_{v\in [k]} d_v}\left(\prob((\overline{\cS^-_{n,1}},\overline{\cS^+_{n,1}})\in\cD_{\overline{d}}\setminus\cH)+\prob((\overline{\cR^-_{n,1}},\overline{\cR^+_{n,1}})\in\cD_{\overline{d}}\setminus\cH)\right)\\
&\leq\theta^{-\sum_{v\in [k]} d_v}\left(2-2\prob((\overline{\cS^-_{n,1}},\overline{\cS^+_{n,1}})\in\cH)(1+o(1))\right)\\
&=o\left(\theta^{-\sum_{v\in [k]} d_v}\right),
\end{aligned}
\end{equation}
where we use the first inequality in \eqref{eq:betagammaineq} for the application of Lemma~\ref{lemma:truncated_selection_sets_decouple_label_degree_whp}. Thus, combining~\eqref{eq:multiple_vertices_decompose_probability_two_expectations},~\eqref{eq:multiple_vertices_deal_with_expectation_B_nd} and~\eqref{eq:multiple_vertices_deal_with_expectation_A_d_B_nd}, we arrive at
\be\ba 
\label{eq:multiple_vertices_probability_as_expectation_of_independent_rvs}
\prob{}&(u_n(a_v)\leq u_v,\ell_n(a_v)\geq L(d_v,y_v),\deg_n(a_v)\geq d_v\text{ for all }v\in [k])\\
&\leq \E{g_n(\overline{\cR^-_{n,1}},\overline{\cR^+_{n,1}})}(1+o(1))+o\big(\theta^{-\sum_{v\in [k]} d_v}\big).
\ea\ee 
Since the elements of $\overline{\cR^\square_{n,1}}$ are i.i.d.\ for both $\square\in \{-,+\}$ and, $\overline{\cR^-_{n,1}}$ is independent of $\overline{\cR^+_{n,1}}$,
\begin{equation}\label{eq:gn}
\E{g_n(\overline{\cR^-_{n,1}},\overline{\cR^+_{n,1}})} = \prod_{v\in [k]}\prob(u_{n,1}(a_v)\leq u_v,\ell_n(a_v)\geq L(d_v,y_v),\deg_n(a_v)\geq d_v).
\end{equation}
To conclude the upper bound, we want to replace the random variable $u_{n,1}(a_v)$ in the above probabilities with $u_n(a_v)$. Define the events $\cE_{n,v}\coloneq\{u_{n,2}(a_v)\leq r_n\}$ for $v\in [k]$, where $(r_n)_{n\in\N}$ is an integer-valued sequence such that $r_n=\omega((h_n^+)^{\gamma_0(1/2-\delta)})$ and $r_n=o(\sqrt{h_n^+})$ (this is possible as $\gamma_0<1/(1-2\delta)$). Since we set the truncation sequence as $t_n=\lceil (h_n^+)^{\gamma_0}\rceil$, we have $h_{t_n}^+=\cO(\max\{I_{\xref},t_n^{1/2-\delta}\})$ by Assumption~$\xref$\ref{item:lb}. As we assume that $I_{\xref}= \cO((h_n^+)^{\gamma_0(1/2-\delta)})$, we obtain  $h_{t_n}^+=\cO((h_n^+)^{\gamma_0(1/2-\delta)})$. We  further set $\gamma_0<(1-2\delta)^{-1}$, so that $h_{t_n}^+\leq \eps h_n^+$ for any $\eps>0$ and all $n$ large. With similar computations as in~\eqref{eq:I<L(d,y)}, we thus see that $t_n<L(d_v,y_v)$ for all $n$ large, for any $v\in[k]$ and any choice of $y_v\in\R$ and $d_v=d_v(n)$ such that $\lim_{n\to\infty}d_v/h_n^+=b_v\in[0,m+1)$. As a result, we can use Lemma~\ref{lemma:bound_u_n_2} to bound each term in the product from above by
\be\ba 				
\prob{}&(\cE_{n,v}\cap\{u_{n,1}(a_v)\leq u_v,\ell_n(a_v)\geq L(d_v,y_v),\deg_n(a_v)\geq d_v\})+o\left(\theta^{-d_v}\right)\\
&\leq \prob(\cE_{n,v}\cap\{u_n(a_v)\leq u_v+r_n,\ell_n(a_v)\geq L(d_v,y_v),\deg_n(a_v)\geq d_v\})+o\left(\theta^{-d_v}\right)\\
&\leq  \prob(u_n(a_v)\leq u_v+r_n,\ell_n(a_v)\geq L(d_v,y_v)\,|\,\deg_n(a_v)\geq d_v)\,\prob(\deg_n(a_v)\geq d_v) +o\left(\theta^{-d_v}\right).
\ea\ee 
Combined with~\eqref{eq:multiple_vertices_probability_as_expectation_of_independent_rvs} and~\eqref{eq:gn}, Theorem~\ref{theorem:joint_degree_distribution}, Proposition~\ref{prop:depthlabelsinglevertex}, and the fact that $r_n=o(\sqrt{h_n^+})$,
\begin{equation}
\begin{aligned}
\prob{}&(u_n(a_v)\leq u_v,\ell_n(a_v)\geq L(d_v,y_v),\deg_n(a_v)\geq d_v\text{ for all }v\in[k]) \\ 
&\leq (1+o(1))\theta^{-\sum_{v\in [k]} d_v}\!\!\prod_{v\in [k]}\prob\Bigg(M\sqrt{\frac{mb_v}{(m+1)^2-b_v}}+N\sqrt{1-\frac{mb_v}{(m+1)^2-b_v}}\leq z_{a_v},M>y_v\Bigg).
\end{aligned}
\end{equation}
\textbf{Lower bound. }Let $\cE_n\coloneq\bigcap_{v\in [k]}\cE_{n,v}$. Using Lemma~\ref{lemma:bound_u_n_2}, we deduce
\begin{equation}
\begin{aligned}
&\prob(u_n(a_v)\leq u_v,\ell_n(a_v)\geq L(d_v,y_v),\deg_n(a_v)\geq d_v\text{ for all }v\in [k])\\
&\geq \prob(\cE_n\cap\{u_n(a_v)\leq u_v,\ell_n(a_v)\geq L(d_v,y_v),\deg_n(a_v)\geq d_v\text{ for all }v\in [k]\})\\
&\geq  \prob(u_{n,1}(a_v)\leq u_v-r_n,\ell_n(a_v)\geq L(d_v,y_v),\deg_n(a_v)\geq d_v\text{ for all }v\in[k])+o\big(\theta^{-\sum_{v\in [k]} d_v}\big).
\end{aligned}
\end{equation}
Now, with similar steps as in the upper bound and by the choice of $r_n$, we obtain
\begin{equation}
\begin{aligned}
&\prob(u_{n,1}(a_v)\leq u_v-r_n,\ell_n(a_1)\geq L(d_v,y_v),\deg_n(a_v)\geq d_v\text{ for all }v\in[k])\\
&\geq  (1+o(1))\theta^{-\sum_{v\in [k]} d_v}\!\!\prod_{v\in [k]}\prob\Bigg(M\sqrt{\frac{mb_v}{(m+1)^2-b_v}}+N\sqrt{1-\frac{mb_v}{(m+1)^2-b_v}}\leq z_{a_v},M>y_v\Bigg).
\end{aligned}
\end{equation}
Combined with the upper bound, this yields~\eqref{eq:multiple_vertices_proof_rephrasement} and concludes the proof.	
\end{proof}

\appendix 

\section{}\label{sec:appendix} 

\begin{proof}[Proof of Theorem~\ref{thrm:multinorm}]
The proof is an adaptation of the case for a single sequence of random variables, as in~\cite[Theorem $1.24$]{Bollobas.2001}. Fix $r,s\in\N$ and write 
\be \ba \label{eq:powertofac}
\mathbb E\bigg[{}&\Big(\frac{X_n-\lambda_n}{\sqrt{\lambda_n}}\Big)^r\Big(\frac{Y_n-\mu_n}{\sqrt{\mu_n}}\Big)^s\bigg]\\
&=\E{\bigg(\sum_{k_1=0}^r \sum_{\ell_1=0}^r c_{r,k_1,\ell_1} (X_n)_{k_1}\lambda_n^{\ell_1-r/2}\bigg)\bigg(\sum_{k_2=0}^s \sum_{\ell_2=0}^s c_{s,k_2,\ell_2} (Y_n)_{k_2}\mu_n^{\ell_2-s/2}\bigg)}\\
&=\sum_{k_1=0}^r \sum_{\ell_1=0}^r\sum_{k_2=0}^s \sum_{\ell_2=0}^s  c_{r,k_1,\ell_1} c_{s,k_2,\ell_2} \E{(X_n)_{k_1}(Y_n)_{k_2}}\lambda_n^{\ell_1-r/2}\mu_n^{\ell_2-s/2},
\ea \ee 
where the $c_{r,k,\ell}$ are constants that do not depend on $n$.  Now, let $P_{1,n}$ and $P_{2,n}$ be independent Poisson random variables with mean $\lambda_n$ and $\mu_n$, respectively, for all $n\in\N$. We set 
\be 
Z_{1,n}\coloneq \frac{P_{1,n}-\lambda_n}{\sqrt{\lambda_n}}\qquad \text{and}\qquad Z_{2,n}\coloneq \frac{P_{2,n}-\mu_n}{\sqrt{\mu_n}}, 
\ee 
and conclude that $(Z_{1,n},Z_{2,n})\toindis (N_1,N_2)$, where $N_1$ and $N_2$ are two independent standard normal random variables. As a result, for any $r,s\in\N$ by using the same steps as in~\eqref{eq:powertofac},
\be 
\E{Z_{1,n}^rZ_{2,n}^s}=\sum_{k_1=0}^r \sum_{\ell_1=0}^r\sum_{k_2=0}^s \sum_{\ell_2=0}^s c_{r,k_1,\ell_1}c_{s,k_2,\ell_2}\lambda_n^{k_1+\ell_1-r/2}\mu_n^{k_2+\ell_2-s/2}, 
\ee 
and also
\be 
\lim_{n\to\infty} \E{Z_{1,n}^rZ_{2,n}^s}=\E{N_1^r}\E{N_2^s}\eqcolon m_rm_s, 
\ee 
where the latter follows from the convergence of the moment-generating function of $Z_{i,n}$ to that of a standard normal random variable and the independence of $Z_{1,n}$ and $Z_{2,n}$. By the assumption on the joint factorial means of $X_n$ and $Y_n$ in~\eqref{eq:factmeanass}, we arrive at 
\be 
\lim_{n\to\infty}\mathbb E\bigg[\Big(\frac{X_n-\lambda_n}{\sqrt{\lambda_n}}\Big)^r\Big(\frac{Y_n-\mu_n}{\sqrt{\mu_n}}\Big)^s\bigg]=m_rm_s,
\ee 
for any fixed $r$ and $s$. As the normal distribution is characterised by its moments, an application of the method of moments yields the desired joint convergence. 
\end{proof}

\begin{proof}[Proof of Lemma~\ref{lemma:replace_selection_sets_by_independent_copies}]

Recall the expression in~\eqref{eq:S-prob}. Without loss of generality we assume that $n$ is large enough so that $A_i\geq (m+1)k+1$ for all $i\in\{I_{\xref},\ldots,n\}$, so that this expression can be used here as well.  Since $ \overline{\cR^-_{n,1}}$ consists of $k$ independent copies of $\cS^-_{n,1}(a_1)$,  we obtain
\be\ba \label{eq:ratiominus}
\frac{\prob(\overline{\cS^-_{n,1}}=\overline{J^-})}{\prob(\overline{\cR^-_{n,1}}=\overline{J^-})}=\prod_{i\in \Omega_1}\Bigg[ \frac{\binom{A_i-k}{m+1}\binom{A_i}{m+1}^{k-1}}{\binom{A_i-1}{m+1}^k}\ind_{\{\sigma_i^-=0\}}+\frac{\binom{A_i-k}{m}\binom{A_i}{m+1}^{k-1}}{\binom{A_i-1}{m}\binom{A_i-1}{m+1}^{k-1}}\ind_{\{\sigma_i^-=1\}}\Bigg].
\ea \ee 
By applying~\eqref{eq:pascal} with $t=A_i-1$, $s=k-1$, and $r=m+1$, as well as with $t=A_i$, $s=1$, and $r=m+1$, we obtain 
\be 
\frac{\binom{A_i-k}{m+1}\binom{A_i}{m+1}^{k-1}}{\binom{A_i-1}{m+1}^k}=\Bigg(1+\sum_{j=2}^k \frac{\binom{A_i-j}{m}}{\binom{A_i-k}{m+1}}\Bigg)^{-1} \Big(1+\frac{m+1}{A_i-(m+1)}\Big)^{k-1}.
\ee 
Bounding each $j$ from above by $k$ and using that $(1+x)^{k-1}=1+(k-1)x+\cO(x^2)$ as $x\downarrow 0$ yields
\be 
\frac{\binom{A_i-k}{m+1}\binom{A_i}{m+1}^{k-1}}{\binom{A_i-1}{m+1}^k}\leq \Big(1+\frac{(m+1)(k-1)}{A_i-(k+m)}\Big)^{-1} \Big(1+\frac{m+1}{A_i-(m+1)}\Big)^{k-1}=1+\cO(A_i^{-2}).
\ee 
Similarly, bounding each $j$ from below by $2$ yields 
\be \ba 
\frac{\binom{A_i-k}{m+1}\binom{A_i}{m+1}^{k-1}}{\binom{A_i-1}{m+1}^k}&\geq \Big(1+\frac{(m+1)(k-1)}{A_i-k}\prod_{\ell=1}^m \frac{A_i-1-\ell}{A_i-k-\ell}\Big)^{-1}\Big(1+\frac{(m+1)(k-1)}{A_i}+\cO(A_i^{-2})\Big)\\
&= 1-\cO(A_i^{-2}). 
\ea \ee 
We obtain with similar computations that 
\be 
\frac{\binom{A_i-k}{m}\binom{A_i}{m+1}^{k-1}}{\binom{A_i-1}{m}\binom{A_i-1}{m+1}^{k-1}}=\Bigg(1+\sum_{j=2}^k \frac{\binom{A_i-j}{m-1}}{\binom{A_i-k}{m}}\Bigg)^{-1}\Big(1+\frac{m+1}{A_i-(m+1)}\Big)^{k-1},
\ee
so that we arrive at the upper bound 
\be 
\frac{\binom{A_i-k}{m}\binom{A_i}{m+1}^{k-1}}{\binom{A_i-1}{m}\binom{A_i-1}{m+1}^{k-1}}\leq \Big(1+\frac{m(k-1)}{A_i-(k+m-1)}\Big)^{-1} \Big(1+\frac{m+1}{A_i-(m+1)}\Big)^{k-1}=1+\cO(A_i^{-1}),
\ee 
and the lower bound 
\be 
\frac{\binom{A_i-k}{m}\binom{A_i}{m+1}^{k-1}}{\binom{A_i-1}{m}\binom{A_i-1}{m+1}^{k-1}}\geq \Big(1+\frac{m(k-1)}{A_i-k}\prod_{\ell=1}^{m-1} \frac{A_i-1-\ell}{A_i-k-\ell}\Big)^{-1}\Big(1+\frac{(m+1)(k-1)}{A_i}+\cO(A_i^{-2})\Big)\geq 1, 
\ee 
where the final lower bound is satisfied for $i$ (and thus $A_i$) sufficiently large by Assumption~$\xref$\ref{item:lb}. Using these upper and lower bounds in
~\eqref{eq:ratiominus}, we conclude 
\be 
\frac{\prob(\overline{\cS^-_{n,1}}=\overline{J^-})}{\prob(\overline{\cR^-_{n,1}}=\overline{J^-})}=\prod_{\substack{i\in \Omega_1\\ \sigma_i^-=0}}\big(1+\cO(A_i^{-2})\big)\prod_{\substack{i\in \Omega_1\\ \sigma_i^-=1}}\big(1+\cO(A_i^{-1})\big).
\ee 
By analogous computations and using~\eqref{eq:S+prob}, we also obtain that
\be 
\frac{\prob(\overline{\cS^+_{n,1}}=\overline{J^+})}{\prob(\overline{\cR^+_{n,1}}=\overline{J^+})}=\prod_{\substack{i\in \Omega_1\\ \sigma_i^-=0}}\big(1+\cO(A_i^{-2})\big)\prod_{\substack{i\in \Omega_1\\ \sigma_i^-=1}}\big(1+\cO(A_i^{-1})\big). 
\ee 
We stress that in all the bounds, the constants that appear in the $\cO$ terms do not depend on $i$, but on $m$ and $k$ only. As $1+x\leq \e^x$ for any $x\in\R$, Assumption~$\xref$\ref{item:lb} yields that, for $\square \in \{-,+\}$,
\be \ba 
\frac{\prob(\overline{\cS^\square_{n,1}}=\overline{J^\square})}{\prob(\overline{\cR^\square_{n,1}}=\overline{J^\square })}&= \exp\bigg(\cO\bigg(\sum_{ i\in \Omega_1}\ind_{\{ \sigma_i^\square =0\}} A_i^{-2}\bigg) +\cO\bigg(\sum_{i\in \Omega_1}\ind_{\{ \sigma_i^\square =1\}} A_i^{-1}\bigg)\bigg)\\
&=\exp\bigg(\cO\bigg(\sum_{ i\in \Omega_1}\ind_{\{ \sigma_i^\square =0\}} i^{-(1+2\delta)}\bigg) +\cO\bigg(\sum_{i\in \Omega_1}\ind_{\{ \sigma_i^\square =1\}} i^{-(1/2+\delta)}\bigg)\bigg).
\ea \ee 
The first sum tends to zero with $n$, since $\Omega_1=\{t_n,\ldots, n\}\cap \cA_n$ and $t_n$ tends to infinity with $n$. If we also show that the second sum tends to zero, the proof is complete. To this end, we use that $(\overline J^-,\overline J^+)\in \cap_{\ell=1}^\Lambda \cG_{\beta_\ell,\gamma_{\ell-1},\gamma_\ell}$ to bound 
\be \ba 
\sum_{i\in\Omega_1}\ind_{\{\sigma_i^\square=1\}}i^{-(1/2+\delta)}&\leq \sum_{\ell=1}^\Lambda \sum_{i\in \Omega_1\cap [(h_n^+)^{\gamma_{\ell-1}},(h_n^+)^{\gamma_\ell}]}\!\!\!\!\!\!\!\!\!\!\ind_{\{\sigma_i^\square=1\}}i^{-(1/2+\delta)} +\sum_{i\in \Omega_1\cap [ (h_n^+)^{\gamma_\Lambda},n]}\!\!\!\!\!\!\!\!\!\!\ind_{\{\sigma_i^\square=1\}}i^{-(1/2+\delta)}\\        &\leq \sum_{\ell=1}^\Lambda (h_n^+)^{-\gamma_{\ell-1}(1/2+\delta)}\!\!\!\!\!\!\!\!\!\!\!\!\!\!\! \sum_{i\in \Omega_1\cap [(h_n^+)^{\gamma_{\ell-1}},(h_n^+)^{\gamma_\ell}]}\!\!\!\!\!\!\!\!\!\!\ind_{\{\sigma_i^\square=1\}} +(h_n^+)^{-\gamma_\Lambda(1/2+\delta)}\!\!\!\!\!\!\!\!\!\!\!\!\!\sum_{i\in \Omega_1\cap [ (h_n^+)^{\gamma_\Lambda},n]}\!\!\!\!\!\!\!\!\!\!\!\!\! \ind_{\{\sigma_i^\square=1\}}.
\ea \ee 
We note that we can indeed partition $\Omega_1=\{t_n,\ldots, n\}$ this way, since $I_{\xref}\leq t_n$ and $t_n=(h_n^+)^{\gamma_0}\leq (h_n^+)^{\gamma_\Lambda}=o(n)$, as stated above Lemma~\ref{lemma:replace_selection_sets_by_independent_copies}.  By the definition of $\cG_{\beta_\ell,\gamma_{\ell-1},\gamma_\ell}$,  it follows   that 
\be
\sum_{i\in\Omega_1\cap[(h_n^+)^{\gamma_{\ell-1}},(h_n^+)^{\gamma_\ell}]}\ind_{\{\sigma_i^\square=1\}}=\sum_{v\in[k]}|J_v^\square\cap [(h_n^+)^{\gamma_{\ell-1}},(h_n^+)^{\gamma_\ell}]|\leq k(h_n^+)^{\beta_\ell}.
\ee
Also, from $(\overline{J^-},\overline{J^+})\in\cB_{\alpha,D}$, we derive $\sum_{i\in\Omega_1} \ind_{\{\sigma_i^-=1\}}=\sum_{v\in [k]}|J^-_v| \leq k(m+1+\alpha) h_n^+$ and $\sum_{i\in\Omega_1} \ind_{\{\sigma_i^+=1\}}\leq kDh_n^+$. We thus arrive at the final upper bound 
\be 
\sum_{i\in\Omega_1}\ind_{\{\sigma_i^\square=1\}}i^{-(1/2+\delta)}\leq \sum_{\ell=1}^\Lambda k(h_n^+)^{\beta_\ell-\gamma_{\ell-1}(1/2+\delta)}+k\max\{(m+1+\alpha),D\}(h_n^+)^{1-\gamma_\Lambda(1/2+\delta)}=o(1), 
\ee
where the final step follows from the second inequality in~\eqref{eq:betagammaineq}, satisfied by $\beta_\ell$ and $\gamma_{\ell-1}$, and by the definition of $\Lambda$ in~\eqref{eq:L}.
\end{proof}

\textbf{Acknowledgements}

BL has received funding from the European Union’s Horizon 2022 research and innovation programme under the Marie Sk\l{}odowska-Curie grant agreement no.~$101108569$, ``DynaNet".

\bibliographystyle{abbrv}
\bibliography{literature}

\end{document}